\documentclass[12pt]{amsart}
\usepackage{pb-diagram}

\usepackage{amscd}
\usepackage{amsfonts}
\usepackage{amsmath}
\usepackage{amssymb}
\usepackage{color}

\newcommand{\tr}{\textnormal{tr}}
\newcommand{\ric}{\textnormal{Ric}}
\newcommand{\rr}{\textnormal{R}}

\newcommand{\ep}{\varepsilon}

\newcommand{\var}{\textnormal{Var}}
\newcommand{\fs}{\textnormal{FS}}
\newcommand{\cH}{\mathcal{H}}

\newcommand{\dbar}{\overline{\partial}}

\newcommand{\ddt}[1]{\frac{\partial #1}{\partial t}}
\newcommand{\dds}[1]{\frac{\partial #1}{\partial s}}

\newcommand{\ddbar}{\sqrt{-1}\partial\dbar}

\newcommand{\cY}{\mathcal{Y}}
\newcommand{\cK}{\mathcal{K}}
\newcommand{\cL}{\mathcal{L}}
\newcommand{\cS}{\mathcal{S}}
\newcommand{\cN}{\mathcal{N}}
\newcommand{\cW}{\mathcal{W}}
\newcommand{\cX}{\mathcal{X}}
\newcommand{\cR}{\mathcal{R}}
\newcommand{\vol}{\textnormal{Vol}}
\newcommand{\osc}{\textnormal{osc}}
\newcommand{\diam}{\textnormal{Diam}}

\newcommand{\CCC}{\mathfrak{C}}
\newcommand{\bF}{\mathbb{F}}

\newtheorem{theorem}{Theorem}[section]
\newtheorem{lemma}[theorem]{Lemma}
\newtheorem{corollary}[theorem]{Corollary}
\newtheorem{proposition}[theorem]{Proposition}

\newtheorem{conjecture}[theorem]{Conjecture}

\numberwithin{equation}{section}

\theoremstyle{definition}
\newtheorem{remark}[theorem]{Remark}

\theoremstyle{definition}
\newtheorem{definition}[theorem]{Definition}

\begin{document}

\title{Finite time singularities of the K\"ahler-Ricci flow}

\author{Wangjian Jian$^*$, Jian Song$^\dagger$ and Gang Tian$^\ddagger$}

\thanks{Wangjian Jian is supported in part by NSFC No.12201610, NSFC No.12288201, MOST No.2021YFA1003100. Jian Song is supported in part by National Science Foundation grant DMS-2203607. Gang Tian is supported in part by NSFC No.11890660, MOST No.2020YFA0712800.}

\address{$^*$ Institute of Mathematics, Academy of Mathematics and Systems Science, Chinese Academy of Sciences, Beijing, 100190, China}

\email{wangjian@amss.ac.cn}

\address{$^\dagger$ Department of Mathematics, Rutgers University, Piscataway, NJ 08854, USA}

\email{jiansong@math.rutgers.edu}

\address{$^\ddagger$ BICMR and SMS, Peking University, Beijing 100871, China}

\email{gtian@math.pku.edu.cn}

\begin{abstract}  We establish the scalar curvature and distance bounds, extending Perelman's work on the Fano K\"ahler-Ricci flow to general finite time solutions of the K\"ahler-Ricci flow. These bounds are achieved by our Li-Yau type and Harnack estimates for weighted Ricci potential functions of the K\"ahler-Ricci flow. We further prove that the Type I blow-ups of the finite time solution always sub-converge in Gromov-Hausdorff sense to an ancient solution on a family of analytic normal varieties with suitable choices of base points. As a consequence, the Type I diameter bound is proved for almost every fibre of collapsing solutions of the K\"ahler-Ricci flow on a Fano fibre bundle. We also apply our estimates to show that every solution of the K\"ahler-Ricci flow with Calabi symmetry must develop Type I singularities,  including both cases of high codimensional contractions and fibre collapsing.

\end{abstract}

\maketitle

{\footnotesize \tableofcontents}


\section{Introduction}\label{intro}

Ricci flow was introduced by R. Hamilton (cf. \cite{Ham}) in 1980s. In last few decades, it has been used to study both topological and geometric classifications of Riemannian manifolds, for instance, a program initiated by Hamilton aims at solving Thurston's Geometrization Conjecture and was completed by Perelman's groundbreaking work (cf. \cite{Per1, Per2, Per3}), in particular, it led to the settlement of the Poincar\'e conjecture.
The Ricci flow also provides a canonical deformation of K\"ahler metrics toward canonical metrics and gives an alternative proof of the existence of K\"ahler-Einstein metrics on a compact K\"ahler manifold with numerically trivial or ample canonical bundle (cf. \cite{Cao}). However, most projective manifolds, or more generally, compact K\"ahler manifolds, do not have a numerically semi-definite canonical bundle and the Ricci flow will develop finite time singularities.
The analytic minimal model program through Ricci flow was proposed by the second and third named authors \cite{ST4} aiming at classifying compact K\"ahler manifolds.
In particular, it provides an alternative approach to classify projective manifolds birationally. Many exciting progresses have been made (cf. \cite{Ts, TiZha, ST1, ST2, ST3, TiZZL162, W18, GPSS,  FoZhZ, TWY, STZ, JS, FL,CL, SW3, Tos, LiTo, FoZh} etc.), including long time existence and convergence to canonical metrics when the canonical class $K_X$ is nef, uniform bound on scalar curvature, Gromov-Hausdorff compactness, diameter bound as well as regularity results. However, fewer results have been proved for finite time singularities and their geometric properties. The purpose of this paper and its sequels is to develop a set of new techniques and estimates for studying finite time singularity of Ricci flow on general K\"ahler manifolds, that is, K\"ahler-Ricci flow. As applications, we will also give fine analysis for finite time singularity in case of certain projective manifolds. Now let us discuss in more details.

Let $X$ be a compact K\"ahler manifold of $\dim_{\mathbb{C}} X =n$ equipped with a K\"ahler metric $g_0\in H^{1,1}(X, \mathbb{R})\cap H^2(X, \mathbb{Q})$. The (unnormalized) K\"ahler-Ricci flow is given by
\begin{equation}\label{unkrflow1'1}
\left\{
\begin{array}{l}
{ \displaystyle \ddt{g(t)} = -\ric(g(t)) ,}\\
\\
g(0)=g_0.
\end{array} \right.
\end{equation}
By \cite{TiZha}, the flow admits a unique smooth solution $g(t)$ for $t\in [0, T)$ for the maximal time $T>0$  given by
\begin{equation}\label{maxT}
T= \sup\{ t>0~|~ [g_0] + tK_X >0 \}.
\end{equation}
When $K_X$ is not nef, the flow must develop finite time singularities at $t=T$. Since we assume the initial K\"ahler class $[g_0] \in H^2(X, \mathbb{Q})$,  the limiting cohomology class $\vartheta = [g_0] + T [K_X]$ is semi-ample by Kawamata's base point free theorem and it induces a unique surjective holomorphic map
\begin{equation}\label{bpfk}
\Phi: X \rightarrow Y \subset \mathbb{CP}^N,
\end{equation}
where $Y$ is a normal variety and $\vartheta$ is the pullback of a K\"ahler class on $Y$. In particular, $\Phi^{-1}(q)$ is a Fano variety  for generic non-biholomorphic point $q\in Y$ of $\Phi$.

\medskip

We present the following conjectural picture as a refined and extended geometric description of {\it the analytic minimal model program with Ricci flow} proposed in \cite{ST1, ST2, ST4}.

\begin{enumerate}

\item When $Y$ is a point, $X$ is Fano, i.e., $c_1(X)>0$. A striking theorem of Perelman states that the type I blow-up of $g(t)$ has uniformly bounded diameter and scalar curvature (cf. \cite{SeT}).  More precisely, if the initial metric $g_0\in c_1(X)$, then the K\"ahler-Ricci flow (\ref{unkrflow1'1}) develops singularities at $T=1$ and there exists $C>0$ such that for all $t\in [0, 1)$,
\begin{equation}\label{pers}
|\rr(x, t)| \leq \frac{C}{1-t},
\end{equation}
\begin{equation}\label{perd}
\textnormal{Diam}(X, g(t)) \leq C(1-t)^{\frac{1}{2}},
\end{equation}
where $\rr(t)$ is the scalar curvature of $g(t)$ .
This leads to the convergence of the Fano K\"ahler-Ricci flow to a unique (singular) K\"ahler-Ricci soliton on a $\mathbb{Q}$-Fano variety (cf. \cite{TiZXH1, TiZXH2, CSW}).

\medskip

\item When $0<\dim Y<n$, $X$ is a Mori fibration (general fiber is Fano). The flow is conjectured to collapse onto $Y$ and to extend through $t=T$ on $Y$ in the Gromov-Hausdorff topology. Furthermore, for any $q\in Y$, the type I blow-up based at $q$ should converge to a complete non-flat K\"ahler-Ricci soliton on a quasi-projective variety conjecturally. In particular, the limiting variety should be a Mori fibration over a quasi-projective variety and every tangent flow should be a Ricci-flat cone.

\medskip

\item When $\dim Y= n$, it is conjectured in \cite{ST4} that the flow will extend through $t=T$ geometrically associated to a birational transform such as a divisorial contraction or flip. Furthermore, for any critical value $q\in Y$ of $\Phi$, the forward and backward Type I blow-ups based at $q$ and $t=T$ should converge to a complete  shrinking K\"ahler-Ricci soliton on a quasi-projective variety $X_T^-$ as $t\rightarrow T^-$ and a complete expanding K\"ahler-Ricci soliton on a quasi-projective variety $X_T^+$ as $t\rightarrow T^+$ conjecturally. It is conjectured in \cite{So13} that the blow-down limits of such a shrinking and expanding solitons should coincide with the tangent cone at $\Phi^{-1}(q)$ for the original flow at $t=T$. $X_T^-$ and $X^+_T$ are birational equivalent and related by a unique algebraic flip.  We hope such soliton transitions at finite time singularities of the K\"ahler-Ricci flow will give a metric uniformization for the special family non-Gorenstein singularities arising from algebraic flips.

\end{enumerate}
Though we present the above picture in the case of projective manifolds, we believe that this picture, with slight modifications, also works for general compact K\"ahler manifolds.

The Fano K\"ahler-Ricci flow in case (1) have produced many deep results built on Perelman's fundamental estimates on scalar curvature and diameter (cf. \cite{SeT}). The well-known Hamilton-Tian conjecture was first proved in the case of Fano manifolds which admit K\"ahler-Einstein metrics (cf. \cite{TiZXH1, TiZXH2}) and later solved in general cases (cf. \cite{TiZZL16, Bam18, CW3}). For more results in this case, see \cite{TiZZZ, PS, PSSW, DeS, HL, GPS, JShi} and references therein.

It is extremely challenging to extend Perelman's estimates to case (2) and (3). We propose the following conjecture as an extension of Perelman's work on the Fano K\"ahler-Ricci flow to all finite time solutions of the K\"ahler-Ricci flow.

\begin{conjecture}\label{mainconj}
Let $g(t)$ be the smooth solution of the K\"ahler-Ricci flow (\ref{unkrflow1'1}) for $t\in [0, T)$. If $T<\infty$, then there exists $C<\infty$ such that for all $(x, t) \in X\times [0, T)$, we have
$$|\rr(x, t)| \leq \frac{C}{T-t}, $$
and for any  $q\in Y$ as in (\ref{bpfk}),
$$\textnormal{Diam}(\Phi^{-1}(q), g(t)) \leq C(T-t)^{\frac{1}{2}}, $$
where $\rr(t)$ is the scalar curvature of $g(t)$ and $\textnormal{Diam}(\Phi^{-1}(q), g(t))$ is the diameter of the pre-image of $q$ by $\Phi$ in $(X, g(t))$.
\end{conjecture}
Conjecture \ref{mainconj} is equivalent to Perelman's estimates (\ref{pers}) and (\ref{perd}) in the case when $X$ is Fano and $g_0\in c_1(X)$ with $Y$ being a single point in (\ref{bpfk}). %
It is well known that for any $q\in Y$, the fibre $\mathcal{F}_q=\Phi^{-1}(q)$  over $q$ is a connected subvariety in $X$. For generic $q$, $\mathcal{F}_q$ is smooth Fano manifold of complex dimension $n-\dim_{\mathbb{C}}Y$. But in general, $\mathcal{F}_q$ can be singular with possibly multiple components. The union of nontrivial $\mathcal{F}_q$, i.e. fibres over non-biholomorphic point $q\in Y$ of $\Phi$, is exactly the singular set of the finite time solution $g(t)$ as $t\rightarrow T$. Conjecture \ref{mainconj} should be the most fundamental step to understand formation of singularities in terms of both geometric and algebraic structures in the Analytic MMP in cases (2) and (3). In this paper, we will develop new techniques and estimates towards Conjecture \ref{mainconj} and prove some new results on this conjecture. Under some symmetry assumptions on $X$, we will be able to solve this conjecture.

\medskip

Let us also review some recent progress in case (2) and (3).  A rough scalar curvature upper bound is obtained in \cite{ZhangZ}. In case (2), the collapsing solutoins on Fano bundle have been studied in \cite{SSW, FZ, TZy}, where various estimates are obtained for the flow. For non-collapsing solutions in case (3), a uniform diameter bound is  proved in \cite{GPSS}.  The Gromov-Hausdorff continuation through finite time singularities are established in \cite{SW0, SW1, SW2, SY, So13} for divisorial contractions and flips for K\"ahler surfaces and  for higher dimensional examples with local Mumford quotients. However, very little is known for Type I blow-ups of finite time solutions with quantitive estimates in the spirit of Conjecture \ref{mainconj} beyond Perelman's fundamental work in the Fano Ricci flow, except for special solutions with large symmetry \cite{FIK, Fong, So1, GS}. We also note that the continuity method to study the minimal model program is proposed in \cite{LaTi}.

The recent developments in the regularity theory of Ricci flow also provide powerful tools for the study of the K\"ahler-Ricci flow as the parabolic analogue of the Cheeger-Colding theory \cite{CC1, CC2}. The Nash entropy along Ricci flow was introduced in \cite{HN}, where the Sobolev inequalities and $\varepsilon$-regularity theorem are obtained. Later in \cite{Bam20a}, \cite{Bam20b} and \cite{Bam20c}, a brand-new regularity theory on the Ricci flow based on the Nash entropy is established, including the heat kernel bounds, volume estimates and $\varepsilon$-regularity theorem, as well as the compactness theory of Ricci flow and the structure theory of non-collapsed limits of Ricci flows. Based on these new theories, the first and second author (cf. \cite{JS}) gave a new proof of the relative volume estimate along the Ricci flow in \cite{TiZZL20}; the first author (cf. \cite{J}) also obtain an improved version of Perelman's volume non-collapsing estimate. These powerful theories continue to play a central role in this paper and in \cite{JST23b, HJST}.

\medskip
\subsection{The scalar curvature estimate and Li-Yau type estimate }\label{be}
~

\smallskip

Suppose the K\"ahler-Ricci flow (\ref{unkrflow1'1}) develops finite time singularities at $T>0$. Without loss of generality, we can assume $T=1$. The limiting cohomology class $\vartheta= [g_0]+ [K_X] $ induces the unique morphism $\Phi: X \rightarrow Y\subset \mathbb{CP}^N$ as in (\ref{bpfk}).   $Y$ is a normal projective variety and $\dim Y \leq n$ is equal to the Kodaira dimension of $\vartheta$. We let $\theta_Y$ be a smooth closed $(1,1)$-form on $Y$ with
\begin{equation}\label{puly}
\Phi^*\theta_Y\in \vartheta.
\end{equation}
That is, $\theta_Y$ is the restriction of a local smooth closed $(1,1)$-form  through a local embedding of $Y$ into some $\mathbb{C}^M$ since $Y$ is normal. For example, we can choose $\theta_Y$ to be the multiple of the Fubini-Study metric $\mathbb{CP}^N$ restricted to $Y \subset \mathbb{CP}^N$.
We abuse the notations by identifying $\theta_Y$ with $\Phi^*\theta_Y$  for conveniences, and say $\theta_Y\in \vartheta$ is a smooth closed $(1,1)$-form on $Y$.

For fixed $\theta_Y$, there exists $u\in C^\infty(X \times [0, 1))$ such that
$$\ric(g(t)) - (1-t)^{-1} g(t) = - (1-t)^{-1} \theta_Y  - \ddbar  u.$$

\begin{definition} \label{defpot}
%
A point $p$ is said to be a {\bf Ricci vertex} at $t\in [0, 1)$ associated to  $\theta_Y$ if
$$ u(p, t) = \inf_X u(\cdot, t). $$
\end{definition}
We have the following local Type I scalar curvature bound near the Ricci vertex as a partial solution to Conjecture \ref{mainconj}.
\begin{theorem}\label{cor1}
Let $g(t)$ be the maximal solution of the K\"ahler-Ricci flow (\ref{unkrflow1'1}) on $X\times [0, 1)$ and $\theta_Y\in \vartheta$ be any smooth closed $(1,1)$-form on $Y$. Then there exists $C=C(n, g_0, \theta_Y)>0$ such that for any Ricci vertex $p$ at $t$ associated to $\theta_Y$, we have
\begin{equation}
(1-t) |\rr(x, t)| \leq  C \left(1+ \frac{d^2_t(x, p)}{1-t} \right),
\end{equation}
for any $(x, t)\in X\times [0, 1)$. In particular, for any $D>0$,
$$\sup_{ B\left( p, t, D(1-t)^{1/2}\right)} |\rr(\cdot, t)| \leq \frac{C(1+D^2)}{1-t},$$
where $B(p, t, r)$ is the geodesic ball centered at $p$ of radius $r$ with respect to $g(t)$.
%
%
\end{theorem}

\begin{remark}\label{scbonPnbhd}
The local Type I scalar curvature bound in Theorem \ref{mainconj} holds not only in the time-slice under consideration, but also in space-time, both on the usual parabolic neighborhood and the $P^*$-parabolic neighborhood. See Corollary \ref{scbonPnbhd2} and the remark after it. More precisely, for any $A, D>0$ and $\delta\in (0,1)$, we have
$$\sup_{ P\left(p,t; D(1-t)^{1/2}, -A(1-t), \delta(1-t)\right) } |\rr| \leq \frac{C}{1-t} , $$
$$\sup_{ P^*\left(p,t; D(1-t)^{1/2}, -A(1-t), \delta(1-t)\right) } |\rr| \leq \frac{C}{1-t} , $$
where $C=C(A, D, \delta)>0$. These estimates allow us to obtain the geometric regularity of the Type-I blow-up $\bF$-limits around the Ricci vertices. See \cite{HJST} for more details.
\end{remark}
Theorem \ref{cor1} shows that for finite time solutions of the K\"ahler-Ricci flow, one would always be able to find a region with uniform Type I scalar curvature bounds with at worst quadratic growth. This will immediately lead to local non-collapsing and pointed Gromov-Hausdorff compactness for each time slice of any Type I blow-up by combining Perelman's $\kappa$-noncollapsing. The quadratic growth of the scalar curvature is also essential in proving effective distance distortion and geometric compactness in \cite{HJST}. Theorem \ref{cor1} is particularly striking in the case when $X$ is a Fano fibration and the flow collapses its Fano fibres. A typical example will be the K\"ahler-Ricci flow on $X= \mathbb{CP}^{n_1}\times \mathbb{CP}^{n_2}$ that collapses either $\mathbb{CP}^{n_1}$ or $\mathbb{CP}^{n_2}$, whenever the initial K\"ahler class is not proportional to $c_1(X)$. A general discussion will be given in Section \ref{Fanobundle}.

There is a lot flexibility in the definition of the Ricci vertex. For given $t\in [0,1)$, let $p$ be the point in the Definition \ref{defpot}, for any radius $r>0$, if we define $q\in B(p,t, r(1-t)^{1/2})$ as the Ricci vertex associated to $\theta_Y$, then Theorem \ref{cor1} (and all the results concerning the Ricci vertex) still hold, with constants further depending on $r$.

In general, the location of the Ricci vertices is extremely difficult to determine as one would hope they can be chosen to be sufficiently close to the singular set. Suitable choices of the $\theta_Y$ will help us locate the Ricci vertices close to a base point in the singular locus of $\Phi: X \rightarrow Y$. This is indeed achieved in Theorem \ref{app3} and Theorem \ref{calmain1} in order to establish strong curvature bounds and effective distance estimates.

%
We define the (weighted) {\bf Ricci potential associated to $\theta_Y$}  by
\begin{equation}\label{defrp2}
v(x, t) = u(x, t) -  \inf_{X} u(\cdot, t)   + 1.
\end{equation}
%
The Ricci potential $v$ is uniquely determined by $\theta_Y$. Moreover, $v(\cdot, t)\in C^\infty(X)$ for all $t\in[0, 1)$, and we have
$$\ric(g(t)) -  (1-t)^{-1} g(t)= - (1-t)^{-1} \theta_Y  - \ddbar v  .$$
Theorem \ref{cor1} is in fact the consequence of the following Li-Yau type estimates for the  gradient and Laplacian estimates of the Ricci potential.

\begin{theorem}\label{main1}
Let $g(t)$ be the maximal solution of the K\"ahler-Ricci flow (\ref{unkrflow1'1}) on $X\times [0, 1)$ and $\theta_Y\in \vartheta$ be any smooth closed $(1,1)$-form on $Y$. Let $v$ be the Ricci potential associated to $\theta_Y$. Then  there exists  $C=C(n, g_0, \theta_Y)>0$, such that
\begin{equation} \label{gleou'0}
\frac{ \left| \Delta v \right| }{ v } + \frac{ |\nabla v|^2}{ v } \leq \frac{C}{1-t} ,
\end{equation}
%
%
%
%
%
on $X\times [0, 1 )$, where the gradient and Laplacian are with respect to $g(t)$.
\end{theorem}
Theorem \ref{main1} is a natural extension of Perelman's estimates (cf. \cite{SeT}) from finite time extinction to general finite time solutions of the K\"ahler-Ricci flow.
As we mentioned earlier, it is very challenging to locate the Ricci vertex as we would hope that the Ricci vertex can be sufficiently close to the pre-image of a given non-regular value of $\Phi$, where the curvature must blow up as $t\rightarrow 1$. We are able to force the Ricci vertices to stay close to a fix fibre $F_q$ for any $q\in Y$ as below, although it is still far from satisfying.
For any $q\in Y$, we can locally embed an open neighborhood $U$ of $q$ into some Euclidean space $\mathbb{C}^M$ with $0$ being the image of $p$. Let $w_1, ..., w_M$ be the standard complex coordinates of $\mathbb{C}^M$. Let $u$ be a Ricci potential associated with $\theta_Y$. We will choose a smooth function $\rho_q$ on $Y$ such that near $q$,
$$ \rho_q = A |w|^2$$
for some suitable $A>0$. Then for any small $\epsilon>0$ such that for all $t$ sufficiently close to $1$,   there exists a Ricci vertex  $p$ at $t$ associated with $\theta'_Y =\theta_Y + \ddbar \rho_q$ such that $p\in \Phi^{-1}(B_{\theta_Y'}(q, \epsilon))$, see Lemma \ref{a'1} for more details. This is certainly not optimal, see Conjecture \ref{maincon3}.

We would also like to conjecture that the Ricci vertex and the $H_{2n}$-center defined by Bamler (cf. \cite{Bam20a}) will be close to each other at the scale of $(1-t)^{1/2}$.
%


\medskip
\subsection{The distance estimate and Harnack estimate}
~
\smallskip

In Perelman's estimates for the scalar curvature and diameter bounds for the Fano K\"ahler-Ricci flow, a contraction argument is employed to prove the diameter bound based on his gradient and Laplacian estimates similar to Theorem \ref{main1}. Unfortunately, this approach does not seem likely to be adapted for general finite time solutions of the K\"ahler-Ricci flow. We have to develop a new scheme in the spirit of the Harnack estimate for the Ricci potential.

Consider the K\"ahler-Ricci flow on a projective manifold $X$ of $\dim_{\mathbb{C}}=n$ as in the previous section. Recall $\Phi: X \rightarrow Y \subset \mathbb{CP}^N$ is the unique surjective holomorphic map induced by the limiting class $\vartheta$. Our goal is to obtain diameter and scalar curvature estimates in a suitable (Type I) neighborhood of $\Phi^{-1}(q)$ for $q\in Y$. For conveniences, we fix a smooth K\"ahler metric $\mathrm{g}_Y$ on the base space $Y$. For example, we can choose $\mathrm{g}_Y$ as a multiple of the Fubini-Study metric of $\mathbb{CP}^N$.

We let
$$\mathcal{B}_{\mathrm{g}_Y}(p, r) = \Phi^{-1} \left( B_{\mathrm{g}_Y}(\Phi(p), r)    \right)$$
be the tubular neighborhood of the fibre containing $p$ with radius with respect to $\mathrm{g}_Y$.

\begin{theorem} \label{cor2}
Let $g(t)$ be a maximal solution of the K\"ahler-Ricci flow (\ref{unkrflow1'1}) on $X\times [0, 1)$ and let $p_t$ be a Ricci vertex associated to a smooth closed $(1,1)$-form $\theta_Y\in \vartheta$ on $Y$. Then for any $A>0$, there exist $C=C(n, g_0, \theta_Y, \mathrm{g}_Y, A)>0$ such that for any $t\in [0, 1)$, if
\begin{equation}\label{cor2a}
\vol\left( \mathcal{B}_{\mathrm{g}_Y}\left( p_t, 10(1-t)^{\frac{1}{2}} \right) , g(t) \right)\leq A (1-t)^{n},
\end{equation}
then
\begin{equation} \label{cor2'1}
\diam \left( \mathcal{B}_{\mathrm{g}_Y}\left( p_{t}, (1-t)^{\frac{1}{2}} \right), g(t) \right)  \leq C(1-t)^{\frac{1}{2}} ,
\end{equation}
\begin{equation} \label{cor2'2}
\sup_{\mathcal{B}_{\mathrm{g}_Y}\left( p_{t}, (1-t)^{\frac{1}{2}} \right) } | \rr( \cdot , t ) | \leq \frac{C}{1-t} ,
\end{equation}
where $\rr(\cdot, t)$ denotes the scalar curvature of $g(t)$. In particular, Conjecture \ref{mainconj} holds on $\mathcal{F}_{\Phi(p)}=\Phi^{-1}(\Phi(p))$, i.e.,
$$
\diam \left( \mathcal{F}_{\Phi(p_t)}, g(t) \right)  \leq C(1-t)^{\frac{1}{2}}, ~\sup_{\mathcal{F}_{\Phi(p_t)} } | \rr( \cdot , t ) | \leq \frac{C}{1-t}.
$$
\end{theorem}

Theorem \ref{cor2} says that a Type I volume bound around the Ricci vertex will imply a Type I diameter bound. Similarly to Remark \ref{scbonPnbhd},  the diameter estimate (\ref{cor2'1}) and scalar curvature estimate (\ref{cor2'2}) hold not only in the time-slice under consideration, but in both the usual parabolic neighborhood and the $P^*$-parabolic neighborhood based at the Ricci vertex. We refer the readers for more details in Theorem \ref{scanddearv}. In this paper, the diameter for a subset $U$ of $(X, g)$ is calculated in the ambient space $(X, g)$ as the extrinsic diameter.

When $X$ is a Fano manifold and $g_0\in c_1(X)$, the volume assumption (\ref{cor2a}) holds automatically because the base space $Y$ is a single point and the total volume of $(X, g(t))$ is the topological intersection number $[g(t)]^n$ with $(1-t)^n$ decay. Therefore Theorem \ref{cor2} provides a new proof for Perelman's diameter and scalar curvature estimates.  We will also verify the condition (\ref{cor2a}) in the case of Fano bundles. As a consequence, we obtain the Type-I diameter and scalar curvature estimates around the fibers of the Ricci vertex, see Section \ref{Fanobundle}.

Theorem \ref{cor2} is a consequence of the following general Harnack estimate on Ricci flow background.
\begin{theorem}\label{main2}
Let $(M,(g(t))_{t\in [0, 1)})$ be a Ricci flow on compact Riemannian manifold of $\dim_{\mathbb{R}} M=n$. Let $$v\in M\times [0, 1)\to \mathbb{R}^+$$ be a positive $C^1$-function.
Given any $\delta>0$, $B>0$, there exists  $\bar{\delta} = \bar{\delta}(n, B)>0$ such that whenever $\delta\in (0, \bar{\delta})$,  there exists  $C=C(n, B, \delta)>0$ such that the following holds. Suppose
\begin{enumerate}
\item $\nu(g(0), 2) \geq -B$;
\medskip

\item $  |\partial_t \ln v | + |\nabla \ln v|^2  \leq B/(1-t) $ on $M\times [0, 1)$ ;
\medskip

\item  $t_0\in [1/2, 1)$ and  $V$ is an open set  of $M$ with  $\vol_{g(t_0)}(V)\leq B (1-t_0)^{\frac{n}{2}}$ ;

\medskip

\item $U$ is a connected open subset  of $M$ and  for any $x\in U$, there exists an $H_n$-center $(z, t_0)$ of $(x, t_0 + \delta^2(1-t_0))$ with $B_{g(t_0)} \left( z, \sqrt{2H_n}\delta(1-t_0)^{1/2} \right) \subset V .$
\end{enumerate}
Then we have
\begin{equation}\label{harin}
\frac{\sup_{U} v ( \cdot, t_0 )}{ \inf_{U} v ( \cdot, t_0 ) }  \leq C.
\end{equation}
\end{theorem}
We shall highlight that in this theorem, we do not  make any assumption on the curvature of the Ricci flow, hence we can apply this theorem when we a priori do not have any control on the curvature of the background Ricci flow.

Condition (4)  of Theorem \ref{main2} is in fact a natural assumption in order to obtain the Harnack estimate as it holds under the assumption of Theorem \ref{cor2} for finite time solutions of the K\"ahler-Ricci flow. One of the key ingredients in the proof of Theorem \ref{main2} is  Bamler's volume non-collapsing estimate near the $H_n$-center \cite{Bam20a} by only assuming the scalar curvature lower bound and the entropy bound.
%
%


\medskip

\section{Applications}\label{app}

We will apply our estimates in Section \ref{be} to both general and special finite time solutions of the K\"ahler-Ricci flow.  Let $g(t)$ be the maximal solution of the K\"ahler-Ricci flow (\ref{unkrflow1'1}) on  $X\times [0, 1)$, where $X$ is a projective manifold of $\dim X =n$. We will use the same notations as in Section \ref{intro}. Let $\vartheta = \lim_{t\rightarrow 1} [ g(t)] $ be the limiting cohomology class at the singular time $T=1$. The unique surjective holomorphic map induced by $\vartheta$ is given by
$$\Phi: X \rightarrow Y, $$
where $Y$ is a projective normal variety of $0\leq \dim Y \leq n$.


\medskip
\subsection{Type I blow-up limits}\label{Analyticstr}

~

\smallskip
In this section, we will prove a compactness theorem for Type I blow-up limits of finite time solutions of the K\"ahler-Ricci flow based at the Ricci vertices.

Consider any sequence of times $t_i\nearrow 1$. Let $(M_i, (g_{i,t})_{t\in [-T_i , 0]})$ be the flows arise from $(X, (g(t))_{t\in [0, 1) })$ by setting
$$ M_i:=X,~~~ g_{i,t}:=(1-t_i)^{-1}g((1-t_i)t+t_i),~~~ t\in [-T_i , 0], $$
where $T_i=t_i/(1-t_i)\to \infty$ as $i\to\infty$.  The $\bF$-convergence of $(M_i, (g_{i,t})_{t\in [-T_i , 0]})$ is established in \cite{Bam20b} based at the $H_{2n}$-centers. By \cite{Bam20b}, let $p_i\in X$ be any base-point, after possibly passing to a subsequence, we can obtain the $\bF$-convergence
\begin{equation}\label{FcorKRF'1}
(M_i, (g_{i,t})_{t\in [-T_i , 0]}, (\nu_{p_i,0; t})_{t\in [-T_i , 0]}) \xrightarrow[i\to\infty]{\bF,\CCC}  (\cX, (\nu_{p_\infty; t})_{t\in (-\infty , 0]})
\end{equation}
within some correspondence $\CCC$, where $\cX$ is a future continuous and $H_{2n}$-concentrated metric flow of full support over $(-\infty , 0]$.

According to \cite{Bam20c}, we can decompose $\cX$ into it's regular and singular part
\begin{equation}\label{rsd'2}
\cX= \cR \sqcup \cS,
\end{equation}
where $\cR$ is dense open subset of $\cX$. Also, $\cR$ carries the structure of a Ricci flow spacetime $(\cR, \mathfrak{t}, \partial_{\mathfrak{t}}, g)$. For any $t\in (-\infty, 0)$, $\cR_t=\cX_t\cap\cR$, we have $(\cX_t, d_t)$ is the metric completion of $(\cR_t, g_t)$.

Our main result is that, based at the Ricci vertices of a given $\theta_Y\in \vartheta$ on $Y$, we establish the Gromov-Hausdorff convergence for $(X_i, (g_{i,t})_{t\in [-T_i , 0]})$, and identify each limiting time-slice with an analytic normal variety instead of a metric space.

\begin{theorem} \label{2main1}
Let $\theta_Y\in \vartheta$ be any smooth closed $(1,1)$-form on $Y$, and let $p_i\in X$ be a Ricci vertex associated with $\theta_Y$ at $t_i$. Then $\cX$ satisfying the following properties.

\begin{enumerate}


 \item Passing to a subsequence, for each $t\in (-\infty, 0]$, $(X, d_{g_{i, t}}, p_i)$ converge in pointed Gromov-Hausdorff sense to $(\cX_t, d_t, q_t)$ for some $q_t\in \cX_{t}$, where $d_t$ is metric induced by $g_t$ on $\cR_t$.

 \smallskip

 \item For each $t\in (-\infty, 0]$, $(\cX_t, d_t)$ is an analytic normal variety of complex dimension $n$,  whose analytic singular set coincides with $\cS_t$.

\end{enumerate}
\end{theorem}
In particular, each time-slice is a singular space (in the sense of \cite{Bam18}) with the singular set of Minkowski co-dimension $\geq 4$ and the corresponding complex singularities are at worst log terminal. It is further proved in \cite{HJST} that $(\cX_t, d_t)$ is equi-continuous in  $t$ with respect to Gromov-Hausdorff distance and Gromov-$W_1$-distance.

\begin{conjecture}
In Theorem \ref{2main1}, $(\cX_t, d_t)$ is a quasi-projective normal variety for each $t\in (-\infty, 0]$. Furthermore, by choosing suitable Ricci vertices, $\cX$ is a complete K\"ahler-Ricci soliton.
%
%
\end{conjecture}

\begin{remark}
For a given point $q\in Y$, if we can verify Conjecture \ref{maincon3} and Conjecture \ref{maincon2} for $q$, then for any sequence of basepoints $p_i\in \Phi^{-1}(q)$, the limiting metric flow $\cX$ in Theorem \ref{2main1} is a shrinking metric soliton on $\cX=\cX_{-1}\times (-\infty, 0]$.
\end{remark}
%


\medskip
\subsection{Collapsing K\"ahler-Ricci flow on Fano bundles}\label{Fanobundle}

~

\smallskip
In this section, we will consider the finite time collapsing solution of the K\"ahler-Ricci flow. Let $g(t)$ be the maximal solution of the K\"ahler-Ricci flow (\ref{unkrflow1'1}) on $X \times [0, 1)$, where $X$ is a projective manifold of $\dim_{\mathbb{C}}=n$.  Let $$\Phi: X \rightarrow Y$$ be the surjective holomorphic map induced by the limiting class $\vartheta=\lim_{t\rightarrow 1} [g(t)]$.
We will consider the special case satisfying the following.

\begin{enumerate}

\item $0\leq m=\dim_{\mathbb{C}} Y<n$.

\medskip

\item The general fibers of $\Phi: X\rightarrow Y$ are biholomorphic.

\end{enumerate}

Condition (1) implies that $X$ is a Fano fibration over $Y$, where the general fibre is a Fano manifold of complex dimension $n-m$. Condition (2) implies that $X$ is a Fano fibre bundle over a Zariski open dense subset of $Y$. When $Y$ is a point, $X$ is a Fano manifold itself. We let $Y^\circ$ be the set of regular values of $\Phi$. Then $Y^\circ$ is a  Zariski open dense subset of $Y$ and
$$X^\circ=\Phi^{-1}(Y^\circ)$$
is a smooth Fano fiber bundle over $Y^\circ$.

The following diameter and scalar curvature estimates for the collapsing solution $g(t)$ are established around the Ricci vertex.

\begin{theorem} \label{app3}
Let $g(t)$ be the maximal solution of the K\"ahler-Ricci flow (\ref{unkrflow1'1}) on $X\times [0, 1)$ as described above. For any open subset $U\subset\subset Y^\circ$, there exists a smooth closed $(1,1)$-form $\theta_Y\in \vartheta$ on $Y$, such that the followings hold for all $t\in [0, 1)$ .

\smallskip

\begin{enumerate}

\item There exists a Ricci vertex $p_t \in \Phi^{-1}(U) $ associated to $\theta_Y$ at time $t$.

\medskip

\item There exists $C=C(n, g_0, \theta_Y, U)>0$ such that
\begin{equation} \label{app3'1}
\diam \left( \mathcal{F}_{\Phi(p_t)}   , g(t) \right)  \leq C(1-t)^{1/2} ,
\end{equation}
\begin{equation} \label{app3'2}
\sup_{ \mathcal{F}_{\Phi(p_t)} } |\rr ( \cdot , t) |\leq \frac{C}{1-t},
\end{equation}
where $\mathcal{F}_{\Phi(p_t)}= \Phi^{-1}(\Phi(p_t))$ is the fibre of $\Phi: X \rightarrow Y$ that contains $p_t$.
\end{enumerate}
\end{theorem}

Theorem \ref{app3} extends Perelman's estimates for K\"ahler-Ricci flow on Fano manifolds as a special case when $\dim_{\mathbb{C}}Y=0$. In fact, the diameter and scalar curvature estimates hold not only for the fibre $\mathcal{F}_{\Phi(p_t)}$ but also for the Type I tubular neighborhood of $\mathcal{F}_{\Phi(p_t)}$ (see Theorem \ref{scanddeonFanobundle}). Theorem \ref{app3} also confirms Conjecture \ref{mainconj} near Ricci vertices in the case of the Fano bundles.

In section \ref{FDonRV}, we propose two conjectures about the Ricci vertices as intermediate steps toward Conjecture \ref{mainconj}. In the second one, say Conjecture \ref{maincon2}, we propose that for any point $q$ at bounded Type-I distance to the Ricci vertex, Conjecture \ref{mainconj} hold for $\Phi^{-1}(q)$.  When $X$ is Fano bundle over $Y$, then we have the following corollary as  $Y^\circ =Y$.

\begin{corollary} \label{mainbund}
Let $g(t)$ be the maximal solution of the K\"ahler-Ricci flow (\ref{unkrflow1'1}) on $X\times [0, 1)$ as described above. If $X$ is a fiber bundle over $Y$, then the following hold.

\begin{enumerate}

\item  The conclusions of Theorem \ref{app3} hold for every open subset of $Y$ since $Y=Y^\circ$.

\medskip

\item  Let  $(X, (g_{i,t})_{t\in [-T_i , 0]})$ be the Type I blow-up
$$  g_{i,t}:=(1-t_i)^{-1}g((1-t_i)t+t_i),~~~ t\in [-T_i , 0], ~T_i \rightarrow \infty $$
and let $p_i$ be a Ricci vertex associated to $\theta_Y\in \vartheta$ on $Y$. After passing to a subsequence, $(X, (g_{i, t})_{ t\in [-T_i , 0]}, p_i)$ converges to ancient solution of the  K\"ahler-Ricci flow on $\mathcal{X}_{t\in (-\infty, 0]}$ as in Theorem \ref{2main1}. For each $t\in (-\infty, 0]$, there exists a surjective holomorphic map
$$\Phi_t: \cX_t \rightarrow \mathbb{C}^m $$
as the pointed limit of $\Phi: (X, g_{i, t}, p_i) \rightarrow (Y, (1-t_i)^{-1}\mathrm{g}_Y, \Phi(p_i))$,
where $m=\dim_{\mathbb{C}} Y$ and $\mathrm{g}_Y$ is a fixed smooth K\"ahler metric on $Y$.
\end{enumerate}
\end{corollary}

In particular,  Corollary \ref{mainbund} confirms Conjecture \ref{maincon2} in the case of Fano bundles. The convergence in Corollary \ref{mainbund} is further smooth on K\"ahler surfaces.

\begin{corollary} \label{rules}
If $X$ is a ruled surface over a Riemann surface $Y$, the K\"ahler-Ricci flow will develop finite time singularities at some $t=T>0$. If the limiting class $\lim_{t\rightarrow T} [g(t)]$ is not big or trivial, then  the limiting space $\cX$ of the Type I blow-ups   in Corollary \ref{mainbund} must be
$$\cX=\mathbb{CP}^1\times \mathbb{C}.$$
In particular, the convergence is smooth.
\end{corollary}

A typical example of $X$ will be $\mathbb{CP}^1\times\mathbb{CP}^1$ as long as the initial K\"ahler class is not proportional to $c_1(X)$. If $X$ is a Hirzebruch surface $F_n$ with $n\geq 2$ or a ruled surface over a Riemann surface of genus greater than $0$, the limiting class of the solution cannot be big or trivial, and so  the result in Corollary \ref{rules} will always hold.

We conjecture that the limiting metric flow $\cX$ in Corollary \ref{mainbund} and Corollary \ref{rules} split in the base direction. Due to \cite{HJ}, in Corollary \ref{rules}, if we can prove the limiting metric split off a real line $\mathbb{R}$, then it split off a complex line $\mathbb{C}$, which is exactly the base space.


\medskip
\subsection{K\"ahler-Ricci flow with Calabi symmetry}\label{Calabisymmetry}

~
\smallskip

In this section, we will study the K\"ahler-Ricci flow on a family of K\"ahler manifolds with large symmetry and classify formation of finite time singularities of the flow.  Calabi \cite{Cal} introduced the family of $U(m+1)$-invariant K\"ahler metrics on $X$  in each K\"ahler class of $X$. These special K\"ahler metrics are said to satisfy the Calabi symmetry.  More precisely, we consider the case when $X$ is the following projective bundle
$$X= \mathbb{P}(\mathcal{O}_{Z} \oplus L^{\oplus (m+1)}) $$
over an $n$-dimensional K\"ahler-Einstein manifold $Z$, where $\mathcal{O}_Z$ is a trivial line bundle of $Z$ and $L$ a negative line bundle over $Z$. We assume that $L$ admits a negatively curved hermitian metric $h$ such that $\omega_Z=-{\rm Ric}(h)$ is a K\"ahler-Einstein metric over $Z$.
The K\"ahler-Ricci flow on $X$ must develop finite time singularities.  Our first main result is to establish the Type I curvature bounds for all solutions of the K\"ahler-Ricci flow with Calabi symmetry.

\begin{theorem} \label{calmain1}
Let $X= \mathbb{P}(\mathcal{O}_{Z} \oplus L^{\oplus (m+1)})$ be the K\"ahler manifold defined above. For any $U(m+1)$-invariant K\"ahler metric $g_0$ on $X$,  the K\"ahler-Ricci flow on $X$ starting with $g_0$  must develop Type I singularities.
\end{theorem}

The results of Theorem \ref{calmain1} are proved in \cite{Fong, So1, GS} in the special case when $X$ is $\mathbb{CP}^{n+1}$, i.e., $Z=\mathbb{CP}^n$ and $m=0$, where either the contraction is of complex co-dimension $1$ or the collapsing fibres are of complex dimension $1$. The case of $m\geq 1$ is much more complicated and the techniques of \cite{Fong, So1} cannot be applied to obtain Type I curvature bounds. In the case of non-collapsing solutions, formation of singularities corresponds to a high codimensional contraction as part of the algebraic flip surgery. It is proved in \cite{SY, So13}, the flow will converge Gromov-Hausdorff topology to the contraction model of $X$, replacing the $0$-section by an isolated singularity, however, without any control on the blow-up limits and curvatures. A crucial step in the proof of Theorem \ref{calmain1} is to verify Condjecture \ref{maincon3} using Theorem \ref{main1} when $q$ is a non-regular value of $\Phi$.

For any $p\in X$ and $t_i \nearrow 1$, the Type I blow-up of $g(t)$ is given by $g_i(t)= (1-t_i)^{-1}g((1-t_i)t+t_i)$ on $X\times [-t_i/( 1-t_i), 0)$. By the work of \cite{So1, GS} and Theorem \ref{calmain1}, the Type I blow-up of $g(t)$ based at $p$ always converge smoothly to a limiting shrinking gradient K\"ahler-Ricci soliton, after possibly passing to a subsequence. We now have the following classification of the Type I blow-up limits of $g(t)$ in Theorem \ref{calmain1}.

The projective bundle $X$ has a special $0$-section $P_0$ whose normal bundle is negative. If the K\"ahler-Ricci flow with Calabi symmetry develop singularities at $t=1$, the limiting cohomology class of the solution $g(t)$ will induce a unique surjective holomorphic map
$$\Phi: X \rightarrow Y$$
as discussed in the introduction. One and only one of the following cases will take place.

\begin{enumerate}

\item $Y$ is a point. In this case, both $X$ and $Z$ must be Fano.

\medskip

\item $Y = Z$, then $X$ is a Fano bundle.

\medskip

\item $Y$ is birational to $X$. $\Phi$ is contraction map such that $\Phi|_{X\setminus P_0}$ is the identity map and $\Phi(P_0)$ is a point of $Y$.   When $m>0$, $\Phi(P_0)$ is an isolated non-orbifold singularity and $\Phi$ is a small contraction.

\end{enumerate}

We are ready to obtain the classification of Type I blow-up limits for the finite time solutions of the K\"ahler-Ricci flow with Calabi symmetry.

\begin{corollary} \label{calmain2}
Suppose $g(t)$ is a solution of the K\"ahler-Ricci flow on $X= \mathbb{P}(\mathcal{O}_{Z} \oplus L^{\oplus (m+1)})$ with an initial $U(m+1)$-K\"ahler metric $g_0$. Let $[0, 1)$ be the maximal time for the flow. Then we have the following classification for the smooth limits of the Type I blow-up $(X, g_j(t), p)$  for any base point $p\in X$.

\begin{enumerate}

\item If $$\liminf_{t\rightarrow 1}{\rm Vol}_{g(t)}(X)>0, $$
the Type I blow-up limit based at any point $p \in P_0$ is the unique shrinking gradient K\"ahler-Ricci soliton metric with Calabi symmetry on the total space of the vector bundle $$ L^{\oplus (m+1)}$$
over $Z$.
The Type I blow-up limit based at any point $p\in X\setminus P_0$ is a flat Euclidean space $\mathbb{C}^{m+n+1}.$

\medskip

\item If
$$\liminf_{t\rightarrow 1}{\rm Vol}_{g(t)}(X)=0, ~\liminf_{t\rightarrow 1} (1-t)^{-(n+m+1)}{\rm Vol}_{g(t)} (X) =\infty,$$
the Type I blow-up limit based at any point $p \in X$ is the product of the flat $\mathbb{C}^n$ and $\mathbb{CP}^{m+1}$ equipped with a Fubini-Study metric.

\medskip

\item If $$\liminf_{t\rightarrow 1} (1-t)^{-(n+m+1)}{\rm Vol}_{g(t)}(X) <\infty,$$ $X$ is Fano and the Type I blow-up limit based at any point $x\in X$ is the unique K\"ahler-Ricci soliton on $X$.

\end{enumerate}

\end{corollary}

Theorem \ref{calmain1} and Corollary \ref{calmain2} generalize the results in \cite{So1, GS}, where $Z=\mathbb{CP}^n$ and $m=0$. When $Z=\mathbb{CP}^n$ and $m=0$, $X$ is simply $\mathbb{CP}^{n+1}$ blow-up at one point. The contracted variety $Y$ is $\mathbb{CP}^{n+1}$. The case of $m>0$ is much more challenging and delicate than the case of $m=0$ as the contracted $0$-section $P_0$ has complex codimension greater $1$.  This poses difficulties in both algebraic and differential geometry as the isolated singularity is not $\mathbb{Q}$-Gorenstein and thus not an orbifold singularity. Our Li-Yau type estimates in Theorem \ref{main1} and Harnack estimates Theorem \ref{main2} play the essential role to achieve the Type I curvature bounds.

This paper is organized as follows.

In section \ref{BandehkRF}, we recall some conventions and notations, and recall some known results that will be used.

In section \ref{setupofKRF}, we set up the corresponding complex Monge-Amp\`ere flows and build up basic estimates.

In section \ref{gandle}, we first establish the gradient and Laplacian estimate of perturbed Ricci potentail. Then we prove the Li-Yau type estimate (say Theorem \ref{main1}) and the local scalar curvature estimate (say Theorem \ref{cor1}).

In section \ref{sechar}, we first prove the Harnack estimate on Ricci flow, say Theorem \ref{main2}. Then we verify some conditions of this theorem in our K\"ahler-Ricci flow set-up, hence prove the distance estimate provided the volume bound, say Theorem \ref{cor2}.

In section \ref{hkanddde}, we prove the short time distance distortion estimate and heat kernel lower bound estimate locally around the Ricci vertices. Then we prove the Gromov-Hausdorff compactness of the Type-I blow-up $\bF$-limits. This proves item $(1)$ of Theorem \ref{2main1}.

In section \ref{cgc}, we prove the complex geometric compactness  of the Type-I blow-up $\bF$-limits, which proves item $(2)$ of Theorem \ref{2main1}.

In section \ref{dandseonbundle}, we verify the volume assumption in Theorem \ref{cor2} on Fano bundles, hence establish the results in section \ref{Fanobundle}.

In section \ref{calsym}, we consider the K\"ahler-Ricci flow on Calabi symmetry manifolds, and establish the results in section \ref{Calabisymmetry}.

In section \ref{FDonRV}, we propose some further conjectures on the Ricci vertices, which are intermediate steps towards Conjecture \ref{mainconj}.


\medskip
\section{ Background on entropy and heat kernels of the Ricci flow }\label{BandehkRF}


\medskip
\subsection{Notations and conventions}


%
Let $(M, g(t))_{t\in I}$ be a smooth Ricci flow on a compact $n$-dimensional manifold with the interval $I\subset \mathbb{R}$. For $(x_0, t_0)\in M\times I$, $A, T^-, T^+\geq 0$, the parabolic neighborhood is defined by
\begin{equation} \label{pnbd}
P(x_0, t_0; A, -T^-, T^+) = B(x_0, t_0, A) \times \left( [t_0- T^-, t_0+ T^+]\cap I\right),
\end{equation}
where we may omit $-T^-$ or $T^+$ if it is zero. The heat operator associated to $(M, g(t))$ is defined by
$$\Box = \ddt{} - \Delta$$
and the conjugate heat operator is defined by
$$\Box^* = - \ddt{} - \Delta + R, $$
where $\Delta$ is the Laplacian associated to $g(t)$ and $R$ is the scalar curvature of $g(t)$.

For any $(x, t)$, $(y, s) \in M \times I$ with $s\leq t$, we denote by $K(x,t; y,s)$ the heat kernel of the Ricci flow based at $(y,s)$ satisfying
\begin{equation}
\Box K(\cdot, \cdot; y,s)=0, ~~ \lim_{t\rightarrow s^+} K(\cdot, t, ; y, s) = \delta_y,
\end{equation}
where $\delta_y$ is the Dirac measure at $y$. Similarly, $K(x, t; \cdot,\cdot)$ is the conjugate heat kernel based at $(x,t)$ satisfying
\begin{equation}
\Box^* K(x, t; \cdot, \cdot)=0, ~ \lim_{s\rightarrow t^-} K(x, t; \cdot, s) =\delta_x.
\end{equation}

Using the conjugate heat kernel, we can define the conjugate heat measure $\nu_{x, t; s}$ based at $(x,t)$ by
\begin{equation}
d\nu_{x,t;s} = K(x,t; \cdot, s) dg(t) = (4\pi\tau)^{-n/2} e^{-f} dg(t),
\end{equation}
where $\tau=t-s$ and $f\in C^\infty(M \times (-\infty, t))$ is called the potential of the conjugate heat measure $\nu_{x,t;s}$.

For two probability measures $\mu_1$ and $\mu_2$ on a Riemannian manifold $(M, g)$, the Wasserstein $W_1$-distance between $\mu_1$ and $\mu_2$ is defined by
\begin{equation}
d^g_{W_1}(\mu_1, \mu_2) = \sup_f \left( \int_M f d\mu_1 - \int_M f d\mu_2 \right),
\end{equation}
where the supremum is taken over all bounded $1$-Lipschitz function on $(M, g)$. The variance between $\mu_1$ and $\mu_2$ is defined by
\begin{equation}
\var(\mu_1, \mu_2) = \int_{(x_1, x_2) \in M \times M}  d_g^2(x_1, x_2) d\mu_1(x) d\mu_2(x_2).
\end{equation}
We have the following basic relation between the Wasserstein $W_1$-distance and the variance
\begin{equation}
d^g_{W_1} (\mu_1, \mu_2) \leq \sqrt{\var(\mu_1, \mu_1)}.
\end{equation}

For any $(x_0, t_0)\in M\times I$, $A, T^-, T^+\geq 0$, we define the $P^*$-parabolic neighborhood
\begin{equation} \label{p*nbd}
P^*(x_0, t_0; A, -T^-, T^+) \subset M\times I ,
\end{equation}
as the set of $(x,t)\in M\times I$ with $t\in [t_0-T^-, t_0+ T^+]$ and
\begin{equation} \label{p*nbd2}
d^{g_{t_0-T^-}}_{W_1}( \nu_{x_0, t_0; t_0-T^-} , \nu_{x, t; t_0-T^-} ) < A .
\end{equation}
As before, we may omit $-T^-$ or $T^+$ if it is zero.

We now define the $H_n$-center at a base point along the Ricci flow.

\begin{definition}\label{hnc}
A point $(z, t)\in M\times I$ is called an $H_n$-center of a point $(x_0, t_0)\in M \times I$ if $t< t_0$ and
\begin{equation}
\var_t(\delta_z, \nu_{x_0, t_0; t}) \leq H_n (t_0 - t),
\end{equation}
where $\var_t$ is the variance with respect to the metric $g(t)$.

\end{definition}
Due to \cite[Proposition 3.12]{Bam20a}, given any $(x_0, t_0)$ and $t<t_0$, there exists at least one $H_n$-center of $(x_0, t_0)$. Immediately, if $(z, t)$ is an $H_n$-center of $(x_0, t_0)$, then we have
\begin{equation}\label{dvh}
d^{g_t}_{W_1}(\delta_z, \nu_{x_0, t_0; t} ) \leq \sqrt{\var(\delta_z, \nu_{x_0, t_0; t} )} \leq \sqrt{H_n(t_0-t)}.
\end{equation}
The following lemma if proved in \cite{Bam20a}, which asserts that the mass of the conjugate heat kernel measure will concentrate around the $H_n$-centers.
\begin{lemma} \label{hnc'2}
If the point $(z, t)$ is an $H_n$-center of $(x_0, t_0)$ with $t< t_0$, then for any $A>0$, we have
$$\nu_{x_0, t_0; t} \left( B\left( z, t, \sqrt{AH_n (t_0-t)} \right) \right) \geq 1- \frac{1}{A} .$$
\end{lemma}

We now define the Nash entropy introduced by Hein-Naber \cite{HN}. Let $d\nu= (4\pi \tau)^{-n/2} e^{-f} dg$ be a probability measure on a closed $n$-dimensional Riemannian manifold $(M, g)$ with $\tau>0$ and $f\in C^\infty(M)$. The Nash entropy is defined by
\begin{equation}
\cN[g, f, \tau] = \int_M f d\nu - \frac{n}{2}.
\end{equation}

We can rewrite the conjugate heat measure based at $(x_0, t_0)\in M\times I$ by
$$d\nu_{x_0, t_0; t} = K(x_0, t_0; \cdot, t) dg(t) = (4\pi \tau)^{-n/2} e^{-f(t)} dg(t),$$
where $\tau = t_0 - t\geq 0$. Then we define the pointed Nash entropy along Ricci flow based at $(x_0, t_0)$ by
\begin{equation}
\cN_{x_0, t_0}(\tau) = \cN [ g(t_0-\tau), f(t_0-\tau), \tau].
\end{equation}
Then we set $\cN_{x_0, t_0}(0) =0$, which makes $\cN_{x_0, t_0}(\tau)$ being continuous at $\tau=0$. We also define
\begin{equation}
\cN_s^*(x_0, t_0)= \cN_{x_0, t_0}(t_0-s),
\end{equation}
for $s< t_0$ and $s\in I$. The pointed Nash entropy $\cN_{x_0, t_0}(\tau)$ is non-increasing when $\tau\geq 0$ is increasing.

For any compact, n-dimensional manifold $(M, g)$, Perelman's $\cW$-functional is defined by, for any $\tau > 0$,
$$\cW[g, f, \tau] = (4\pi \tau)^{-n/2} \int_M \left(\tau ( |\nabla f|^2 + R) + f \right) e^{-f} dg,$$
with $f\in C^\infty(M)$ so that $\int_M (4\pi \tau)^{-n/2} e^{-f} dg =1$, and Perelman's $\mu$-functional and $\nu$-functional are defined by
$$\mu[g, \tau]=\underset{\int_M (4\pi \tau)^{-n/2} e^{-f} dg =1}{\inf} \cW[g, f, \tau],$$
and
$$\nu[g, \tau]=\underset{0<\tau'<\tau}{\inf} \mu[g, \tau'].$$
If $(M, (g_t)_{t\in [0, T)})$ is a Ricci flow, then the functions $t\to \mu[g_t, T-t]$ and $t\to \nu[g_t, T-t]$ are non-decreasing. It is proved in \cite{Bam20a} that
\begin{equation}\label{NM}
\cN^*_t (x_0, t_0) \geq \mu[g(t), t_0-t],
\end{equation}
for any $t<t_0$.
%


\medskip
\subsection{Preliminary results on entropy and heat kernel bounds}
In \cite{Bam20a}, Bamler established systematic results on the Nash entropy and heat kernel bounds on Ricci flow background. Let us recall some results that will be used in our theory.

The following quantitative volume estimates are established in \cite{Bam20a}, which extends Perelman's volume non-collapsing estimates.

\begin{lemma} \label{lvnc}
Let $(M, g(t))_{t\in [-r^2, 0]}$ be a solution of the Ricci flow. If $R\leq r^{-2}, ~ on~ B_{g(0)}(x, r)$, then we have
\begin{equation}
\vol_{g(0)} (B_{g(0)}(x, r)) \geq c \exp\left(\cN^*_{-r^2}(x, 0)\right) r^n.
\end{equation}

\end{lemma}

The assumption on the scalar curvature upper bound can be replaced using the $H_n$-center as proved in \cite{Bam20a}.

\begin{lemma} \label{hcv}
Let $(M, g(t))_{t\in [-r^2, 0]}$ be a solution of the Ricci flow. Suppose $(z, -r^2)$ is an $H_n$-center of $(x_0, 0)$ and
$$R(\cdot, -r^2) \geq R_{min},$$
for some fixed $R_{min} \in \mathbb{R}$. Then there exists $c= c(R_{min} r^2)>0$ such that
\begin{equation}
\vol_{g(-r^2)}(B_{g(-r^2)}(z, (2H_n)^{1/2} r)) \geq c \exp\left(\cN^*_{-r^2}(x, 0)\right) r^n.
\end{equation}

\end{lemma}

Next, we have the following heat kernel upper bound estimate, which is proved in \cite[Theorem 7.2]{Bam20a}.
\begin{lemma}
Let $(M, g(t))_{t\in I}$ be a solution of the Ricci flow. Suppose that on $M \times [s, t]$,
$$[s, t]\subset I, ~ R\geq R_{min}.$$
Let $(z, s) \in M\times I$ be an $H_n$-center of $(x, t)\in M\times I$. Then there exist $C=C(R_{min}(t-s))<\infty$, such that for any $y\in M$, we have
\begin{equation}
K(x, t; y,s) \leq C (t-s)^{-n/2} \exp\left( -\cN^*_{s}(x, t) \right) \exp\left( - \frac{ d^2_s(z, y)}{C (t-s)} \right) .
\end{equation}

\end{lemma}

Using this heat kernel upper bound estimate and Perelman's Harnack inequality, we have the following estimate which relates the $W_1$-distance to the $\mathcal{L}$-length, see \cite[Lemma 21.2]{Bam20c}.
\begin{lemma} \label{wdn}
Let $(M, g(t))_{t\in (-T, 0)}$ be a solution of the Ricci flow for some $T>0$. Suppose $(s, t) \subset (-T, 0)$ with $s>-T+\epsilon>0$ for some $\epsilon>0$.  Let  $\gamma: [0, t-s] \rightarrow M \times (-T, 0)$ be a $C^1$ spacetime curve with
$$\gamma(\tau) \in M \times \{ t-\tau\}, ~~\gamma(0)=x, ~\gamma(t-s) = y. $$
Then there exists $C=C(\epsilon)>0$ such that
\begin{equation}
d^{g(s)}_{W_1} (\delta_{y, s}, \nu_{x,t; s}) \leq C \left(1+ \frac{\cL(\gamma)}{2(t-s)^{1/2}} -\cN^*_s(x, t)  \right) ^{1/2} (t-s)^{1/2}.
\end{equation}
\end{lemma}

Next, we have the following Lemma.
\begin{lemma}\label{W1balongwl}
Let $(M, g(t))_{t\in I}$ be a smooth Ricci flow on a compact $n$-dimensional manifold with the interval $I\subset \mathbb{R}$. Assume that $(x_0, t_0)\in M\times I$, $r_0\leq 1$ satisfy that $[t_0-2r_0^2, t_0]\subset I$, $R(x_0, t)\leq Yr_0^{-2}$ for all $t\in [t_0-r_0^2, t_0]$, and $\cN^*_{t_0-r_0^2}(x_0, t_0)\geq -Y$, then we have
$$d^{g(t_0-r_0^2)}_{W_1} (\nu_{x_0, t_0; t_0-r_0^2} , \delta_{x_0} )\leq Cr_0$$
for some constant $C=C(n, Y)<\infty$.
\end{lemma}

Finally, let us recall the following result, which was proved by Perelman in \cite{Per1}.

\begin{lemma} \label{lcenter}
Let $(M, g(t))_{t\in (-T, 0)}$ be a solution of the Ricci flow for some $T>0$. Suppose $[s, t] \subset (-T, 0)$. Then for any $x\in M$, there exists a point $y\in M$, such that
$$ \ell_{(x, t)} (y, s) \leq \frac{n}{2}.$$
\end{lemma}

We will call the point $(y,s)$ is an $\ell_n$-center of $(x,t)$ in Lemma \ref{lcenter}.


\medskip
\subsection{Metric flows and the $\bF$-convergence}
Let $(X, d)$ be a complete, separable metric space and denote by $\mathcal{B}(X)$ the Borel algebra generated by the open subsets of $X$. A probability measure on X is a measure $\mu$ on $\mathcal{B}(X)$ with $\mu(X)=1$. We denote by $\mathcal{P}(X)$ the set of probability measures on $X$. Denote by $\Phi:\mathbb{R}\to (0,1)$ the antiderivative satisfies that $\Phi'(x)=(4\pi)^{-1/2}e^{-x^2/4}$, $\lim_{x\to-\infty}\Phi(x)=0$, $\lim_{x\to\infty}\Phi(x)=1$.

\begin{definition}[Metric Flow Pairs, Definitions 3.2, 5.1 in \cite{Bam20b}]\label{mfpairs}
A metric flow over $I\subseteq \mathbb{R}$ is a tuple
$$(\mathcal{X},\mathfrak{t},(d_t)_{t\in I},(\nu_{x;s})_{x\in \mathcal{X},s \in I\cap (-\infty,\mathfrak{t}(x)]}),$$
where $\mathcal{X}$ is a set, $\mathfrak{t}:\mathcal{X}\to I$ is a function, $d_t$ are metrics on the level sets $\mathcal{X}_t:=\mathfrak{t}^{-1}(t)$, such that $(\mathcal{X}_t, d_t)$ is a complete and separable metric space for all $t$, and $\nu_{x;s}\in \mathcal{P}(\mathcal{X}_s)$, $s\leq \mathfrak{t}(x)$ are such that $\nu_{x;\mathfrak{t}(x)}=\delta_x$ and the following hold:
\begin{enumerate}
    \item (Gradient estimate for heat flows) For $s,t\in I$, $s<t$, $T\geq 0$, if $u_s:\mathcal{X}_s\to [0,1]$ is such that $\Phi^{-1}\circ u_s$ is $T^{-\frac{1}{2}}$-Lipschitz (or just measurable if $T=0$), then either $u_t:\mathcal{X}_t\to [0,1]$, $x\mapsto \int_{\mathcal{X}_s}u_s d\nu_{x;s}$, is constant or $\Phi^{-1}\circ u_t$ is $(T+t-s)^{-\frac{1}{2}}$-Lipschitz,
    \item (Reproduction formula) For $t_1 \leq t_2 \leq t_3$ in $I$, $\nu_{x;t_1}(E)=\int_{\mathcal{X}_{t_2}}\nu_{y;t_1}(E)d\nu_{x;t_2}(y)$ for $x\in \mathcal{X}_{t_3}$ and all Borel sets $E\subseteq \mathcal{X}_{t_1}$.
\end{enumerate}

A conjugate heat flow on $\mathcal{X}$ is a family $\mu_t \in \mathcal{P}(\mathcal{X}_t)$, $t\in I'$, such that for $s\leq t$ in $I'$, we have $\mu_s(E)=\int_{\mathcal{X}_t} \nu_{x;s}(E) d\mu_t(x)$ for any Borel subset $E\subseteq \mathcal{X}_s$. A metric flow pair $(\mathcal{X},(\mu_t)_{t\in I'})$ consists of a metric flow $\mathcal{X}$, along with a conjugate heat flow $(\mu_t)_{t\in I'}$ such that $\text{supp}(\mu_t)=\mathcal{X}_t$ and $|I\setminus I'|=0$.
\end{definition}
Next, we have the following definitions.
\begin{definition}[Correspondences and $\mathbb{F}$-Distance, Definitions 5.4, 5.6 in \cite{Bam20b}]\label{corrspandFcon}
Given metric flows $(\mathcal{X}^i)_{i\in \mathcal{I}}$ defined over $I'^{,i}$, a correspondence over $I''\subseteq \mathbb{R}$ is a pair
$$
\mathfrak{C}=\left( (Z_t,d_t)_{t\in I''},(\varphi_t^i)_{t\in I''^{,i},i\in \mathcal{I}}\right)
$$
where $(Z_t,d_t^Z)$ are metric spaces, $I''^{,i}\subseteq I'^{,i}\cap I''$, and $\varphi_t^i:(\mathcal{X}_t^i,d_t^i)\to (Z_t,d_t^Z)$ are isometric embeddings.

The $\mathbb{F}$-distance between metric flow pairs $(\mathcal{X}^j,(\mu_t^j)_{t\in I'^{,j}})$, $j=1,2$, within $\mathfrak{C}$ is the infimum of $r>0$ such that there exists a measurable set $E\subseteq I''$ such that $I''\setminus E\subseteq I''^{,1}\cap I''^{,2}$, $|E|\leq r^2$, and there exist couplings $q_t$ of $(\mu_t^1,\mu_t^2)$, $t\in I''\setminus E$, such that for all $s,t\in I''\setminus E$ with $s\leq t$, we have
$$\int_{\mathcal{X}_t^1 \times \mathcal{X}_t^2}d_{W_1}^{Z_s}\left( (\varphi_s^1)_{\ast}\nu_{x^1;s}^1 , (\varphi_s^2)_{\ast}\nu_{x^2;s}^2 \right) dq_t(x^1,x^2) \leq r.$$
The $\mathbb{F}$-distance between metric flow pairs is the infimum of $\mathbb{F}$-distances within a correspondence $\mathfrak{C}$, where $\mathfrak{C}$ is varied among all correspondences.
\end{definition}
%


\bigskip
\section{Complex Monge-Amp\`ere flows and basic estimates}\label{setupofKRF}
In this section, we will first build the basic set-ups, and recall the well-known parabolic Schwarz Lemma, which is of basic importance in our theory. Then we will define the so-called localized Ricci potentials, and prove their basic estimates.


\medskip
\subsection{Basic set-ups and the parabolic Schwarz Lemma}\label{bsuandpsl}

We use the set-up from Section \ref{be}, that is, we start with the unnormalized K\"aher-Ricci flow
\begin{equation}\label{unkrflow2}
\left\{
\begin{array}{l}
{ \displaystyle \ddt{\omega(t)} = -\ric(\omega(t)) ,}\\
\\
\omega(0)=\omega_0,
\end{array} \right.
\end{equation}
on a projective manifold $X$ of complex dimension $n\geq 2$ for some initial K\"ahler metric $\omega_0 \in H^{1,1}(X, \mathbb{R})\cap H^2(X, \mathbb{Q})$, which develops singularities at time $T=1$. The limiting cohomology class $\vartheta= [\omega_0]+ [K_X] \in H^{1,1}(X, \mathbb{R})\cap H^2(X, \mathbb{Q})$ is a semi-ample $\mathbb{Q}$-line bundle, induces a unique surjective holomorphic map
\begin{equation}
\Phi: X \rightarrow Y \subset \mathbb{CP}^N ,
\end{equation}
where $Y$ is a normal projective variety and $\dim Y$ is equal to the Kodaira dimension of $\vartheta$. We define $Y_{\textnormal{sing}}$ to be the critical values of $\Phi$ and let
\begin{equation}\label{singset'2}
Y^{\circ} = Y \setminus Y_{\textnormal{sing}}, ~~ X^\circ = \Phi^{-1}(Y^{\circ}), ~~ X_{\textnormal{sing}} = \Phi^{-1}(Y_{\textnormal{sing}}).
\end{equation}

We then consider the following normalized K\"ahler-Ricci flow
\begin{equation}\label{nkrflow2}
\left\{
\begin{array}{l}
{ \displaystyle \dds{\tilde{\omega}(s)} = -\ric(\tilde{\omega}(s)) + \tilde{\omega}(s),}\\
\\
\tilde{\omega}(0) =\omega_0 ,
\end{array} \right.
\end{equation}
which has a long-time solution with $s\in [0, \infty)$. The relations between the unnormalized K\"ahler-Ricci flow (\ref{unkrflow2}) and normalized K\"ahler-Ricci flow (\ref{nkrflow2}) are given by
\begin{equation}\label{rb unkrf and nkrf}
s=-\ln (1-t),~~ t=1-e^{-s},~~ \tilde{\omega}(s)=(1-t)^{-1}\omega(t),~~t\in [0,1).
\end{equation}

The K\"ahler class for the normalized flow is given by
$$[\tilde{\omega}(s)] = e^{s} [\omega_0] + (e^{s} - 1) [K_X] = - [K_X] + e^{s} ([\omega_0] + [K_X]) $$
and $[\omega_0] + [K_X]$ is exactly the limiting K\"ahler class of the unnomarlized Kahler-Ricci flow (\ref{unkrflow2}) as well as the pullback of an ample divisor on $Y$.
%


%
We can always find a smooth closed $(1,1)$-form $\chi \in -[K_X]$ such that
\begin{equation}\label{omegaY}
\omega_Y= \omega_0 - \chi     ,
\end{equation}
is the restriction of the Fubini-Study metric $\omega_{\fs}$ on $\mathbb{CP}^N$ to $Y$. We can also choose a smooth volume form $\Omega$ such that
$$-\ddbar \log \Omega = \chi$$
since $-\chi \in [K_X]$.

Now the unnormalized flow (\ref{unkrflow2}) can be reduced to the complex Monge-Amp\`ere flow as below
\begin{equation}\label{mauflow2}
\left\{
\begin{array}{l}
{ \displaystyle \ddt{\phi} = \log \frac{ ((1-t)\omega_0 + t\omega_Y + \ddbar \phi )^n }{\Omega}  ,~~ t\in [0, 1), }\\
\\
\phi|_{ t=0} =0
\end{array} \right.
\end{equation}
and the nomarlized flow (\ref{nkrflow2}) can be reduced the complex Monge-Amp\`ere flow as below
\begin{equation}\label{maflow2}
\left\{
\begin{array}{l}
{ \displaystyle \dds{\varphi} = \log \frac{ (\omega_0 + (e^{s}-1) \omega_Y + \ddbar \varphi )^n }{\Omega} + \varphi, ~~ s\in [0, \infty), }\\
\\
\varphi|_{ s=0} = 0 .
\end{array} \right.
\end{equation}
The relation between $\phi$ and $\varphi$ is given by
\begin{equation}\label{rbphivarphi}
\varphi(s) = e^s\phi(t(s)) + n ( e^s -s -1), ~~ t(s) = 1 - e^{-s}.
\end{equation}
We have the following well-known parabolic Schwarz lemma.
\begin{lemma}[\bf{Parabolic Schwarz Lemma}] \label{pschwarz}
Let $\beta$ be any K\"ahler metric on $\mathbb{CP}^N$. For the solution to the unnormalized flow $\omega(t)$, we have
\begin{equation} \label{pschwarz'1}
\tr_{\omega(t)}\beta \leq C,
\end{equation}
and
\begin{equation} \label{pschwarz'2}
\left( \ddt{} - \Delta_{\omega(t)} \right) \tr_{\omega(t)}\beta \leq -C^{-1}|\nabla \tr_{\omega(t)}\beta|_{\omega(t)}^2+C,
\end{equation}
on $X\times [0, 1)$.

For the solution to the normalized flow $\tilde{\omega}(s)$, we have
\begin{equation} \label{pschwarz'3}
\tr_{\tilde{\omega}(s)}(e^s\beta) \leq C,
\end{equation}
and
\begin{equation} \label{pschwarz'4}
\left( \dds{} - \Delta_{\tilde{\omega}(s)} \right) \tr_{\tilde{\omega}(s)}(e^s\beta) \leq -C^{-1}|\nabla \tr_{\tilde{\omega}(s)}(e^s\beta)|_{\tilde{\omega}(s)}^2+Ce^{-s},
\end{equation}
on $X\times [0, \infty)$. Here $C<\infty$ is a constant, depends on $n, \omega_0$ and the upper bound for the bisectional curvature of $\beta$.
\end{lemma}
\begin{proof}[\bf{Proof}]
The parabolic Schwarz lemma for the solution to the unnormalized flow $\omega(t)$ is well-known. Then (\ref{pschwarz'3}) and (\ref{pschwarz'4}) comes from the rescaling of (\ref{pschwarz'1}) and (\ref{pschwarz'2}), respectively. Indeed, recall the rescaling relations (\ref{rb unkrf and nkrf}), then we can compute
\begin{eqnarray*}
\left( \dds{} - \Delta_{\tilde\omega(s)} \right) \tr_{\tilde\omega(s)}(e^s\beta) & = & e^{-s}\left( \ddt{} - \Delta_{\omega(t)} \right) \tr_{\omega(t)}\beta \\
&\leq& e^{-s}\left( -C^{-1}|\nabla \tr_{\omega(t)}\beta|_{\omega(t)}^2+C \right)\\
&=&-C^{-1}|\nabla \tr_{\tilde\omega(s)}(e^s\beta)|_{\tilde\omega(s)}^2+Ce^{-s}.
\end{eqnarray*}
This completes the proof.
\end{proof}


\medskip
\subsection{Localized Ricci potentials and their basic estimates}\label{lrpandbe}
In this subsection, we focus on the normalized K\"ahler-Ricci flow (\ref{nkrflow2}).

All the operators $\nabla,\Delta,\left\langle,\right\rangle$ in this subsection are with respect to the solution of the normalized flow $\tilde\omega(s)$. We denote by $\tilde g(s)$ the Riemannian metric associated to $\tilde\omega(s)$, and $ R_{\tilde g}(s)$ the scalar curvature of $\tilde\omega(s)$.

%

We define
\begin{equation} \label{defofu0}
u_0(t) =\ddt{\phi} + \frac{ \phi + nt }{ 1-t },
\end{equation}
on the unnormalized flow $X\times [0 , 1)$. By using the relation between $\phi$ and $\varphi$ in (\ref{rbphivarphi}), we have
$$  u_0 (t(s)) = \dds{\varphi} (s), ~~ t(s) = 1 - e^{-s}. $$
For convenience, we still denote by $u_0(s)$ the function $u_0 (t(s))$, which is a function on the normalized flow $X\times [0, \infty)$, then $u_0$ satisfies the following coupled equations
\begin{equation}\label{ceofu0}
\left\{
\begin{array}{l}
\dds {} u_0 = \Delta u_0 + \tr_{\tilde\omega(s)}(e^s\omega_Y) + u_0=n-R_{\tilde g}(s) + u_0 , \\
\\
\ric (\tilde\omega(s)) = \tilde\omega(s) - e^s\omega_Y - \ddbar u_0 .
\end{array} \right.
\end{equation}
Now, let $\theta_Y$ be a smooth closed $(1,1)$-form on $Y$ with $\Phi^*\theta_Y\in \vartheta$. Then we have
\begin{equation} \label{defofrho}
\omega_Y-\theta_Y=\ddbar \rho ,
\end{equation}
where $\rho$ is a smooth function on $\mathbb{CP}^N$. We still denote by $\rho$ the pullback function $\pi\circ\rho$.
%
%
Then in the normalized flow, we define
\begin{equation} \label{defofu'2}
u = u_0 + e^s\rho,
\end{equation}
on $X \times [0, \infty)$. In our later proofs, we will see that the dependence of our constants on $\theta_Y$ coincides that on $\|\rho \|_{C^4(\omega_{\fs})}$. When we are in the unnormalized flow, we still denote by $u(t)$ the function $u (s(t))$ with $s(t)=-\log(1-t)$, and we can check that
$$\ric(\omega(t)) - (1-t)^{-1} \omega(t) = - (1-t)^{-1} \theta_Y  - \ddbar  u,$$
for all $t\in [0, 1)$. Denote by
\begin{equation} \label{defofalp}
\alpha = \ddbar\rho.
\end{equation}
Of course, we can view $\alpha$ as a smooth form on $\mathbb{CP}^N$. Now $u$ is a smooth function, satisfying the following coupled equations
\begin{equation}\label{ceofu}
\left\{
\begin{array}{l}
\dds {} u = \Delta u  -  \tr_{\tilde\omega(s)}(e^s(\alpha-\omega_Y)) + u = n- R_{\tilde g}(s) + u , \\
\\
\ric (\tilde\omega(s)) = \tilde\omega(s) + e^s(\alpha-\omega_Y) - \ddbar u ,
\end{array} \right.
\end{equation}
on $X \times [0, \infty)$. For $s\in [0, \infty)$, we denote by
$$a(s) = a(t(s)) = \inf_{X} u ( \cdot , s) .$$
%

%

%
We certainly can require that the Ricci vertex $p_s $ associated to $\theta_Y$ at time $s$ varies continuously in $s$, but not smoothly in general. Hence in general $a(s)$ is only continuous, but not smooth. We have the following important monotonicity estimates.

\begin{lemma} \label{a'2}
There exists constant $B_0=B_0(n , \omega_0, \|\rho \|_{C^2(\omega_{\fs})})<\infty$, such that the following statements hold.
\begin{enumerate}
\item The function $e^{-s}\left( a - B_0 \right)$ is monotonicitly increasing in $s\in [0, \infty)$.
\item The function $e^{-s}\left( a + B_0 \right)$ is monotonicitly decreasing in $s\in [0, \infty)$.
\end{enumerate}
Here the norm $\|\rho \|_{C^2(\omega_{\fs})}$ is taken in the space $\mathbb{CP}^N$.
\end{lemma}
\begin{proof}[\bf{Proof}]
Throughout the proof, all the constants will depend at most on $n, \omega_0$ and $\|\rho \|_{C^2(\omega_{\fs})}$.
First, by the parabolic Schwarz lemma and (\ref{ceofu}), we have
\begin{eqnarray*}
\dds{} (e^{-s}u) & = & e^{-s}\dds{}u - e^{-s}u \\
&=& e^{-s}\left( \Delta u  -  \tr_{\tilde\omega(s)}(e^s(\alpha-\omega_Y)) + u \right) - e^{-s}u\\
&\geq& e^{-s}\left( \Delta u - C \right) \\
&=& \Delta (e^{-s}u) - Ce^{-s},
\end{eqnarray*}
hence we obtain
\begin{equation}\label{inofa}
\dds{} \left( e^{-s}( u - C ) \right) \geq \Delta \left( e^{-s}( u - C ) \right) .
\end{equation}
Since $a(s) = \inf_{X} u ( \cdot , s)$, we have
$$
e^{-s}( a(s) - C ) = \inf_{X} e^{-s}( u - C ),
$$
and this infimum is still attained at the Ricci vertex $p_s$, we obtain from (\ref{inofa}) and the maximum principle that $e^{-s}\left( a(s) - C \right)$ is monotonicitly increasing in $s$. This proves the item (1).
Next, again by (\ref{ceofu}), we have
\begin{eqnarray*}
\dds{} (e^{-s}u) & = & e^{-s}\dds{}u - e^{-s}u \\
&=& e^{-s}\left( n- R_{\tilde g}(s) + u \right) - e^{-s}u \leq Ce^{-s},
\end{eqnarray*}
since we have uniform lower bound on the scalar curvature. Hence we obtain $\partial_s(e^{-s}( u + C ))\leq 0$, and this proves item (2).
\end{proof}
Next, we need the following lemma. Let $s_0\in (0, \infty)$ be any given time.
\begin{lemma} \label{a'3}
For any constant $T\geq 0$, for $s\in [0, s_0+T]$ we have
\begin{equation} \label{a'4}
e^{s-s_0}a(s_0)-B\leq a(s),
\end{equation}
for some constant $B=B(n , \omega_0, \|\rho \|_{C^2(\omega_{\fs})}, T)<\infty$.
\end{lemma}
\begin{proof}[\bf{Proof}]
First, according to Lemma \ref{a'2}, the function $e^{-s}\left( a(s) - B_0 \right)$ is monotonicitly increasing in $s\in [0, \infty)$. Hence for $s\in [s_0, s_0+T]$, we have
$$
e^{-s}\left( a(s) - B_0 \right)\geq e^{-s_0}\left( a(s_0) - B_0 \right),
$$
hence we have
$$
a(s) \geq e^{s-s_0}\left( a(s_0) - B_0 \right)\geq e^{s-s_0}a(s_0) - B_0e^T .
$$
where $B_0$ is the constant from Lemma \ref{a'2}.

Next, again by Lemma \ref{a'2}, the function $e^{-s}\left( a(s) + B_0 \right)$ is monotonicitly decreasing in $s\in [0, \infty)$. Hence for $s\in [0, s_0]$, we have
$$
e^{-s}\left( a(s) + B_0 \right)\geq e^{-s_0}\left( a(s_0) + B_0 \right),
$$
hence we have
$$
a(s) \geq e^{s-s_0}\left( a(s_0) + B_0 \right) - B_0 \geq e^{s-s_0}a(s_0) - B_0,
$$
for $s\in [0, s_0]$. The lemma is proved.
\end{proof}
Finally, we show that with a first choice of $\rho$, we can locate the Ricci vertex in any fixed domain which is a pullback of an open subset on $\mathbb{CP}^N$. For any given $q\in Y$, we let $Z_0, ..., Z_N$ be the homogeneous coordinates of $\mathbb{CP}^N$, such that if we denote by $z_j =Z_j/Z_0$, then $z_j=0$ at $q$ for $j=1, ..., N$. Then we let $\eta:\mathbb{R} \to \mathbb{R}$ be a smooth increasing function, such that $\eta(x)=x$ for $x\in (-\infty, 1]$ and $\eta(x)\equiv 2$ for $x\in [10, \infty)$.

We then choose $\rho$ to be $\rho_q =A \eta(|z|^2)$, that is, $u=u_0+e^s\rho_q $, hence $\theta_Y=\omega_Y-\ddbar\rho$ correspondingly. Then we have

\begin{lemma} \label{a'1}
Given any $r\in (0, 1]$, there exists constant $A_0=A_0(n , \omega_0, r)<\infty$, such that whenever $A\geq A_0$, then for any $s\in [0, \infty)$, $u ( \cdot , s)$ must achieve its minimum on the region $\pi^{-1}(\left\{|z|<r\right\})$.

As a consequence, for any $s\in [0, \infty)$, if $p$ is the Ricci vertex associated to $\theta_Y=\omega_Y-\ddbar\rho$ at time $t=1-e^{-s}$, then we have $p\in \pi^{-1}(\left\{|z|<r\right\})$.
\end{lemma}

\begin{proof}[\bf{Proof}]
On the one hand, from (\ref{ceofu0}) and the maximum principle, we have
$$|u_0|\leq C_0e^s, ~~ on ~~ X\times [0, \infty) , $$
for some constant $C_0=C_0(n, \omega_0)<\infty$. On $X_q=\pi^{-1}(q)$, we have $\rho_q|_{X_q}=0$, hence
$$u ( \cdot , s)|_{X_q}\leq C_0e^s. $$
On the other hand, on the region $\pi^{-1}(\left\{|z|\geq r\right\})$, we have $\rho_q\geq Ar^2$, hence
$$u ( \cdot , s) = u_0 ( \cdot , s) + e^sAr^2 \geq 50C_0e^s , $$
if we choose $A_0=A_0(n , \omega_0, r)$ large enough. This proves the lemma.
\end{proof}
%


\bigskip
\section{Gradient and Laplacian estimates for the weighted Ricci potentials}\label{gandle}
In this section, we are under the set-up of Section \ref{setupofKRF}. We focus on the normalized K\"ahler-Ricci flow (\ref{nkrflow2}) and it's corresponding complex Monge-Amp\`ere flow (\ref{maflow2}).

All the operators $\nabla,\Delta,\left\langle,\right\rangle, |\cdot|$ in this subsection are with respect to the solution of the normalized flow $\tilde\omega(s)$.


\medskip
\subsection{Perturbed gradient and Laplacian estimates}\label{perganle}
We have defined the Ricci potential $u$ in (\ref{defofu'2}). We now compute the evolution equation of its gradient and Laplacian.

\begin{lemma} \label{eogandlu}

On $X\times [0, \infty) $, we have
\begin{equation} \label{eogu}
\begin{split}
\left( \frac{\partial}{\partial s}-\Delta \right) |\nabla   u|^2 =  & |\nabla  u|^2 - |\nabla\nabla  u|^2 - |\nabla\overline{\nabla}  u|^2 \\
&  - 2Re\left\{ \nabla \tr_{\tilde\omega(s)}(e^s(\alpha-\omega_Y))\cdot\overline{\nabla} u\right\},
\end{split}
\end{equation}
and if we denote by $K_0:= - \Delta u + \tr_{\tilde\omega(s)}(e^s(\alpha-\omega_Y))=  R_{\tilde g}(s) - n $, then
\begin{equation} \label{eolu}
\begin{split}
\left( \frac{\partial}{\partial s}-\Delta \right) K_0 = K_0 + |\nabla\overline{\nabla} u|^2 - 2 \left\langle e^s(\alpha-\omega_Y) , \ddbar u\right\rangle + |e^s(\alpha-\omega_Y)|^2.
%
\end{split}
\end{equation}
\end{lemma}

\begin{proof}[\bf{Proof}]

From the equation (\ref{ceofu}), we can first compute
\[
\begin{split}
&\left( \frac{\partial}{\partial s}-\Delta \right) |\nabla   u|^2 =  -|\nabla  u|^2 - |\nabla\nabla  u|^2 - |\nabla\overline{\nabla}  u|^2 +2Re\left\{ \nabla (\partial_s-\Delta) u \cdot \overline{\nabla} u \right\} \\
& =   -|\nabla  u|^2 - |\nabla\nabla  u|^2 - |\nabla\overline{\nabla}  u|^2 +2Re\left\{ \nabla \left( u - \tr_{\tilde\omega(s)}(e^s(\alpha-\omega_Y)) \right) \cdot\overline{\nabla} u\right\} \\
& =  |\nabla  u|^2 - |\nabla\nabla  u|^2 - |\nabla\overline{\nabla}  u|^2 - 2Re\left\{ \nabla \tr_{\tilde\omega(s)}(e^s(\alpha-\omega_Y))\cdot\overline{\nabla} u\right\}. \\
\end{split}
\]
This proves (\ref{eogu}). Next, by the second equation in (\ref{ceofu}), we can compute
\[
\begin{split}
&\left( \frac{\partial}{\partial s} - \Delta \right) K_0 = \left( \frac{\partial}{\partial s} - \Delta \right) R_{\tilde g}(s) \\
& = | \ric (\tilde\omega(s)) |^2 - R_{\tilde g}(s)  \\
& = \left\langle \ric (\tilde\omega(s)) ,  e^s(\alpha-\omega_Y) - \ddbar  u \right\rangle  \\
& = \left\langle \tilde\omega(s) + e^s(\alpha-\omega_Y) - \ddbar u , e^s(\alpha-\omega_Y) - \ddbar  u \right\rangle  \\
& = K_0 + |\nabla\overline{\nabla} u|^2 - 2 \left\langle e^s(\alpha-\omega_Y) , \ddbar u\right\rangle + |e^s(\alpha-\omega_Y)|^2,
\end{split}
\]
this completes the proof.
\end{proof}
Then we have the following gradient and Laplacian estimates for the localized Ricci potentials $u$, which are based at any given time $s_0\in (0, \infty)$.

\begin{proposition} \label{gandleou}
For any constants $B_1, B_2<\infty$, there exists constant $C<\infty$ depending on $n, \omega_0, \|\rho \|_{C^4(\omega_{\fs})}, B_1, B_2$, such that the following statement holds.

Given any time $s_0>0$, assume $b: [ 0, s_0 ]\to \mathbb{R}$ is a $C^1$ is a function satisfies
\begin{enumerate}
\item $b(s)\leq a(s)$ for all $s\in [0, s_0]$;
\item $b'(s)\leq b(s) + B_1 $ for all $s\in [0, s_0]$;
\item $|b(s)|\leq B_2e^s $ for all $s\in [0, s_0]$.
\end{enumerate}
Then we have the gradient estimate
\begin{equation} \label{geou}
\frac{ |\nabla u|^2}{ u - b + 1 } \leq C,
\end{equation}
and the Laplacian estimate
\begin{equation} \label{leou}
\frac{ \left| \Delta u \right|}{ u - b + 1 } \leq C,
\end{equation}
on $X\times [0, s_0]$. Here the norm $\|\rho \|_{C^4(\omega_{\fs})}$ is taken in the space $\mathbb{CP}^N$.
\end{proposition}
\begin{proof}[\bf{Proof}]
Throughout the proof, all the constants will depend at most on $n, \omega_0$ and $\|\rho \|_{C^4(\omega_{\fs})}$, $B_1$, $B_2$. All the computations in this proof are on $X\times [0, s_0] $.
Since $b(s)\leq a(s)$ for all $s\in [0, s_0]$, we have
$$
u - b + 1 \geq u - a + 1 \geq 1,
$$
on $X\times [0, s_0] $. Hence the LHS terms of (\ref{geou}) and (\ref{leou}) are well-defined.
\medskip

\noindent {\bf Gradient estimate:} First, we prove (\ref{geou}). By equations (\ref{ceofu}) and (\ref{eogu}), we can compute
\begin{equation} \label{eogu'2}
\begin{split}
& \left( \dds{} - \Delta \right)  \frac{ |\nabla u|^2 }{ u - b + 1 } \\
& = \frac{\left( \partial_s - \Delta \right)|\nabla u|^2}{ u - b + 1 }  - \frac{ 2 |\nabla u|^4 }{( u - b + 1)^3} + \frac{ 2 Re\left\{ \nabla |\nabla u|^2\cdot\overline{\nabla} u\right\}}{( u - b + 1 )^2} \\
&\qquad\qquad- \frac{ |\nabla u|^2 }{ ( u - b + 1 )^2 } \left( \partial_s - \Delta \right) ( u - b ) \\
& = \frac{|\nabla  u|^2 - |\nabla\nabla  u|^2 - |\nabla\overline{\nabla}  u|^2 - 2Re\left\{ \nabla \tr_{\tilde\omega(s)}(e^s(\alpha-\omega_Y))\cdot\overline{\nabla} u\right\}}{ u - b + 1 } - \frac{ 2 |\nabla u|^4 }{( u - b + 1)^3} \\
&\quad + \frac{ 2 Re\left\{ \nabla |\nabla u|^2\cdot\overline{\nabla} u\right\}}{( u - b + 1 )^2} - \frac{ |\nabla u|^2 }{ ( u -b + 1 )^2 } \left(  -  \tr_{\tilde\omega(s)}(e^s(\alpha-\omega_Y)) + u - b'  \right) \\
& = - \frac{|\nabla\nabla  u|^2 + |\nabla\overline{\nabla}  u|^2 }{ u - b + 1 } - \frac{ 2 |\nabla u|^4 }{( u - b + 1)^3} - \frac{2Re\left\{ \nabla \tr_{\tilde\omega(s)}(e^s(\alpha-\omega_Y))\cdot\overline{\nabla} u\right\}}{ u - b + 1 } \\
& \quad + \frac{ 2 Re\left\{ \nabla |\nabla u|^2\cdot\overline{\nabla} u\right\}}{( u - b + 1 )^2} + \frac{ |\nabla u|^2 }{ ( u - b + 1 )^2 } \left( b' -b + 1 +  \tr_{\tilde\omega(s)}(e^s(\alpha-\omega_Y)) \right) .
\end{split}
\end{equation}
In order to apply the parabolic Schwarz lemma, say Lemma \ref{pschwarz}, we find a large constant $A$, such that $\alpha + A\omega_Y\geq 100\omega_Y$, that is, if we define
$$
\beta = \alpha + A\omega_Y,
$$
then $\beta$ is a K\"ahler metric on $\mathbb{CP}^N$, and it's upper bound of the bisectional curvature depends on $\|\rho \|_{C^4(\omega_{\fs})}$. Now we rewrite (\ref{eogu'2}) as
\begin{equation} \label{eogu'3}
\begin{split}
& \left( \dds{} - \Delta \right)  \frac{ |\nabla u|^2 }{ u - b + 1 } \\
& = - \frac{|\nabla\nabla  u|^2 + |\nabla\overline{\nabla}  u|^2 }{ u - b + 1 } - \frac{ 2 |\nabla u|^4 }{( u - b + 1)^3} + \frac{ 2 Re\left\{ \nabla |\nabla u|^2\cdot\overline{\nabla} u\right\}}{( u - b + 1 )^2}  \\
& \quad - \frac{2Re\left\{ \nabla \tr_{\tilde\omega(s)}(e^s\beta)\cdot\overline{\nabla} u\right\}}{ u - b + 1 } + \frac{2(A+1)Re\left\{ \nabla \tr_{\tilde\omega(s)}(e^s\omega_Y)\cdot\overline{\nabla} u\right\}}{ u - b + 1 } \\
& \quad + \frac{ |\nabla u|^2 }{ ( u - b + 1 )^2 } \left( b' -b + 1 +  \tr_{\tilde\omega(s)}(e^s(\alpha-\omega_Y)) \right) .
\end{split}
\end{equation}

Now we set
$$ \cH = \frac{ |\nabla u|^2 }{ u - b + 1 } + \tr_{\tilde\omega(s)}( e^s \beta) + \tr_{\tilde\omega(s)}( e^s \omega_Y).$$
Then we can compute, for a small constant $\ep >0$ to be determined in the course of the proof, we have
\begin{equation} \label{eogu'4}
\begin{split}
& \left( \dds{} - \Delta \right)  \cH \\
&= - \frac{|\nabla \nabla u|^2 + |\nabla\overline{\nabla} u|^2}{ u - b + 1 }  - \frac{ 2\ep  |\nabla u|^4 }{( u - b + 1)^3} + \frac{ 2(1-\ep ) }{ u - b + 1 } Re\left\{ \nabla \cH \cdot \overline{\nabla} u\right\}  \\
&\quad + \frac{ 2\ep  Re\left\{ \nabla |\nabla u|^2\cdot\overline{\nabla} u\right\}}{( u - b + 1)^2} - \frac{ 2 ( 2 - \ep  ) Re\left\{ \nabla \tr_{\tilde\omega(s)}(e^s\beta)\cdot\overline{\nabla}u\right\} }{ u - b + 1 }  \\
&\quad + \frac{ 2 ( A + \ep ) Re\left\{ \nabla \tr_{\tilde\omega(s)}(e^s\omega_Y)\cdot\overline{\nabla}u\right\} }{ u - b + 1 } + \left( \partial_s -\Delta \right) \tr_{\tilde\omega(s)} ( e^s \beta ) \\
&\quad + \left( \partial_s -\Delta \right) \tr_{\tilde\omega(s)} ( e^s \omega_Y ) + \frac{ |\nabla u|^2 }{ ( u - b + 1 )^2 } \left( b' - b + 1 + \tr_{\tilde\omega(s)}(e^s(\alpha-\omega_Y)) \right) .
\end{split}
\end{equation}
The parabolic Schwarz lemma, say Lemma \ref{pschwarz}, implies that there exists constant $C_0<\infty$ such that
\begin{equation} \label{pschwarz'5}
\left( \dds{} - \Delta \right) \tr_{\tilde\omega(s)}(e^s\beta) \leq -C_0^{-1}|\nabla \tr_{\tilde\omega(s)}(e^s\beta)|^2+C_0e^{-s},
\end{equation}
and
\begin{equation} \label{pschwarz'6}
\left( \dds{} - \Delta \right) \tr_{\tilde\omega(s)}(e^s\omega_Y) \leq -C_0^{-1}|\nabla \tr_{\tilde\omega(s)}(e^s\omega_Y)|^2+C_0e^{-s}.
\end{equation}
Then we can estimate
$$
\frac{ 2\ep  Re\left\{ \nabla |\nabla u|^2\cdot\overline{\nabla} u\right\}}{( u - b + 1)^2}\leq  \frac{ \ep  |\nabla u|^4 }{( u - b + 1)^3} + \ep \frac{|\nabla \nabla u|^2 + |\nabla\overline{\nabla} u|^2}{ u - b + 1},
$$
$$
- \frac{ 2 ( 2 - \ep  ) Re\left\{ \nabla \tr_{\tilde\omega(s)}(e^s\beta)\cdot\overline{\nabla}u\right\} }{ u - b + 1 } \leq  (100C_0)^{-1}|\nabla \tr_{\tilde\omega(s)}(e^s\beta)|^2 + \frac{C|\nabla u|^2}{( u - b + 1 )^2},
$$
$$
\frac{ 2 ( A + \ep  ) Re\left\{ \nabla \tr_{\tilde\omega(s)}(e^s\omega_Y)\cdot\overline{\nabla}u\right\} }{ u - b + 1 } \leq  (100C_0)^{-1}|\nabla \tr_{\tilde\omega(s)}(e^s\omega_Y)|^2 + \frac{C|\nabla u|^2}{( u - b + 1 )^2},
$$
and using the assumption that $b'(s)\leq b(s) + B_1 $ for all $s\in [0, s_0]$, we have
$$
\frac{ |\nabla u|^2 }{ ( u - b + 1 )^2 } \left( b' -b + 1 + \tr_{\tilde\omega(s)}(e^s(\alpha-\omega_Y)) \right) \leq \frac{ C |\nabla u|^2}{( u -b + 1 )^2}.
$$
Inserting these estimates and (\ref{pschwarz'5}) and (\ref{pschwarz'6}) into (\ref{eogu'4}), we can conclude, if we choose $\ep $ small enough, then
\begin{equation} \label{eogu'5}
\begin{split}
\left( \dds{} - \Delta \right)  \cH \leq - & \frac{ \ep  |\nabla u|^4 }{( u - b + 1)^3} + \frac{ 2(1-\ep ) }{ u - b + 1 } Re\left\{ \nabla \cH \cdot \overline{\nabla} u\right\} \\
&+  \frac{ C|\nabla u|^2}{( u - b + 1 )^2} + C_0e^{-s} .
\end{split}
\end{equation}
Now assume that $\cH$ achieves its maximum at some point $(\tilde x, \tilde s)$ with $\tilde s \in (0, s_0]$, then at the point $(\tilde x, \tilde s)$, we have $\nabla \cH=0$, hence we obtain from (\ref{eogu'5}) that
$$ 0 \leq \left( \dds {} - \Delta \right) \cH   \leq - \frac{\ep }{4} \frac{ |\nabla u|^4 }{( u - b + 1)^3} + C\frac{|\nabla u|^2}{( u - b + 1 )^2} + Ce^{-s} , $$
at $(\tilde x, \tilde s)$. But we have $|u_0|\leq C_0e^s$ globally on $X\times[0, \infty]$, which implies that
\begin{equation} \label{C0boundonu}
|u|\leq |u_0| + e^s|\rho| \leq Ce^s,
\end{equation}
globally on $X\times[0, \infty]$. Hence from our assumption, we have
$$
| u - b + 1 |\leq |u| + |b| + 1 \leq Ce^s + B_2e^s \leq Ce^s ,
$$
on $X\times[0, s_0]$. Hence we obtain
$$
\ep   \frac{|\nabla u|^4}{ ( u - b + 1 )^2} \leq  C \frac{|\nabla  u|^2}{ u - b + 1 } + Ce^{-s}( u - b + 1) \leq  C \frac{|\nabla u|^2}{ u - b + 1 } + C ,
$$
at $(\tilde x, \tilde s)$, which implies
$$ \frac{|\nabla u|^2}{ u - b + 1 }( \tilde x, \tilde s ) \leq C . $$
Again by applying the parabolic Schwarz lemma, we conclude that $\cH (\tilde x, \tilde s)\leq C$, hence $\cH \leq C$ on $X\times[0, s_0]$, which proves (\ref{geou}).
\medskip

\noindent {\bf Laplacian estimate:} Next, we prove estimate (\ref{leou}). Recall that we denote by $K_0=  - \Delta u + \tr_{\tilde\omega(s)}(e^s(\alpha-\omega_Y)) = R_{\tilde g}(s) - n $, we then define
$$K_1 = \frac{K_0 + C_0}{u-b+1}, $$
where $C_0=C_0(n, \omega_0)<\infty$ is a constant such that $K_0 + C_0\geq 0$ globally. By applying (\ref{eolu}) in Lemma \ref{eogandlu}, we can compute
\begin{equation} \label{eolu'2}
\begin{split}
& \left( \dds{} - \Delta \right)  K_1 \\
& = \frac{  \left( \partial_s - \Delta \right) K_0  }{ u - b + 1 } - \frac{ K_0 + C_0 }{( u - b + 1 )^2} \left( \partial_s - \Delta \right) ( u - b ) + \frac{ 2  Re\left\{ \nabla K_1 \cdot \overline{\nabla} u\right\} }{ u - b + 1 } \\
& = \frac{ |\nabla\overline{\nabla} u|^2 - 2 \left\langle e^s(\alpha-\omega_Y) , \ddbar u\right\rangle + |e^s(\alpha-\omega_Y)|^2 }{ u - b + 1 } + \frac{ 2  Re\left\{ \nabla K_1 \cdot \overline{\nabla} u\right\}  }{ u - b + 1 } \\
&\quad+ \frac{ K_0 + C_0 }{ ( u - b + 1 )^2 } \left( b' - b + 1 + \tr_{\tilde\omega(s)}(e^s(\alpha-\omega_Y)) \right) - \frac{ C_0 }{ u - b + 1 } .
\end{split}
\end{equation}
Now we set
$$\cK = K_1 + 100 \cH ,  $$
then we can combine (\ref{eogu'4}) (with $\epsilon = 0$ there) and (\ref{eolu'2}) to obtain that
\begin{equation} \label{eolu'3}
\begin{split}
& \left( \dds{} - \Delta \right)  \cK \\
& = - 100 \frac{|\nabla \nabla u|^2 + |\nabla\overline{\nabla} u|^2}{ u - b + 1 } + \frac{ 2Re\left\{ \nabla \cK \cdot \overline{\nabla} u\right\} }{ u - b + 1 } - \frac{ C_0 }{ u - b + 1 }  \\
&\quad + \frac{ | \nabla \overline{\nabla} u|^2 - 2 \left\langle e^s(\alpha-\omega_Y) , \ddbar u\right\rangle + |e^s(\alpha-\omega_Y)|^2 }{ u - b + 1 } \\
&\quad - \frac{ 400 Re\left\{ \nabla \tr_{\tilde\omega(s)} (e^s\beta)\cdot\overline{\nabla}u\right\} }{ u - b + 1 } + \frac{ 200 A  Re\left\{ \nabla \tr_{\tilde\omega(s)}(e^s\omega_Y)\cdot\overline{\nabla}u\right\} }{ u - b + 1 }  \\
&\quad + \left( \partial_s -\Delta \right) \tr_{\tilde\omega(s)} ( e^s \beta ) + \left( \partial_s -\Delta \right) \tr_{\tilde\omega(s)} ( e^s \omega_Y ) \\
&\quad + \frac{  K_0 + C_0 + 100 |\nabla u|^2 }{ ( u - b + 1 )^2 } \left( b' - b + 1 + (A-1) \tr_{\tilde\omega(s)}(e^s\omega_Y) \right) .
\end{split}
\end{equation}
First, by the gradient estimate, we can compute
$$
- \frac{ 400 Re\left\{ \nabla \tr_{\tilde\omega(s)}(e^s\beta)\cdot\overline{\nabla}u\right\} }{ u - b + 1 } \leq  (100C_0)^{-1}|\nabla \tr_{\tilde\omega(s)}(e^s\beta)|^2 + C ,
$$
$$
\frac{ 200 A  Re\left\{ \nabla \tr_{\tilde\omega(s)}(e^s\omega_Y)\cdot\overline{\nabla}u\right\} }{ u - b + 1 } \leq  (100C_0)^{-1}|\nabla \tr_{\tilde\omega(s)}(e^s\omega_Y)|^2 + C .
$$
Next, by the parabolic Schwarz lemma, we can estimate
$$
\frac{ - 2 \left\langle e^s(\alpha-\omega_Y) , \ddbar u\right\rangle }{ u - b + 1 }  \leq \frac{|\nabla\overline{\nabla} u|^2}{ u - b + 1 } + C .
$$
Finally, using the assumption that $b'(s)\leq b(s) + B_1 $ for all $s\in [0, s_0]$ and the gradient estimate, the last term in (\ref{eolu'3}) is bounded above by (recall that $K_0 + C_0\geq 0$ globally)
$$
\frac{|\nabla\overline{\nabla} u|^2}{ u - b + 1 } + C.
$$
Inserting these estimates into (\ref{eolu'3}), and using the parabolic Schwarz Lemma (\ref{pschwarz'5}) and (\ref{pschwarz'6}) again, we conclude
\begin{equation} \label{eolu'4}
\begin{split}
& \left( \dds{} - \Delta \right)  \cK \leq - 10 \frac{ |\nabla\overline{\nabla} u|^2}{ u - b + 1 } + \frac{ 2 Re\left\{ \nabla \cK \cdot \overline{\nabla} u\right\} }{ u - b + 1 } + C. \\
\end{split}
\end{equation}
Now, assume that $\cK$ achieves it's maximum at some point $(\tilde x, \tilde s)$ with $\tilde s\in (0, s_0]$, then at the point $(\tilde x, \tilde s)$, we have $\nabla \cK=0$, hence we obtain from (\ref{eolu'4}) that
$$ 0 \leq \left( \dds {} - \Delta \right) \cK  \leq - \frac{ |\nabla\overline{\nabla} u|^2}{ u - b + 1 } + C  , $$
at $(\tilde x, \tilde s)$, which implies that
$$
\frac{ |\Delta u|^2}{ ( u - b + 1 )^2 } \leq C \frac{ |\nabla\overline{\nabla} u|^2}{ u - b + 1 } \leq C ,
$$
at $(\tilde x, \tilde s)$. Again by applying the parabolic Schwarz lemma and , we conclude that $\cK (\tilde x, \tilde s)\leq C$, hence $\cK \leq C$ on $X\times[0, s_0]$.

The upper bound of $\frac{ \Delta u}{ u - b + 1}$ follows from the expression
$$ \frac{ \Delta u}{ u - b + 1} = \frac{ n - R_{\tilde g}( s ) + \tr_{\tilde\omega(s)}(e^s(\alpha-\omega_Y)) }{ u - b + 1 },
$$
combining with the parabolic Schwarz lemma and the uniform lower bound of the scalar curvature, which proves (\ref{leou}).
Therefore Proposition \ref{gandleou} is proved.
\end{proof}


\medskip
\subsection{Proof of Theorem \ref{cor1} and Theorem \ref{main1}}\label{proofofmain1}
We now come to prove Theorem \ref{main1} and Corollary \ref{cor1}.


%

%
%

%

%
First we prove Theorem \ref{main1}.
\begin{proof}[\bf{Proof of Theorem \ref{main1}}]
Throughout the proof, all the constants will depend at most on $n, \omega_0, \|\rho \|_{C^4(\omega_{\fs})}$. The dependence of the constants on $\theta_Y$ is the dependence on $\|\rho \|_{C^4(\omega_{\fs})}$ here.

Let $s_0\in (0, \infty)$ be any given time. Then we set
$$
b(s):=e^{s-s_0}a(s_0)-B, ~~ s\in [0, s_0].
$$
where $B=B(n , \omega_0, \|\rho \|_{C^2(\omega_{\fs})}, 0)<\infty$ is the constant from Lemma \ref{a'3} (with $T=0$ there). Then by Lemma \ref{a'3}, we have $b(s)\leq a(s)$ for $s\in [0, s_0]$. Next, we can compute
$$
b'(s)=e^{s-s_0} a(s_0) = b(s)+ B.
$$
Finally, by (\ref{C0boundonu}), we have $|a(s_0)|\leq Ce^{s_0}$, hence we can compute
$$
|b(s)| \leq e^{s-s_0}|a(s_0)| + B \leq Ce^{s}.
$$
Hence we conclude that $b(s)$ is a smooth function, satisfying the three assumptions in Proposition \ref{gandleou}. Hence we can apply Proposition \ref{gandleou} to obtain that
\begin{equation} \label{geou'2}
\frac{ \left| \Delta u \right|}{ u - b + 1 } + \frac{ |\nabla u|^2}{ u - b + 1 } \leq C,
\end{equation}
on $X\times [0, s_0]$. Evaluate (\ref{geou'2}) at time $s_0$, and note that $b(s_0)=a(s_0)-B$, we conclude
\begin{equation} \label{gleou'3}
\left| \Delta u \right| + |\nabla u|^2 \leq C(u-a+B+1)\leq C(u-a+1) ,
\end{equation}
on $X\times \left\{s_0\right\}$. Translating (\ref{gleou'3}) into the unnormalized flow, we obtain (\ref{gleou'0}), since $v=u-a+1$. Since the time $s_0\in (0, \infty)$ is arbitrarily chosen, this completes the proof of Theorem \ref{main1}.
\end{proof}
We have the following immediate corollary.

\begin{corollary}\label{scbonPnbhd2}
Assume as in Theorem \ref{main1}. For any $D<\infty$, there exists constant $C=C(n , \omega_0, \theta_Y, D)<\infty$, such that
\begin{equation} \label{main2'1}
\sup_{ B\left( p_t , t , D (1-t)^{1/2} \right) \times [t-D(1-t), 1-D^{-1}(1-t) ] }~ R_g \leq \frac{C}{ 1 - t } .
\end{equation}
\end{corollary}
\begin{proof}
In the proof, we compute in the normalized flow. Let $s_0\in (0, \infty)$ be a given time. Let $p_{s_0}=p_{t_0}\in X$ be a Ricci vertex associated to $\theta_Y$, where $t_0=1-e^{-s_0}$.

For any constant $D<\infty$, let $B=B(n , \omega_0, \|\rho \|_{C^2(\omega_{\fs})}, D)$ be the constant from Lemma \ref{a'3}, then we define
\begin{equation}\label{cordob}
b(s) = e^{s-s_0}a(s_0)-B.
\end{equation}
We then denote by
\begin{equation}\label{cordovs0}
v_{s_0}=u-b(s)+1.
\end{equation}
From Lemma \ref{a'3}, we have $b(s) \leq a(s)$ for $s\in [0, s_0+D]$, hence we have $v_{s_0}\geq 1$ for $s\in [0 , s_0+D]$. Now by Theorem \ref{main1}, we have
\begin{equation} \label{gleou'4}
\left| \Delta u \right| + |\nabla u|^2 \leq Cv = C(u-a+1) ,
\end{equation}
on $X\times [0, \infty)$. Hence on $X\times [0, s_0+D]$, we have
\begin{equation} \label{eofdsv}
\begin{split}
&|\partial_s v_{s_0}| = | \partial_s u - b' | = | \Delta u  -  \tr_{\tilde\omega(s)}(e^s(\alpha-\omega_Y)) + u - e^{s-s_0}a(s_0) | \\
& \leq | \Delta u | + v_{s_0} + C \leq C(u-a+1) + v_{s_0} + C \leq Cv_{s_0} ,
\end{split}
\end{equation}
and
\begin{equation} \label{eoflpv}
| \Delta v_{s_0} | + |\nabla v_{s_0}|^2 = | \Delta u | + |\nabla u |^2 \leq C(u-a+1) \leq Cv_{s_0} .
\end{equation}
In conclusion, on $X\times [0, s_0+D]$, we have
\begin{equation} \label{eofdsvandgradientv}
\frac{|\partial_s v_{s_0}|}{v_{s_0}} + \frac{| \Delta v_{s_0} |}{v_{s_0}} + \frac{|\nabla v_{s_0}|^2}{v_{s_0}} \leq C,
\end{equation}
for some constant $C=C(n , \omega_0, \|\rho \|_{C^4(\omega_{\fs})}, D)<\infty$.

Now, for any $x\in X$, integrate the gradient estimate in (\ref{eofdsvandgradientv}) along any $\tilde g(s_0)$-minimizing geodesic connecting $x$ to $p_{s_0}$, by the fact that $ v_{s_0}(p_{s_0}, s_0) = B+1 $, we have
\begin{equation} \label{C0boundofv1}
v_{s_0}(x, s_0) \leq 2v_{s_0}(p_{s_0}, s_0) + Cd_{\tilde g(s_0)}^2 (x, p_{s_0}) \leq C\left(1+d_{\tilde g(s_0)}^2 (x, p_{s_0})\right).
\end{equation}
Hence for any $x\in B_{\tilde g}\left( p_{s_0} , s_0 , D \right)$, $v_{s_0}(x, s_0)\leq C$. Then for any $s\in [s_0-D, s_0+D]$, we can integrate the time derivative estimate in (\ref{eofdsvandgradientv}) to obtain
\begin{equation} \label{C0boundofv2}
v_{s_0}(x, s) \leq CD v_{s_0}(x, s_0) \leq C.
\end{equation}
Hence we conclude that $v_{s_0}\leq C$ on $B_{\tilde g}\left( p_{s_0} , s_0 , D \right)\times [s_0-D, s_0+D]$. Now using the Laplacian estimate in (\ref{eofdsvandgradientv}), by the parabolic Schwarz lemma, we have
\begin{equation} \label{C0boundofv3}
R_{\tilde g}(s) = n -  \Delta v_{s_0} + \tr_{\tilde\omega(s)}(e^s(\alpha-\omega_Y)) \leq Cv_{s_0} + C \leq C ,
\end{equation}
on $B_{\tilde g}\left( p_{s_0} , s_0 , D \right)\times [s_0-D, s_0+D]$. Translating (\ref{C0boundofv3}) into the unnormalized flow, we complete the proof.
\end{proof}
Corollary \ref{scbonPnbhd2} proves the first estimate in Remark \ref{scbonPnbhd}. The second estimate in Remark \ref{scbonPnbhd} is proved in Proposition \ref{viprop3}. The proof of Theorem \ref{cor1} is immediate.

\begin{proof}[Proof of Theorem \ref{cor1}]
This follows immediately from (\ref{C0boundofv1}).
\end{proof}


\bigskip
\section{Harnack estimate on Ricci flow and K\"ahler-Ricci flow} \label{sechar}
In this section, we first prove our Harnack estimate on Ricci flow background, say Theorem \ref{main2}. Then we extend it to our finite time solution of the K\"ahler-Ricci flow in Section \ref{setupofKRF}.


\medskip
\subsection{Harnack estimate on Ricci flow background}

In this subsection, we consider a general Ricci flow $(M,(g(t))_{t\in [0, 1)})$ on compact, $n$-dimensional manifold. We recall our Harnack estimate in the following.

\begin{theorem}[\bf Theorem \ref{main2}]\label{gosce}
Let $(M,(g(t))_{t\in [0, 1)})$ be a Ricci flow on compact Riemannian manifold of $\dim_{\mathbb{R}} M=n$. Let
$$v\in M\times [0, 1)\to \mathbb{R}^+$$
be a positive $C^1$-function.
Given any $\delta>0$, $B<\infty$, there exists  $\bar{\delta} = \bar{\delta}(n, B)>0$ such that whenever $\delta\in (0, \bar{\delta})$,  there exists  $C=C(n, B, \delta)<\infty$ such that the following holds. Suppose
\begin{enumerate}
\item $\nu(g(0), 2) \geq -B; $
\medskip

\item $  |\partial_t \ln v | + |\nabla \ln v|^2  \leq B/(1-t) $ on $M\times [0, 1)$;
\medskip

\item  $t_0\in [1/2, 1)$ and  $V$ is an open set  of $M$ with  $\vol_{g(t_0)}(V)\leq B (1-t_0)^{\frac{n}{2}}$ ;

\medskip

\item $U$ is a connected open subset  of $M$ and  for any $x\in U$, there exists an $H_n$-center $(z, t_0)$ of $(x, t_0 + \delta^2(1-t_0))$ with $B_{g(t_0)} \left( z, \sqrt{2H_n}\delta(1-t_0)^{1/2} \right) \subset V .$
\end{enumerate}
Then we have
\begin{equation}\label{harin2}
\frac{\sup_{U} v ( \cdot, t_0 )}{ \inf_{U} v ( \cdot, t_0 ) }  \leq C.
\end{equation}
\end{theorem}

\begin{proof}[\bf{Proof}]
We will determine $\bar{\delta} = \bar{\delta}(n, B)\in (0, \frac{1}{100})$ in the course of the proof. We assume that $\delta\in (0, \bar{\delta})$.

We denote by
$$ f = \ln v : M\times [0, 1)\to \mathbb{R}, $$
which is a $C^1$-function.
Let $N\geq 0$ be the positive integer such that
\begin{equation} \label{loov}
10N\leq  \underset{U}{\osc}~ f (\cdot, t_0) \leq 10 ( N + 1).
\end{equation}
We only need to obtain an upper bound of $N$.

Denote by $r_0:=(1-t_0)^{1/2}$. Then we set, for each $\ell=1, 2, \dots, N$,
\begin{equation} \label{doAl}
A_{\ell}:=\underset{U}{\inf}~ f (\cdot, t_0) + 10\ell.
\end{equation}
Since we have assumed that $U$ is an open connected subset of $M$, by the continuity of $v$ and (\ref{loov}), for each $\ell=1, 2, ..., N$, there exists $x_{\ell} \in U$ such that
$$ f (x_\ell, t_0) = A_\ell.$$
Then we have, by applying the assumption (2), for each $\ell=1, 2, \dots, N$,
\begin{equation} \label{xl1}
\left|f(x_\ell, t_0 + \delta^2r_0^2) - f(x_\ell, t_0)\right| \leq \delta^2r_0^2\cdot B\left(1-(t_0 + \delta^2r_0^2)\right)^{-1} \leq 2\delta^2 B.
\end{equation}
Hence if we choose $\bar{\delta} = \bar{\delta}(B)$ small enough such that $\bar{\delta}\leq \frac{1}{100A}$, we then have
\begin{equation} \label{vxl'1}
f(x_\ell, t_0 + \delta^2r_0^2)\in [A_\ell-1 , A_\ell + 1].
\end{equation}
for each $\ell=1, 2, \dots, N$.

Now for each $(x_\ell, t_0 + \delta^2r_0^2)$, by Lemma \ref{lcenter}, there exists $y_\ell \in M$ such that
\begin{equation} \label{lcenter'1}
\ell_{(x_\ell , t_0 + \delta^2r_0^2)}( y_\ell , t_0) \leq  \frac{n}{2}.
\end{equation}
Hence by definition, we can find a smooth space-time curve
$$\gamma_\ell: [0, \delta^2r_0^2] \rightarrow M \times [0, 1),$$
which is parametrized by backward time (i.e. $\gamma_\ell(\tau) \in M \times \{ t_0 + \delta^2r_0^2 - \tau \}$), joining $(x_\ell , t_0 + \delta^2r_0^2)$ to $( y_\ell , t_0)$, and satisfies that
\begin{equation} \label{lcenter'2}
\frac{ \mathcal{L}(\gamma_\ell)}{2 \delta r_0} \leq \frac{n}{2}.
\end{equation}
Then we have the following distortion estimates.
\begin{lemma} \label{xy}
If we choose $\bar{\delta}\leq 1$, then there exists constant $C_0=C_0(n)<\infty$ such that
\begin{equation}\label{xy1}
| f (x_\ell , t_0 + \delta^2r_0^2) - f ( y_\ell , t_0 )| \leq C_0\delta B,
\end{equation}
for each $\ell =1, 2, ..., N$.
\end{lemma}
\begin{proof}[\bf{Proof}]
For each $\ell =1, 2, ..., N$, we can compute
\[
\begin{split}
\frac{d}{d\tau}&  f(\gamma_\ell(\tau), t_0 + \delta^2r_0^2-\tau)\\
&= \langle \nabla  f(\gamma_\ell(\tau), t_0 + \delta^2r_0^2-\tau), \dot{\gamma}_\ell(\tau)\rangle_{ g(t_0 + \delta^2r_0^2 - \tau) } - \ddt{}  f (\gamma_\ell(\tau), t_0 + \delta^2r_0^2-\tau),
\end{split}
\]
hence we have
\begin{equation}\label{xy2}
\begin{split}
&\left | f (x_\ell , t_0 + \delta^2r_0^2) - f ( y_\ell , t_0 ) \right|=\left| \int_0^{\delta^2r_0^2} \frac{d}{d\tau} f (\gamma_\ell(\tau), t_0 + \delta^2r_0^2-\tau) d\tau \right| \\
\leq& \left| \int_0^{\delta^2r_0^2} \langle \nabla f (\gamma_\ell(\tau), t_0 + \delta^2r_0^2-\tau), \dot{\gamma}_\ell(\tau)\rangle_{g(t_0 + \delta^2r_0^2 - \tau)} d\tau\right| \\
&\qquad\qquad\qquad\qquad\qquad + \left| \int_0^{\delta^2r_0^2}  \ddt{} f (\gamma_\ell(\tau), t_0 + \delta^2r_0^2-\tau) d\tau\right|.
\end{split}
\end{equation}
First, using the assumption (2), we can compute
\[
\begin{split}
\left| \int_0^{\delta^2r_0^2}  \ddt{} f (\gamma_\ell(\tau), t_0 + \delta^2r_0^2-\tau) d\tau\right| \leq& \int_0^{\delta^2r_0^2} B\left(1-(t_0 + \delta^2r_0^2)\right)^{-1} d\tau \\
\leq& \int_0^{\delta^2r_0^2} B2r_0^{-2} d\tau \leq 2\delta^2B.
\end{split}
\]
Next, since the scalar curvature is uniformly bounded from below, say $R\geq -2n$ on $M\times [1/2, 1)$ , using the assumption (2) and (\ref{lcenter'2}), we have
\begin{eqnarray*}
&& \left| \int_0^{\delta^2r_0^2} \langle \nabla  f (\gamma_\ell(\tau), t_0 + \delta^2r_0^2-\tau), \dot{\gamma}_\ell(\tau)\rangle_{g(t_0 + \delta^2r_0^2 - \tau)} d\tau\right| \\
&\leq& \int_0^{\delta^2r_0^2} |\nabla f|_{g(t_0 + \delta^2r_0^2 - \tau)} (\gamma_\ell(\tau), t_0 + \delta^2r_0^2-\tau)\cdot |\dot\gamma_\ell (\tau)|_{g(t_0 + \delta^2r_0^2 - \tau)} d \tau \\
&\leq& 2B^{\frac{1}{2}} r_0^{-1} \int_0^{\delta^2r_0^2} |\dot\gamma_\ell (\tau)|_{g(t_0 + \delta^2r_0^2 - \tau)} d \tau \\
&\leq& 2B^{\frac{1}{2}} r_0^{-1} \left(  \int_0^{\delta^2r_0^2} \tau^{-1/2} d\tau \right)^{1/2} \left(  \int_0^{\delta^2r_0^2} \tau^{1/2} |\dot\gamma_\ell (\tau)|^2_{g(t_0 + \delta^2r_0^2 - \tau)} d \tau \right)^{1/2}\\
&\leq& 4B^{\frac{1}{2}} \delta^{\frac{1}{2}} r_0^{-\frac{1}{2}} \left( \int_0^{\delta^2r_0^2} \tau^{1/2} \left( |\dot\gamma_\ell (\tau)|^2_{g(t_0 + \delta^2r_0^2 - \tau)} + R(\gamma_\ell(\tau)) + 2n \right) d \tau \right)^{1/2}\\
&\leq & 4B^{\frac{1}{2}} \delta^{\frac{1}{2}} r_0^{-\frac{1}{2}} \left( \mathcal{L}(\gamma_\ell) + \int_0^{\delta^2r_0^2} 2n d\tau \right)^{1/2}\\
&\leq & 4B^{\frac{1}{2}} \delta^{\frac{1}{2}} r_0^{-\frac{1}{2}} \left( 2n\delta r_0 + 2n\delta^2 r_0^2 \right)^{1/2}\leq C_0\delta B^{\frac{1}{2}},
\end{eqnarray*}
for some constant $C_0=C_0(n)<\infty$. Plugging these two inequalities into (\ref{xy2}), we obtain
$$\left | f (x_\ell , t_0 + \delta^2r_0^2) - f ( y_\ell , t_0 ) \right|\leq 2\delta^2B + C_0\delta B^{\frac{1}{2}} \leq C_0\delta B.$$
for some constant $C_0=C_0(n)<\infty$. This completes the proof of Lemma \ref{xy}.
\end{proof}

Combining Lemma \ref{xy} and (\ref{vxl'1}), we obtain that if we choose $\bar{\delta} = \bar{\delta}(n, B)>0$ small enough such that $\bar{\delta}\leq \frac{1}{100C_0B}$, where $C_0$ is the constant in Lemma \ref{xy}, we then have
\begin{equation} \label{vyl'1}
f (y_\ell , t_0)\in [A_\ell- 2, A_\ell + 2].
\end{equation}
for each $\ell=1, 2, \dots, N$. From this, we have the following distortion estimates.

\begin{lemma} \label{yball}
For any $\ell_1, \ell_2 \in \{1, 2, ..., N\}$ with $\ell_1 \ne \ell_2$, we have
$$d_{g(t_0)}(y_{\ell_1}, y_{\ell_2}) \geq B^{-\frac{1}{2}} r_0 . $$
\end{lemma}

\begin{proof}[\bf{Proof}]
We will prove this by contradiction. Suppose there exist $1 \leq \ell_1 \ne \ell_2 \leq N$ such that
$$d_{g(t_0)}(y_{\ell_1}, y_{\ell_2}) \leq B^{-\frac{1}{2}} r_0,$$
then from assumption (2), we have
\begin{eqnarray*}
\left| f (y_{\ell_1} , t_0) - f (y_{\ell_2}, t_0) \right| \leq B^{1/2} r_0^{-1} d_{g(t_0)}(y_{\ell_1}, y_{\ell_2}) \leq 1.
\end{eqnarray*}
On the other hand, by the definition of $A_\ell$, say (\ref{doAl}), we have
$$| A_{\ell_1} - A_{\ell_2} | \geq 10,$$
hence we obtain from (\ref{vyl'1}) that
\begin{eqnarray*}
\left| f (y_{\ell_1} , t_0) - f (y_{\ell_2}, t_0) \right| \geq 6.
\end{eqnarray*}
This is a contradiction, which proves Lemma \ref{yball}.
\end{proof}

We now use the assumption (3) to find an $H_{n}$-center $(z_\ell, t_0)$ of $(x_\ell, t_0 + \delta^2r_0^2)$ for each $ \ell =1, ..., N$, such that
$$B_{g(t_0)} \left( z_\ell, \sqrt{2H_n}\delta r_0 \right) \subset V(t_0). $$
Then we have the following lemma.

\begin{lemma} \label{zball'1}
There exists $C_0=C_0(n)<\infty$, such that for each $ \ell =1, ..., N$, we have
\begin{equation} \label{zball'2}
B_{g(t_0)}\left(z_\ell, \sqrt{2H_{n}}\delta r_0\right) \subset B_{g(t_0) }\left(y_\ell, C_0 \delta r_0\right).
\end{equation}
In particular, if we choose $\bar{\delta} = \bar{\delta}(n, B)>0$ small enough, then the balls
$$\left\{ B_{g(t_0)}\left(z_\ell, \sqrt{2H_{n}}\delta r_0\right) \right\}_{\ell=1}^N$$
are mutually disjoint, all contained in $V$.
\end{lemma}

\begin{proof}[\bf{Proof}]
First by (\ref{dvh}) and the definition of $H_{n}$-center, we have
\begin{equation} \label{Wxlzl}
d^{g(t_0)}_{W_1} (\nu_{x_\ell, t_0 + \delta^2r_0^2; s_0}, \delta_{z_\ell}) \leq \left( \var(\nu_{x_\ell, t_0 + \delta^2r_0^2; t_0}, \delta_{z_\ell})\right)^{1/2} \leq \sqrt{H_{n}} \delta r_0.
\end{equation}
On the other hand, from Lemma \ref{wdn} and (\ref{lcenter'1}), we have
\begin{equation} \label{Wxlyl}
\begin{split}
d^{g(t_0)}_{W_1} (\nu_{x_\ell, t_0 + \delta^2r_0^2; t_0}, \delta_{y_\ell}) &\leq C_0 \left( 1+ \frac{\mathcal{L}(\gamma_\ell)}{2 \delta r_0} - \cN^*_{t_0} (x_\ell, t_0 + \delta^2r_0^2) \right)^{1/2} \delta r_0\\
&\leq C_0 \left( 1+ n + C_0 \right)^{1/2} \delta r_0\leq C_0 \delta r_0,
\end{split}
\end{equation}
since we have the uniform lower bound of the Nash entropy
\begin{equation} \label{Nashbxl}
\cN^*_{t_0} (x_\ell, t_0 + \delta^2 r_0^2) \geq \mu [g(t_0), \delta^2 r_0^2] \geq -B,
\end{equation}
from (\ref{NM}), where we have used the assumption (1). Combining (\ref{Wxlzl}) and (\ref{Wxlyl}), and the triangle inequality of the $W_1$-Wasserstein distance, we have
\begin{eqnarray*}
d_{g(t_0)}(y_\ell, z_\ell)& = & d^{g(t_0)}_{W_1} ( \delta_{y_\ell} , \delta_{z_\ell} )  \\
& \leq & d^{g(t_0)}_{W_1} (\delta_{y_\ell} , \nu_{x_\ell, t_0 + \delta^2r_0^2; t_0} ) + d^{g(t_0)}_{W_1} ( \nu_{x_\ell, t_0 + \delta^2r_0^2; t_0} , \delta_{z_\ell} ) \\
&\leq& C_0 \delta r_0,
\end{eqnarray*}
for each $ \ell =1, ..., N$. This proves (\ref{zball'2}).

Finally, Lemma \ref{yball} implies that the balls $\left\{ B_{g(t_0)}\left(y_\ell, \frac{1}{3}B^{-\frac{1}{2}} r_0\right)  \right\}_{\ell=1}^N$ are mutually disjoint. Hence if we choose $\bar{\delta} = \bar{\delta}(n, B)>0$ small enough such that $\bar{\delta}\leq \frac{1}{100C_0B}$, then the balls $\left\{ B_{g(t_0)}(y_\ell, C_0\delta r_0)  \right\}_{\ell=1}^N$ are mutually disjoint, which implies by using (\ref{zball'2}) that the balls $\left\{ B_{g(t_0)}\left(z_\ell, \sqrt{2H_{n}}\delta r_0\right) \right\}_{\ell=1}^N$ are mutually disjoint. This finishes the proof of Lemma \ref{zball'1}.
\end{proof}

Now we can fix the constant $\bar{\delta} = \bar{\delta}(n, B)>0$ such that the above argument holds. By Lemma \ref{hcv}, since $(z_\ell, t_0)$ is an $H_{n}$-center of $(x_\ell, t_0 + \delta^2r_0^2)$ for each $ \ell =1, ..., N$, using the Nash entropy bound (\ref{Nashbxl}), we have the volume non-collapsing estimates
$$\vol_{g(t_0)}\left( B_{g(t_0)}\left(z_\ell, \sqrt{2H_{n}}\delta r_0\right) \right) \geq C^{-1} (\delta r_0)^{n},  $$
for some constant $C=C(n, B)<\infty$. Now, since the balls
$$\left\{ B_{g(t_0)}\left(z_\ell, \sqrt{2H_{n}}\delta r_0\right) \right\}_{\ell=1}^N$$
are mutually disjoint by Lemma \ref{zball'1}, all contained in $V$, using the assumption (3), we have
\begin{equation} \label{zball'5}
B r_0^{n} \geq \vol_{g(t_0)}\left(V\right)\geq \sum_{\ell=1}^{N} \vol_{g(t_0)}\left( B_{g(t_0)}\left(z_\ell, \sqrt{2H_{n}}\delta r_0\right) \right) \geq N C^{-1} \delta^{n}r_0^{n},
\end{equation}
which gives the estimate $N \leq C(n, \delta, B)$.

This completes the proof of Theorem \ref{gosce}.
\end{proof}


\medskip
\subsection{Harnack estimate for the K\"ahler-Ricci flow}

In this subsection,  we are under the set-up of Section \ref{setupofKRF}, that is, we have the solution to the unnormalized K\"ahler-Ricci flow, which we denote by $\omega(t), t\in [0, 1)$; and we have the solution to the normalized K\"ahler-Ricci flow, which we denote by $\tilde\omega(s),~~ s\in [0, \infty)$. These two flows are related by
\begin{equation}\label{rb unkrf and nkrf'3}
s=-\ln (1-t),~~ t=1-e^{-s},~~ \tilde{\omega}(s)=(1-t)^{-1}\omega(t).
\end{equation}
We use the notations and conventions from Section \ref{setupofKRF}. Our final goal is to prove Corollary \ref{cor2}.

Let $\Phi: X \rightarrow Y \subset \mathbb{CP}^N$ be the unique surjective holomorphic map induced by the limiting class $\vartheta$. We fix a smooth K\"ahler metric $\omega_B$ on the base space $Y$, where the subscript $B$ denotes the base space. For example, we can choose $\omega_B=\lambda \omega_\fs$ for any $\lambda>0$. The reason we use the notation $\omega_B$ is that we already used $\omega_Y$ in (\ref{omegaY}).

For any $p\in X$, we denote by $\mathcal{B}_{\omega_B}(p, r) = \Phi^{-1}\left(B_{\omega_B} ( \Phi(p) , r) \right)$ the pre-image of the geodesic ball $B_{\omega_B}(\Phi(p), r)$.

First, we have the following proposition for the unnormalized flow $\omega(t), t\in [0, 1)$.
%

\begin{proposition} \label{loe'1}
For any $ A\in ( 1 , \infty ) $, $\theta\in (0, 1)$, there exists constant $C<\infty$, which depends on $n, \omega_0, \omega_B, A, \theta$, such that the following statement holds.

Let $t_0\in [0, 1)$ and $p\in X$ be given. Suppose $B_{\omega_B} \left( \Phi(p) , (1-t_0)^{1/2} \right)\cap Y$ has only one component and we have the following volume bound
\begin{equation}\label{loeassump1}
\vol_{g(t_0)}\left(\mathcal{B}_{\omega_B}\left(p, 5(1-t_0)^{1/2} \right)\right)\leq A (1-t_0)^{n}.
\end{equation}
Let $w: \mathcal{B}_{\omega_B}\left(p, 5(1-t_0)^{1/2} \right)\times [t_0, t_0 + \theta (1-t_0) ]\to \mathbb{R}^+$ be a positive $C^1$-function satisfying that
\begin{equation}\label{loeassump2}
\frac{|\partial_t w|}{w} + \frac{|\nabla w|^2}{w^2} \leq \frac{A}{1-t} .
\end{equation}
Then we have
\begin{equation} \label{loe'2}
\frac{\sup_{ \mathcal{B}_{\omega_B}\left(p, (1-t_0)^{1/2} \right) } w ( \cdot, t_0 )}{ \inf_{ \mathcal{B}_{\omega_B}\left(p, (1-t_0)^{1/2} \right) } w ( \cdot, t_0 ) }  \leq C.
\end{equation}
\end{proposition}

\begin{proof}[\bf{Proof}]
Throughout the proof, all the constants will depend at most on $n, \omega_0, \omega_B$ and $A, \theta$.
Let $\delta\in (0, \theta)$ be a small constant, whose value will be determined in the course of the proof.

Denote by $r_0:=(1-t_0)^{1/2}$. In order to apply Theorem \ref{gosce}, we define
$$
U:=\mathcal{B}_{\omega_B}\left(p, r_0 \right),
$$
and
$$
V:=\mathcal{B}_{\omega_B}\left(p, 5r_0 \right).
$$
First, we have $U$ is connected, since $B_{\omega_B}\left(\Phi(p), r_0 \right)\cap Y$ has only one component and all the regular fibres are connected.

Next, the assumption (1) of Theorem \ref{gosce} holds trivially, and the assumption (3) of Theorem \ref{gosce} holds due to the assumption (\ref{loeassump1}).

It remains to verify the assumption (4) of Theorem \ref{gosce}, and justify that (\ref{loeassump2}) is enough to replace the assumption (2) in Theorem \ref{gosce}, since here we only assume that (\ref{loeassump2}) holds on $V \times [t_0, t_0 + \theta (1-t_0) ]$.
%

%

%
First, we have the following lemma, which verifies the assumption (4) of Theorem \ref{gosce} on our K\"ahler-Ricci flow background.
%

%
\begin{lemma} \label{zball'3}
If we choose $\bar{\delta}$ small enough, then the following statement holds.

Let $x\in U$ and $ \delta \in ( 0 , \bar{\delta}) $. For any $t \in [ t_0, t_0 + \delta^2 (1-t_0) )$, if $(z, t )$ is an $H_{2n}$-center of $(x, t_0 + \delta^2(1-t_0) )$, then we have
\begin{equation} \label{zball'4}
B_{g( t )}\left(z, \sqrt{2H_{2n}}\delta r_0 \right) \subset \mathcal{B}_{\omega_B}\left(p, 3r_0 \right) .
\end{equation}
\end{lemma}

\begin{proof}[\bf{Proof}]
We will determine $\bar{\delta}>0$ in the course of the proof. Let $ \delta \in ( 0 , \bar{\delta}) $. Fix a time $t_1 \in [ t_0, t_0 + \delta^2 (1-t_0) )$.

We first choose a smooth cut-off function $\eta$ on $\mathbb{CP}^N$ satisfying
\begin{enumerate}

\medskip
\item $\eta \equiv 1$ in $B_{\omega_B}( \Phi(p) ,  r_0 )$ and $0\leq \eta \leq 1$,

\medskip
\item $\textnormal{supp} ~\eta \subset\subset B_{\omega_B}\left(\Phi(p) , 2 r_0 \right)$,

\medskip
\item $|\nabla \eta|^2_{\omega_B} \leq C_0 (1-t_0)^{-1}$,

\medskip
\item $ |\nabla^2 \eta|_{\omega_B} \leq C_0 (1-t_0)^{-1},$

\end{enumerate}
\medskip
for some fixed constant $C_0 < \infty$. We still use $\eta$ to denote $\eta\circ\Phi$, which is a smooth cut-off function on $X$. We now define
$$\tilde \eta (\cdot , t) = e^{-B_0 (1-t_0)^{-1}(t-t_0)}\eta (\cdot) , $$
for some large constant $B_0$ to be determined. By the parabolic Schwarz Lemma, say Lemma \ref{pschwarz}, for all $t\in [0, 1)$, we have
$$\left| \Delta_{\omega(t)} \eta \right| \leq CC_0 (1-t_0)^{-1}.$$
Hence if we fix $B_0=2CC_0$, then we can compute
\begin{equation} \label{eta'1}
\Box_{\omega(t)} \tilde \eta = e^{-B_0 (1-t_0)^{-1}(t-t_0)} \left(  - B_0 (1-t_0)^{-1} - \Delta_{\omega(t)} \eta \right) \leq  0,
\end{equation}
for all $t\in [0, 1)$. Furthermore, for all $t\in [t_0, t_0 + \delta^2(1-t_0)]$, we have
\begin{enumerate}

\medskip
\item $0\leq \tilde \eta (\cdot , t) \leq 1; $

\medskip
\item $\tilde \eta (\cdot , t) \geq e^{-B_0}$ on $\mathcal{B}_{\omega_B}\left(p, r_0 \right)$;

\medskip
\item  $ \textnormal{supp}~\tilde \eta \subset\subset \mathcal{B}_{\omega_B}\left(p, 2r_0 \right)$.

\end{enumerate}
\medskip
Now for any $x\in U$, we let $(z, t_1)$ be any $H_{2n}$-center of $(x, t_0+\delta^2(1-t_0))$. Combining the above three properties for $\tilde \eta$ with (\ref{eta'1}), we can compute
\begin{eqnarray*}
&&\nu_{x, t_0 + \delta^2(1-t_0) ; t_1} \left( \mathcal{B}_{\omega_B}\left(p, 2r_0 \right) \right)\\
&\geq & \int_X \tilde \eta ~d\nu_{x, t_0 + \delta^2(1-t_0) ; t_1} \\
& =& \int_X \tilde \eta ~d\nu_{x, t_0 + \delta^2(1-t_0) ; t_0 + \delta^2(1-t_0)}  - \int_{ t_0 }^{ t_0 + \delta^2(1-t_0) } \int_X  \Box_{\omega(t)} \tilde \eta ~  d\nu_{x, t; t_0} dt\\
&\geq& \tilde\eta(x, t_0 + \delta^2(1-t_0))\\
&\geq & e^{-B_0}.
\end{eqnarray*}
In the last inequality, we have used $x \in U = \mathcal{B}_{\omega_B}\left(p, r_0 \right)$. Since $(z, t_1)$ is an $H_{2n}$-center of $(x, t_0 + \delta^2(1-t_0))$, by using Lemma \ref{hnc'2}, we have
\begin{equation}\label{hcd}
d_{g(t_1)} \left(z, \mathcal{B}_{\omega_B}\left(p, 2r_0 \right) \right) \leq C \sqrt{ t_0 + \delta^2(1-t_0) - t_1 } \leq C \delta r_0.
\end{equation}
Hence, for any $\tilde{x}\in B_{g(t_1)}\left(z, \sqrt{2H_{2n}}\delta r_0 \right)$, we can find a point $y\in \mathcal{B}_{\omega_B}\left(p, 2r_0 \right))$, and a smooth curve $\gamma: [0, 1]\to X$ with $\gamma(0)=\tilde{x}$ and $\gamma(1)=y$, such that
$$
\int_{0}^{1} |\dot\gamma (a)|_{g(t_1)} da < \sqrt{2H_{2n}}\delta r_0 + C \delta r_0 \leq C \delta r_0.
$$
Then we define $\tilde \gamma: [0, 1]\to \mathbb{CP}^N$ by $\tilde \gamma:= \pi \circ \gamma$, then by applying the Schwarz Lemma, say Lemma \ref{pschwarz}, for all $a\in [0, 1]$, we have
$$|\dot {\tilde \gamma} (a)|^2_{\omega_B} = \Phi^* \omega_B \left( \dot {\gamma} (a) , \dot {\gamma} (a)  \right) \leq C g(t_1) \left( \dot {\gamma} (a) , \dot {\gamma} (a)  \right) = C |\dot\gamma (a)|^2_{g(t_1)},$$
hence we have
$$ d_{\omega_B}( \Phi(\tilde{x}) , \Phi(y) ) \leq \int_{0}^{1} |\dot{\tilde \gamma} (a)|_{\omega_B} da \leq \int_{0}^{1} C|\dot\gamma (a)|_{g(t_1)} da \leq C \delta r_0, $$
which implies that
\[
\begin{split}
d_{\omega_B}( \Phi(\tilde{x}) , \Phi(p) ) &\leq d_{\omega_B}( \Phi(\tilde{x}) , \Phi(y) ) + d_{\omega_B}( \Phi(y) , \Phi(p) )\\
&\leq C \delta r_0 + 2 r_0 < 3 r_0,
\end{split}
\]
provided that we choose $\bar{\delta} \leq \frac{1}{100C}$, from which we conclude that
$$ B_{g(t_1)}\left(z, \sqrt{2H_{2n}}\delta r_0 \right) \subset \mathcal{B}_{\omega_B}\left(p, 3r_0 \right) .$$

This completes the proof of Lemma \ref{zball'3}.
\end{proof}

Next, we have the following lemma, which will justify that (\ref{loeassump2}) is enough to replace the assumption (2) in Theorem \ref{gosce}.
\begin{lemma} \label{zball'6}
If we choose $\bar{\delta}$ small enough, then the following statement holds.

Let $x\in U$ and $ \delta \in ( 0 , \bar{\delta}) $. Let $(y, t_0)$ be an $\ell_{2n}$-center of $(x, t_0+\delta^2(1-t_0))$, let $\gamma: [0, \delta^2 (1-t_0) ] \rightarrow X \times [0, 1)$ be any smooth minimizing $\mathcal{L}$-geodesic connecting $(x, t_0+\delta^2(1-t_0))$ to $(y, t_0)$. Then we have
\begin{equation} \label{zball'7}
\gamma ( \tau ) \in \mathcal{B}_{\omega_B}\left(p, 3r_0 \right) \times \left\{ t_0 + \delta^2(1-t_0) - \tau \right\} ,
\end{equation}
for all $\tau \in [0, \delta^2 (1-t_0) ]$.
\end{lemma}
\begin{proof}
According to our assumption, we have
$$ \frac{ \mathcal{L}(\gamma) }{2 \delta r_0} \leq n.  $$
Fix $\tau_0 \in ( 0, \delta^2 (1-t_0) ]$. Since we have uniform lower bound on scalar curvature $R_g \geq -C$, we can then compute
\begin{eqnarray*}
2 n \delta r_0&\geq& \mathcal{L}(\gamma) = \int_0^{\delta^2r_0^2} \tau^{1/2} \left( |\dot\gamma (\tau)|^2_{g(t_0 + \delta^2r_0^2 - \tau)} + R_g(\gamma(\tau)) \right) d \tau  \\
&\geq& \int_0^{ \tau_0 } \tau^{1/2} \left( |\dot\gamma (\tau)|^2_{g(t_0 + \delta^2r_0^2 - \tau)} + R_g(\gamma (\tau)) \right) d \tau + \int_{ \tau_0 }^{ \delta^2r_0^2 } \tau^{1/2} \left( - C \right) d \tau  \\
&\geq& \mathcal{L}(\gamma|_{ [ 0 , \tau_0 ] }) - C \delta^3 r_0^3 ,\\
\end{eqnarray*}
which we can rewrite as
\begin{equation} \label{zball'8}
\mathcal{L}( \gamma|_{ [ 0 , \tau_0 ] } ) \leq C \delta r_0.
\end{equation}
Denote by $ t_1 = t_0 + \delta^2(1-t_0) - \tau_0 $.  Hence, we can apply Lemma \ref{wdn} to compute
\begin{equation}\label{zball'9}
\begin{split}
&d^{g( t_1 )}_{W_1} \left (\nu_{x_0, t_0 + \delta^2(1-t_0) ; t_1 } , \delta_{ \gamma ( \tau_0 ) } \right )  \\
\leq& C \left(1+ \frac{ \mathcal{L}( \gamma|_{ [ 0 , \tau_0 ] } ) }{2 \tau_0 ^{1/2}} - \cN^*_{ t_1 }(x_0, t_0 + \delta^2(1-t_0) ) \right) ^{1/2} \tau_0 ^{1/2}  \\
\leq& C \left(1+ \frac{ C \delta r_0 }{2 \tau_0 ^{1/2}} + C \right) ^{1/2} \tau_0 ^{1/2} \leq C \delta^{1/2} r_0^{1/2} \tau_0^{1/4} + C \tau_0 ^{1/2} \leq C \delta r_0.
\end{split}
\end{equation}
On the other hand, let $(z, t_1 )$ be an $H_{2n}$-center of $(x, t_0 + \delta^2(1-t_0) )$, then we can apply (\ref{dvh}) to compute
\begin{equation}\label{zball'10}
\begin{split}
d^{g( t_1 )}_{W_1} \left (\nu_{x_0, t_0 + \delta^2(1-t_0) ; t_1 } , \delta_{ z } \right ) \leq C \tau_0^{1/2} \leq C \delta r_0 .
\end{split}
\end{equation}
Combining (\ref{zball'9}) with (\ref{zball'10}), we can apply the triangle inequality to compute
\begin{equation}\label{zball'11}
\begin{split}
&d_{g( t_1 )} \left ( \gamma ( \tau_0 ) , z \right ) = d^{g( t_1 )}_{W_1} \left ( \delta_{ \gamma ( \tau_0 )} , \delta_{ z }  \right) \\
\leq & d^{g( t_1 )}_{W_1} \left ( \delta_{ \gamma ( \tau_0 ) } , \nu_{x_0, t_0 + \delta^2(1-t_0) ; t_1 } \right ) + d^{g( t_1 )}_{W_1} \left (\nu_{x_0, t_0 + \delta^2(1-t_0) ; t_1 } , \delta_{ z } \right ) \leq C \delta r_0.  \\
\end{split}
\end{equation}
But according to (\ref{hcd}) in Lemma \ref{zball'3}, we have
\begin{equation}\label{zball'12}
d_{g(t_1)} \left(z, \mathcal{B}_{\omega_B}\left(p, 2r_0 \right) \right) \leq C \delta r_0 ,
\end{equation}
hence we can combine (\ref{zball'11}) with (\ref{zball'12}) to obtain
\begin{equation}\label{zball'13}
d_{g(t_1)} \left( \gamma ( \tau_0 ) , \mathcal{B}_{\omega_B}\left(p, 2r_0 \right) \right) \leq C \delta r_0 .
\end{equation}
%
Now, by applying the Schwarz lemma, totally the same argument as in Lemma \ref{zball'3} gives
$$ d_{\omega_B}( \Phi( \gamma ( \tau_0 ) ) , p ) < 3 r_0 ,  $$
provided that we choose $\bar{\delta}$ small enough, hence we have
$$ \gamma ( \tau_0 ) \in \mathcal{B}_{\omega_B}\left(p, 3r_0 \right) \times \left\{ t_0 + \delta^2(1-t_0) - \tau_0 \right\} . $$
This completes the proof of Lemma \ref{zball'6}.
\end{proof}

Now we can justify that (\ref{loeassump2}) is enough to replace the assumption (2) in Theorem \ref{gosce}. There are three places in the proof of Theorem \ref{gosce} where we have used the assumption (2) in Theorem \ref{gosce}, namely (\ref{xl1}), Lemma \ref{xy} and Lemma \ref{yball}.

First of all, (\ref{loeassump2}) is enough for (\ref{xl1}) since $x_{\ell} \in U \subset V$.

Next, for the $\cL$-geodesic $\gamma_\ell$ in Theorem \ref{gosce}, we can apply Lemma \ref{zball'6} to obtain that the whole space-time curve $\gamma_\ell$ is contained in $V \times [t_0, t_0 + \theta (1-t_0) ]$, hence (\ref{loeassump2}) is enough for the application in Lemma \ref{xy}.

Finally, according Lemma \ref{zball'6} again, the point $(y_\ell, t_0)$ in Theorem \ref{gosce}, which is an $\ell_{2n}$-center of $(x_\ell, t_0 + \delta^2r_0^2)$ with $x_{\ell} \in U$, satisfies that $ (y_\ell, t_0) \in \mathcal{B}_{\omega_B}\left(p, 3r_0 \right) \times \left\{ t_0 \right\} $. By applying the Schwarz lemma again, totally the same argument as in Lemma \ref{zball'3} gives
\begin{equation}\label{zball'14}
B_{g(t_0)} \left( y_\ell , C_1^{-1}r_0 \right) \subset \mathcal{B}_{\omega_B}\left(p, 4r_0 \right) ,
\end{equation}
for some constant $C_1=C_1(n, \omega_0)<\infty$. We then enlarge the constant $B$ in Theorem \ref{gosce} such that $B\geq 10C_1$ (now $B$ need to depend on $n, \omega_0$). Hence for the points $y_{\ell_1}, y_{\ell_2}$ in Lemma \ref{yball}, if $d_{g(t_0)}(y_{\ell_1}, y_{\ell_2}) \leq B^{-\frac{1}{2}} r_0$, then by using (\ref{zball'14}), any $g(t_0)$-minimizing geodesic connecting $y_{\ell_1}$ to $y_{\ell_2}$ must be contained in $\mathcal{B}_{\omega_B}\left(p, 4r_0 \right) \subset V$, hence (\ref{loeassump2}) is enough for the application in Lemma \ref{yball}.

Now, we can fix the constant $\bar{\delta}$ such that our previous arguments and Theorem \ref{gosce} holds, and fix $\delta = \bar{\delta}/2$. Then we conclude from Theorem \ref{gosce} that
$$
\frac{\sup_{ U } w ( \cdot, t_0 )}{ \inf_{ U } w ( \cdot, t_0 ) }  \leq C.
$$
This completes the proof of Proposition \ref{loe'1}.
\end{proof}
\begin{theorem} \label{scanddearv}
Let $g(t)$ be a maximal solution of the K\"ahler-Ricci flow (\ref{unkrflow1'1}) on $X\times [0, 1)$ as described above. For any $A, T<\infty$ and smooth closed $(1,1)$-form $\theta_Y\in \vartheta$ on $Y$, there exist $\ep>0, C<\infty$, depending on $n, g_0, \theta_Y, g_B, A,  T$, such that the following holds.

Suppose $t_0\in [1/2, 1)$ and $p_{t_0}$ is a Ricci vertex associated to $\theta_Y$. If  $t\in [t_0-T(1-t_0), 1-T^{-1}(1-t_0) ]$ and
\begin{equation}\label{scanddearv1}
\vol_{g(t)}\left( \mathcal{B}_{g_B}\left( p_{t_0}, 10(1-t)^{1/2} \right) \right)\leq A (1-t)^{n} .
\end{equation}
Then we have the estimates
\begin{equation} \label{scanddearv2}
\diam \left( \mathcal{B}_{g_B}\left( p_{t_0}, (1-t)^{1/2} \right), g(t) \right)  \leq C(1-t)^{1/2} ,
\end{equation}
\begin{equation} \label{scanddearv3}
\sup_{\mathcal{B}_{g_B}\left( p_{t_0}, (1-t)^{1/2} \right) } | \rr( \cdot , t ) | \leq \frac{C}{1-t},
\end{equation}
where $\rr(\cdot, t)$ denotes the scalar curvature of $g(t)$.
\end{theorem}
\begin{proof}

First, since $Y$ is a normal variety, we can find a constant $\ep_0>0$ (depending on $\omega_0$), such that for any $r\in (0, \ep_0)$ and $p\in Y$, we have $B_{\omega_B} \left( p , r \right)\cap Y$ has only one component.

Then we can choose $\ep>0$ (depending on $T$) such that if $t_0\in (0 ,1)$ satisfies that $1-t_0\leq \ep$, then $t_0-T(1-t_0)\geq 0$. Then for any $t_1\in [t_0-T(1-t_0), 1-T^{-1}(1-t_0) ]$, we have $1-t_1\leq (1+T)(1-t_0)\leq (1+T)\ep$. We then choose $\ep=\ep(\omega_0, T) >0$ small enough, such that $2(1+T)^{1/2}\ep^{1/2}\leq \ep_0$, hence $B_{\omega_B} \left( p , 2(1-t_1)^{1/2} \right)\cap Y$ has only one component for any $p\in Y$. Then we assume that we have the following volume bound
\begin{equation}\label{cor2a2}
\vol_{g(t_1)}\left(\mathcal{B}_{\omega_B}\left(p_{t_0}, 10(1-t_1)^{1/2} \right) \right)\leq A (1-t_1)^{n},
\end{equation}
where $p_{t_0}$ denotes the Ricci vertex associated to $\theta_Y$ at time $t_0$. We need to prove the estimates (\ref{cor2'1}) and (\ref{cor2'2}) at the time $t_1$.

Let $B_0=B(n , \omega_0, \|\rho \|_{C^2(\omega_{\fs})}, \ln (2T))$ be the constant from Lemma \ref{a'3}, then we define
\begin{equation}\label{cordob3}
b(s) = e^{s-s_0}a(s_0)-B_0,
\end{equation}
where $s_0=-\ln (1-t_0)$, and $s$ is the time parameter in the normalized flow $(X, \tilde g(s))$. By Lemma \ref{a'3}, we have $b(s)\leq a(s)$ for all $s\in [0 , s_0+\ln (2T)]$. We then denote by
\begin{equation}\label{cordov0}
v = u - b(s)+1,
\end{equation}
which is a smooth function on the normalized flow $(X, \tilde g(s))$, $s\in [0, \infty)$. Then we have $v\geq 1$ on $X\times [0 , s_0+\ln (2T)]$.

As before, we still denote by $v(t)=v(s(t))$, where $s(t)=-\ln(1-t)$, which makes $v$ a function on the normalized flow $(X, g(t))$, $t\in [0 , 1)$, hence we have $v\geq 1$ on $X\times [0, 1-(2T)^{-1}(1-t_0)]$. By (\ref{eofdsvandgradientv}) (with $s_0$ being replaced by $s_i$ there), we have
\begin{equation} \label{eofdsvandgradientv0}
\frac{|\partial_t v|}{v} + \frac{| \Delta v |}{v} + \frac{|\nabla v|^2}{v^2} \leq \frac{C}{1-t},
\end{equation}
on $(X, g(t))$, $t\in [0, 1-(2T)^{-1}(1-t_0)]$, for some constant $C<\infty$, which depends on $n , \omega_0, \|\rho \|_{C^4(\omega_{\fs})}, T$.

Since $t_1<1-T^{-1}(1-t_0)$, we have, for $\theta=1/2$,
$$
t_1+ \theta (1-t_1)\leq \theta + (1-\theta)[1-T^{-1}(1-t_0)] = 1-(2T)^{-1}(1-t_0),
$$
which implies that $[t_1, t_1+ \theta (1-t_1)]\subset [0, 1-(2T)^{-1}(1-t_0)]$. Hence we can apply Proposition \ref{loe'1} to obtain
\begin{equation} \label{loe3}
\frac{\sup_{ \mathcal{B}_{\omega_B}\left(p_{t_0}, 2(1-t_1)^{1/2} \right) } v ( \cdot, t_1 )}{ \inf_{ \mathcal{B}_{\omega_B}\left(p_{t_0}, 2(1-t_1)^{1/2} \right) } v ( \cdot, t_1 ) }  \leq C ,
\end{equation}
for some constant $C<\infty$ depends on $n, \omega_0, \omega_B, \|\rho \|_{C^4(\omega_{\fs})}, A, T$. On the other hand, we have
$$
v( p_{t_0} , t_0 ) = B_0+1,
$$
hence by the time derivative estimate in (\ref{eofdsvandgradientv0}), we have
\begin{equation} \label{ubofvps0}
v( p_{t_0} , t_1 ) \leq C.
\end{equation}
Combining (\ref{loe3}) and (\ref{ubofvps0}), we obtain that
\begin{equation} \label{ubofvps02}
\sup_{ \mathcal{B}_{\omega_B}\left(p_{t_0}, 2(1-t_1)^{1/2} \right) } v ( \cdot, t_1 ) \leq C.
\end{equation}
But on $X$, by the parabolic Schwarz lemma and (\ref{eofdsvandgradientv0}), we have
$$
R_g(t_1) = (1-t_1^{-1})[n - \Delta v + \tr_{\omega(t_1)}(\alpha-\omega_Y) ]\leq \frac{Cv(t_1)}{1-t_1} ,
$$
hence from (\ref{ubofvps02}), we have
\begin{equation} \label{ubofRt1}
\sup_{ \mathcal{B}_{\omega_B}\left(p_{t_0}, 2(1-t_1)^{1/2} \right) } R_g ( \cdot, t_1 ) \leq \frac{C}{1-t_1}.
\end{equation}
This proves (\ref{scanddearv3}). Next we choose a maximal set of points
$$\left\{ z_\ell \right\}_{\ell=1}^{Q}\subset \mathcal{B}_{\omega_B}\left(p_{t_0}, (1-t_1)^{1/2} \right),$$
such that $\left\{ B_{g(t_1)}\left(z_\ell, \delta(1-t_1)^{1/2} \right) \right\}_{\ell=1}^{Q}$ are mutually disjoint, hence the balls
$$\left\{ B_{g(t_1)}\left(z_\ell, 2\delta(1-t_1)^{1/2} \right) \right\}_{\ell=1}^{Q}$$
cover the domain $\mathcal{B}_{\omega_B}\left(p_{t_0}, (1-t_1)^{1/2} \right)$. Here $\delta>0$ is a constant, depends on $n, \omega_0, \omega_B$, satisfying that, for any $z\in \mathcal{B}_{\omega_B}\left(p_{t_0}, (1-t_1)^{1/2} \right)$, we have
$$B_{g(t_1)}\left(z, 2\delta(1-t_1)^{1/2} \right)\subset \mathcal{B}_{\omega_B}\left(p_{t_0}, 2(1-t_1)^{1/2} \right).$$
Such $\delta$ exists due to the parabolic Schwarz Lemma, say Lemma \ref{pschwarz}. Then by (\ref{ubofRt1}) and Perelman's volume non-collapsing estimate, we have
\begin{equation}\label{vncpatzl}
\vol_{g(t_1)}\left( B_{g(t_1)}\left(z_\ell, \delta(1-t_1)^{1/2} \right) \right) \geq C^{-1} (1-t_1)^n .
\end{equation}
Hence by (\ref{cor2a2}), we have
\begin{equation} \label{zlballQ}
\begin{split}
A (1-t_1)^{n}&\geq \vol_{g(t_1)}\left(\mathcal{B}_{\omega_B}\left(p_{t_0}, 10(1-t_1)^{1/2} \right)\right) \\
&\geq \sum_{\ell=1}^{Q} \vol_{g(t_1)}\left( B_{g(t_1)}\left(z_\ell, \delta(1-t_1)^{1/2} \right) \right)\geq Q C^{-1} (1-t_1)^n , \\
\end{split}
\end{equation}
which implies $Q\leq C$ for some constant $C<\infty$ depends on $n, \omega_0, \omega_B, \|\rho \|_{C^4(\omega_{\fs})}$, $A, T$. Hence any two points in $\mathcal{B}_{\omega_B}\left(p_{t_0}, (1-t_1)^{1/2} \right)$ can be connected by a smooth curve in $\mathcal{B}_{\omega_B}\left(p_{t_0}, 2(1-t_1)^{1/2} \right)$ of $d_{g(t_1)}$-length no greater than $ Q (1-t_1)^{1/2}\leq C (1-t_1)^{1/2} $. This proves (\ref{scanddearv2}).
The above arguments hold for $t_0\in [1-\ep, 1)$, where $\ep>0$ is a fixed constant. When $t_0\in [0, 1-\ep)$, for any $t\leq 1-T^{-1}(1-t_0)$, we have $1-t\geq \ep T^{-1}$, hence the result hold trivially. This completes the proof.
\end{proof}
Now we can prove Theorem \ref{cor2}.
\begin{proof}[Proof of Theorem \ref{cor2}]
This is a direct corollary of Theorem \ref{scanddearv}.
\end{proof}
%


\bigskip
\section{Gromov-Hausdorff convergence of Ricci flows with locally bounded scalar curvature}\label{hkanddde}

In this section, we consider the more general set-up of Ricci flows with locally bounded scalar curvature. First, we need to define in what sense our scalar curvature is locally bounded. Let $(M, g(t))_{t\in I}$ be a smooth Ricci flow on a compact $n$-dimensional manifold with the interval $I\subset \mathbb{R}$.

\begin{definition}[Based barrier of the scalar curvature] \label{bbofsc}
Let $v:M\times I\to \mathbb{R}$ be a $C^1$-function and $C<\infty$ be a constant.  We call $v$ a \textbf{$C$-barrier of $R_g$} if the following hold on $M\times I$:
\begin{enumerate}
\item $v\geq 1$;
\item $|\partial_t \ln v| + |\nabla \ln v|^2 \leq C$;
\item $R_{g} \leq Cv$.
\end{enumerate}
Let $(x_0, t_0)\in M\times I$ and $B<\infty$. Then we say $v$ is \textbf{$B$-based at $(x_0, t_0)$} if
$$
v(x_0, t_0)\leq B.
$$
\end{definition}
\begin{remark}
For finite time solution of K\"ahler-Ricci flow on projective manifolds, such based barrier functions arise naturally from the normalized Ricci potential.
Suppose $I=[-T, 0]$ for some $T\in (0, \infty]$. Let $\lambda>0$ be a rescaling factor. Denote by $\tilde g_t= \lambda^{-2} g_{\lambda^2 t}$ and $\tilde v(t)=v(\lambda^2 t)$, where $t\in [-\lambda^{-2}T, 0]$. If $v$ is $C$-barrier $B$-based at $(x_0, t_0)$ of $R_g$, then $\tilde v$ is $\lambda^{2}C$-barrier $B$-based at $(x_0, \lambda^{-2}t_0)$ of $R_{\tilde g}$.
\end{remark}
Now, let $( M_i, (g_{i,t})_{t\in [-T_i, 0]} , (p_i, 0) )$ be a sequence of pointed Ricci flows on compact manifolds of dimension $n$ and $T_{\infty}:= \lim_{i\to\infty} T_i$. By the results of \cite{Bam20b}, passing to a subsequence, we can obtain $\bF$-convergence (see Definition \ref{mfpairs} and Definition \ref{corrspandFcon}) on compact time-intervals
\begin{equation}\label{FcofRF'1}
(M_i, (g_{i,t})_{t\in [-T_i , 0]}, (\nu_{p_i,0; t})_{t\in [-T_i , 0]}) \xrightarrow[i\to\infty]{\bF,\CCC}  (\cX, (\nu_{p_\infty; t})_{t\in (-T_{\infty} , 0]}),
\end{equation}
within some correspondence $\CCC$, where $\cX$ is a future continuous and $H_{n}$-concentrated metric flow of full support over $(-T_{\infty} , 0]$.
For the non-collapsing assumption, we assume that, for some uniform $Y_0<\infty$, we have
\begin{equation}\label{lbofNashatbasep}
\nu[g_{i, -T_i}, 2T_i]\geq -Y_0.
\end{equation}
According to \cite{Bam20c}, we can decompose $\cX$ into its regular and singular part
\begin{equation}\label{rsd'1}
\cX= \cR \sqcup \cS,
\end{equation}
where $\cR$ is a dense open subset of $\cX$. The singular set $\cS$ has parabolic Minkowski dimension $\leq n-2$. Also, $\cR$ carries the structure of a Ricci flow spacetime $(\cR, \mathfrak{t}, \partial_{\mathfrak{t}}, g)$. For any $t\in (-T_{\infty}, 0)$, $\cR_t=\cX_t\cap\cR$, we have that $(\cX_t, d_t)$ is the metric completion of $(\cR_t, g_t)$.
For the local scalar curvature bound assumption, we suppose there exist a sequence of constants $C_i<\infty$ and a sequence of functions $v_i$, such that $v_i$ is a $C_i$-barrier of $R_{g_i}$ and $Y_0$-based at $(p_i, 0)$ for each $i$.

The $\bF$-convergence is a rather technical and analytical notion of convergence, our main result of this section is the following improvement on the convergence.

\begin{theorem}\label{mainforRF2}
Suppose we have $C_i\leq Y_0$ for all $i$. Then for every $t\in (-T_\infty, 0)$ where the $\bF$-convergence (\ref{FcofRF'1}) is time-wise, passing to a subsequence, we have that $(M_{i}, d_{g_{i,t}}, p_i)$ converge to $(\cX_{t}, d_{t}, q_t)$ in the Gromov-Hausdorff topology for some $q_t\in \cX_{t}$.
\end{theorem}
Throughout this section, unless otherwise stated, all the constants will depend at most on $n , Y_0$. We will omit this dependence in this section for convenience.


\medskip
\subsection{Propagation of the barrier function in parabolic neighborhoods}

We start with a simple lemma, which will be used repeatedly. This lemma states that the boundedness of the barrier function at a point propagates to give a bound on the barrier function in the usual parabolic neighborhood of that point.
\begin{lemma} \label{viprop2}
For any $A, D<\infty$, there exists a constant $C<\infty$ depending on $A, D$, such that the following statements hold on the Ricci flow $(M_i, (g_{i,t})_{t\in [-T_i , 0]})$.
For any $(x_0, t_0)\in M_i\times [-T_i, 0]$, if
$$
v_i(x_0, t_0)\leq A,
$$
then we have
\begin{enumerate}
\item $v_i(x_0, t)\leq C$ for all $t\in [ t_0-D , t_0+D ]\cap [-T_i, 0]$;
\item $v_i(x, t_0)\leq C$ for all $x\in B_{g_i}( x_0, t_0, D )$.
\end{enumerate}
\end{lemma}
The next lemma states that the boundedness of the barrier function propagated in the $P^*$-parabolic neighborhoods.
\begin{proposition} \label{viprop3}
For any $\eta\in (0, 1)$, $A, D, T^{\pm}<\infty$, there exists a constant $C<\infty$ depends on $\eta, A, D, T^{\pm}$, such that the following statements hold on the flow $(M_i, (g_{i,t})_{t\in [-T_i , 0]})$.
Suppose $(x_0, t_0)\in M_i\times [-T_i+ T^- +10\eta , 0]$ satisfies
$$
v_i(x_0, t_0)\leq A .
$$
Then for any $(y_0, s_0)\in P^*(x_0, t_0; D, -T^-, T^+)$, we have
$$
v_i(y_0, s_0)\leq C .
$$
\end{proposition}
\begin{proof}
Throughout this proof, all the constants depend at most on $\eta, A, D, T^{\pm}$. All the times we consider in this proof is in $[-T_i+ \eta , 0]$, hence we have $R_{g_i}\geq -C$ when we need the lower scalar curvature bound.

Since $v_i(x_0, t_0)\leq A$, by Lemma \ref{viprop2},
\begin{equation}\label{boundvi1}
v_i(x_0, t)\leq C ,
\end{equation}
for all $t\in [t_0-T^- , t_0+ T^+]\cap [-T_i , 0]$. Hence we have
\begin{equation}
R_{g_i}(x_0, t)\leq C ,
\end{equation}
for all $t\in [t_0-T^- , t_0+ T^+]\cap [-T_i , 0]$. Hence from the uniform lower bound of the $\nu$-entropy, by Lemma \ref{W1balongwl}, we have
\begin{equation}\label{dW1bx0viP}
\begin{split}
d^{g_{i, t_0-T^-}}_{W_1} (\nu_{x_0, t_0; t_0-T^-} , \delta_{x_0} )  \leq C .
\end{split}
\end{equation}
Next, let $(z_0, t_0 - T^-)$ be an $\ell_n$-center of $(y_0, s_0)$, that is, if we denote by $\tau_0=s_0-(t_0 - T^-)$, then we can find smooth spacetime curve $\gamma:[0, \tau_0]\to M_i\times [t_0 - T^-, s_0]$ connecting $(y_0, s_0)$ to $(z_0, t_0 - T^-)$, such that
$$
\cL (\gamma) \leq n\tau_0^{1/2} .
$$
Hence by Lemma \ref{wdn}, we have
\begin{equation}\label{dW1by0z0P}
\begin{split}
d^{g_{i, t_0-T^-}}_{W_1} (\nu_{y_0, s_0; t_0-T^-} , \delta_{z_0} )  \leq C .
\end{split}
\end{equation}
Now, combine (\ref{dW1bx0viP}), (\ref{dW1by0z0P}) with the condition $(y_0, s_0)\in P^*(x_0, t_0; D, -T^-, T^+)$, we can apply the triangle inequality to obtain
\begin{equation}\label{deofz0x0}
\begin{split}
& d_{g_{i, t_0-T^-}} ( z_0 , x_0 ) \leq d^{g_{i, t_0-T^-}}_{W_1} ( \delta_{z_0} , \nu_{y_0, s_0; t_0-T^-} )  \\
& + d^{g_{i, t_0-T^-}}_{W_1} ( \nu_{y_0, s_0; t_0-T^-} , \nu_{x_0, t_0; t_0-T^-} )  + d^{g_{i, t_0-T^-}}_{W_1} (\nu_{x_0, t_0; t_0-T^-} , \delta_{x_0} )  \leq C .
\end{split}
\end{equation}
Hence from (\ref{boundvi1}) and Lemma \ref{viprop2}, we have
$$
v_i( z_0, t_0-T^- ) \leq C .
$$
Finally, along the reduced curve $\gamma$, we can compute
\[
\begin{split}
&\left | \ln v_i ( y_0, s_0 ) - \ln v_i ( z_0, t_0-T^- ) \right| = \left| \int_0^{\tau_0} \frac{d}{d\tau} \ln v_i (\gamma(\tau), s_0 - \tau) d\tau \right| \\
\leq& \left| \int_0^{\tau_0} \langle \nabla \ln v_i (\gamma(\tau), s_0 -\tau), \dot{\gamma}(\tau)\rangle_{g_{i, s_0-\tau}} d\tau \right|  + \left| \int_0^{\tau_0}  \ddt{} \ln v_i (\gamma(\tau), s_0 -\tau) d\tau\right| \\
\leq& C \int_0^{\tau_0} |\dot{\gamma}(\tau)|_{g_{i, s_0-\tau}} d\tau + C \\
\leq& C\left(  \int_0^{\tau_0} \tau^{-\frac{1}{2}} d\tau \right)^{\frac{1}{2}} \left( \int_0^{\tau_0} \tau^{\frac{1}{2}} \left( |\dot\gamma (\tau)|_{g_{i, s_0-\tau}}^2 + R_{g_i}(\gamma(\tau), s_0-\tau) + C \right) d \tau \right)^{\frac{1}{2}} + C \\
\leq& C\left( \cL (\gamma) + C \right)^{\frac{1}{2}} + C \leq C , \\
\end{split}
\]
hence we conclude that $v_i ( y_0, s_0 )\leq C$. This completes the proof.
\end{proof}


\medskip
\subsection{Local short time distance distortion estimate}

Next, we have the following short time distance distortion estimate, whose proof is modeled on \cite[Theorem 1.1]{BZ17}.
\begin{proposition}\label{stdde}
For any $\eta\in (0, 1)$, $A, D<\infty$, there exist constants $\delta\in (0, \eta)$, $C<\infty$, both depending on $\eta, A, D$, such that the following statements hold on the Ricci flow $(M_i, (g_{i,t})_{t\in [-T_i , 0]})$.
Suppose $(x_0, t_0)\in M_i\times [-T_i+10\eta, 0]$ satisfies that $v_i(x_0, t_0)\leq A$ then for any $y_0\in B_{g_i}( x_0, t_0, D )$, we have
$$
d_{g_{i,t}}( y_0 , x_0 ) \leq C,
$$
for all $t\in [ t_0-\delta , \min\left\{t_0+\delta, 0\right\} ]$.
\end{proposition}
\begin{proof}

%
Throughout this proof, all the constants depend at most on $\eta, A, D$. All the times we consider in this proof is in $[-T_i+ \eta , 0]$, hence we have $R_{g_i}\geq -C$ when we need the lower scalar curvature bound.

Let $(z_0, t_0-\eta)$ be an $\ell_{n}$-center of $(x_0, t_0)$, hence we have
$$\ell_{(x_0, t_0)}(z_0, t_0-\eta)\leq n.$$
We denote by $K_i(x,t; y,s)$, $s<t$ the heat kernel along the flow $g_{i,t}$. Then we consider the function $K(x,t):=K_i(x,t;z_0, t_0-\eta)$, which satisfies that $\partial_t K=\Delta_{g_{i}} K$. Then we have
$$\frac{d}{dt} \int_{M_i} K(\cdot, t)d g_{i,t}=\int_{M_i} (\Delta_{g_{i}} K(\cdot, t) - R_{ g_{i}}(\cdot, t)K(\cdot, t))d g_{i,t}\leq C\int_{M_i} K(\cdot, t)d g_{i,t}. $$
Hence for $t\in [t_0-\frac{\eta}{2}, \min\left\{t_0+\frac{\eta}{2}, 0\right\} ]$, we have
\begin{equation}\label{ibofhk1}
\int_{M_i} K(\cdot, t)d g_{i,t}\leq e^{C(t-(t_0-\eta))}\leq C.
\end{equation}
Also, for all $t\in [t_0-\frac{\eta}{2}, \min\left\{t_0+\frac{\eta}{2}, 0\right\} ]$, by \cite{ZhQ11}, we have
\begin{equation}\label{ubofhk1}
K(\cdot, t)\leq B_1,
\end{equation}
on $M_i$, for some constant $B_1<\infty$. Hence we can apply \cite[Theorem 3.2]{ZhQ06} to obtain that
\begin{equation}\label{gbofhk1}
\left| \nabla \sqrt{\ln  \frac{B_1}{K(\cdot, t)} } \right|_{ g_{i,t}} \leq \sqrt{\frac{1}{ t - (t_0-\frac{\eta}{4}) }}\leq C,
\end{equation}
for all $t\in [t_0-\frac{\eta}{8}, \min\left\{t_0+\frac{\eta}{8}, 0\right\} ]$. Let $\gamma:[0,1]\to M_i$ be a $ g_{i,t_0}$-minimizing geodesic connecting $(x_0, t_0)$ to $(y_0, t_0)$. Then by Perelman's Harnack estimate, we have
$$
K(\gamma(0), t_0)=K(x_0, t_0)\geq \frac{1}{(4\pi(t_0-(t_0-\eta)))^n}e^{-\ell_{(x_0, t_0)}(z_0, t_0-\eta)}\geq c_0,
$$
for some constant $c_0>0$. Hence from $d_{g_{i, t_0}}(x_0, y_0)\leq D$, we can integrate (\ref{gbofhk1}) at $t=t_0$ along $\gamma$ to obtain that
\begin{equation}\label{hllbalonggamma1}
K(\gamma(s), t_0)\geq c_1, ~~for ~~all~~s\in[0,1],
\end{equation}
for some constant $c_1>0$. Next we can apply \cite[Lemma 3.1]{BZ17} to obtain that
\begin{equation}\label{dbohk1}
|\partial_t K(\gamma(s), t)|\leq B_2(R_{ g_{i}}(\gamma(s), t)+C),
\end{equation}
for all $s\in[0,1]$, $t\in [t_0-\frac{\eta}{8}, \min\left\{t_0+\frac{\eta}{8}, 0\right\} ]$, for some constant $B_2<\infty$.

Now, by our assumption, $v_i(x_0, t_0)\leq A$, hence by Lemma \ref{viprop2}, we have $v_i(\gamma(s), t_0)\leq C$ for all $s\in[0,1]$. Hence by Lemma \ref{viprop2} again, we have
\begin{equation}\label{upbohvi}
v_i(\gamma(s), t)\leq C,
\end{equation}
for all $s\in[0,1]$, $t\in [t_0-\eta, 0 ]$. Hence we have $R_{g_{i}}(\gamma(s), t)\leq C$ for all $s\in[0,1]$, $t\in [t_0-\eta, 0 ] $. Inserting this estimate into (\ref{dbohk1}), we obtain
\begin{equation}\label{dbohk2}
|\partial_t K(\gamma(s), t)|\leq C,
\end{equation}
for all $s\in[0,1]$, $t\in [t_0-\frac{\eta}{8}, \min\left\{t_0+\frac{\eta}{8}, 0\right\} ] $. Using (\ref{hllbalonggamma1}), we obtain that, if we choose $\delta>0$ small enough, then we have
\begin{equation}\label{lbohk2}
K(\gamma(s), t) \geq c_1/2,
\end{equation}
for all $s\in[0,1]$, $t\in [t_0-\delta, \min\left\{t_0+\delta, 0\right\} ]$.
Now fix $t\in [t_0-\delta, \min\left\{t_0+\delta, 0\right\}]$. We let $Q\geq 1$ be maximal subject to the fact that there are parameters $0\leq a_1 <a_2<\dots< a_Q\leq 1$ such that the balls $B_{g_{i}}(\gamma(a_1),t,1)$, $\dots$, $B_{g_{i}}(\gamma(a_Q),t,1)$ are mutually disjoint. Then the balls $B_{g_{i}}(\gamma(a_1),t,2)$, ... $B_{g_{i}}(\gamma(a_Q),t,2)$ cover $\gamma([0,1])$. Hence we have $d_{ g_{i,t}}(x_0, y_0)\leq 4Q$. So we only need to bound $Q$.
First, from (\ref{gbofhk1}) and (\ref{lbohk2}), we have
\begin{equation}\label{lbohk3}
K(\cdot, t) \geq c_2, ~~on ~~ B_{ g_{i}}(\gamma(a_k),t,1),
\end{equation}
for all $1\leq k\leq Q$, for some constant $c_2>0$.

Next, from (\ref{upbohvi}), we can apply Lemma \ref{viprop2} again to obtain
\begin{equation}\label{upbohvi2}
v_i(\cdot , t)\leq C, ~~on ~~ B_{ g_{i}}(\gamma(a_k),t,1),
\end{equation}
for all $1\leq k\leq Q$. Hence we have $R_{g_{i}}(\cdot , t)\leq C, ~~on ~~ B_{ g_{i}}(\gamma(a_k),t,1)$ for all $1\leq k\leq Q$. Hence we can apply Perelman's volume non-collapsing estimate to obtain that
\begin{equation}\label{vnclpatak}
\vol_{ g_{i,t}}(B_{ g_{i}}(\gamma(a_k),t,1))\geq c_3,
\end{equation}
for all $1\leq k\leq Q$, for some constant $c_3>0$. Combining (\ref{ibofhk1}), (\ref{lbohk3}) and (\ref{vnclpatak}), we have
$$
C\geq \int_{M_i} K(\cdot, t)d g_{i,t}\geq \sum_{k=1}^{Q}\int_{B_{ g_{i}}(\gamma(a_k),t,1)} K(\cdot, t)d g_{i,t}\geq Q\cdot c_2 \cdot c_3,
$$
hence we have $Q\leq C$. This completes the proof of Proposition \ref{stdde}.
\end{proof}
%


\medskip
\subsection{Local heat kernel lower bound}

Next, we have the following heat kernel lower bound.
\begin{proposition}\label{hklb}
For any $\eta\in (0, 1)$, $A, D<\infty$, there exists constant $C<\infty$ depends on $\eta, A, D$, such that the following statements hold on the Ricci flow $(M_i, (g_{i,t})_{t\in [-T_i , 0]})$.
If $(x_0, t_0)\in M_i\times [-T_i+10\eta, 0]$ satisfies
$$
v_i(x_0, t_0)\leq A,
$$
then for any $s_0\in [ \max\left\{t_0-\eta^{-1}, -T_i+\eta\right\}  , t_0-\eta ]$ and $y_0\in B_{g_i}( x_0, s_0, D )$, we have
$$
K_i(x_0,t_0;y_0, s_0) \geq C^{-1},
$$
where $K_i(x,t; y,s)$, $s<t$ denotes the heat kernel along the flow $g_{i,t}$.
\end{proposition}
\begin{proof}

%
Throughout this proof, all the constants depend at most on $\eta, A, D$. All the times we consider in this proof is in $[-T_i+ \eta , 0]$, hence we have $R_{g_i}\geq -C$ when we need the lower scalar curvature bound. We let $0<\delta<\eta/2$ to be determined.

By our assumption, $v_i(x_0, t_0)\leq A$, hence by Lemma \ref{viprop2}, we have
\begin{equation}\label{upbohvi4}
v_i(x_0, t)\leq C,
\end{equation}
for all $t\in [ s_0 , t_0 ]$. Hence we have
$$
R_{g_{i}}(x_0, t)\leq C,
$$
for all $t\in [ s_0 , t_0 ]$. Consider the spacetime curve defined by $\gamma(\tau)=(x_0, t_0-\tau)$, $\tau\in [0, t_0-(s_0 + \delta)]$, then we have
$$\cL(\gamma)= \int_{0}^{t_0-(s_0 + \delta)} \tau^{1/2}R_{g_i}(x_0, t_0-\tau) d\tau \leq C.$$
Hence we can apply Lemma \ref{wdn} to obtain
\begin{equation}\label{dw1x0}
\begin{split}
&d^{g_{i,s_0 + \delta}}_{W_1} ( \nu_{x_0, t_0; s_0 + \delta} , \delta_{x_0} )  \\
\leq& C \left(1+ \frac{\cL(\gamma)}{2(t_0- (s_0 + \delta))^{\frac{1}{2}}} - \cN^*_{s_0 + \delta}(x_0, t_0) \right) ^{1/2} (t_0- (s_0 + \delta))^{1/2}\leq C.\\
\end{split}
\end{equation}
Next let $(w_0,s_0 + \delta)$ be an $H_{n}$-center of $(x_0,t_0)$, then from (\ref{dvh}) we have
\begin{equation}\label{dvh2}
d^{g_{i,s_0 + \delta}}_{W_1}( \nu_{x_0, t_0; s_0 + \delta} , \delta_{w_0} ) \leq \sqrt{\var( \nu_{x_0, t_0; s_0 + \delta} , \delta_{w_0} )} \leq \sqrt{H_{n}(t_0-(s_0 + \delta))}\leq C.
\end{equation}
Combining (\ref{dw1x0}) and (\ref{dvh2}) and the triangle inequality, we obtain
\[
\begin{split}
d_{g_{i, s_0 + \delta}}&(x_0, w_0)=d^{g_{i, s_0 + \delta}}_{W_1} ( \delta_{x_0} , \delta_{w_0} )\\
&\leq d^{g_{i, s_0 + \delta}}_{W_1} ( \delta_{x_0} , \nu_{x_0, t_0; s_0 + \delta} ) + d^{g_{i, s_0 + \delta}}_{W_1} ( \nu_{x_0, t_0; s_0 + \delta} , \delta_{w_0} ) \leq  C,
\end{split}
\]
hence we can find a constant $B_1$ large enough, such that
\begin{equation}\label{bw0inbxo}
B_{g_i}( w_0, s_0 + \delta, \sqrt{2H_{n}(t_0-(s_0 + \delta))} )\subset B_{g_i}( x_0, s_0 + \delta, B_1) .
\end{equation}
Next, we need the following lemma.
\begin{lemma}\label{hklbaty0}
There exists constant $C<\infty$, such that
$$
K_i(y_0,s_0 + \delta ;y_0, s_0) \geq C^{-1}.
$$
\end{lemma}
\begin{proof}
First, by (\ref{upbohvi4}) we have $v_i(x_0, s_0)\leq C$. Hence by Lemma \ref{viprop2} again,  we have $v_i(y_0, s_0)\leq C$. Hence by Lemma \ref{viprop2} again, we have
\begin{equation}\label{upbohvi3}
v_i(y_0, t)\leq C,
\end{equation}
for all $t\in [s_0, s_0+\delta]$. Hence we have $R_{g_{i}}(y_0, t)\leq C$ for all $t\in [s_0, s_0+\delta]$. Now let $\gamma$ be the spacetime curve defined by $\gamma(\tau)=(y_0, s_0+\delta-\tau)$, $\tau\in[0, \delta]$, then we can compute
\begin{equation}\label{eofrl1}
\begin{split}
\mathcal{L}(\gamma)= \int_{0}^{\delta}\sqrt{\tau}R_{g_i}(y_0, s_0+\delta-\tau)d\tau \leq \int_{0}^{\delta}\sqrt{\tau}\cdot Cd\tau \leq C\delta^{3/2},
\end{split}
\end{equation}
hence we have
$$
\ell_{(y_0,s_0 + \delta)}(y_0, s_0) \leq \frac{\mathcal{L}(\gamma)}{2\sqrt{\delta}} \leq C.
$$
Then by Perelman's Harnack estimate, we have
$$
K_i(y_0,s_0 + \delta ;y_0, s_0) \geq \frac{1}{(4\pi(s_0 + \delta-s_0))^n}e^{-\ell_{(y_0,s_0 + \delta)}(y_0, s_0)}\geq C^{-1},
$$
which proves the lemma.
\end{proof}
Now, for all $t\in [s_0+\frac{\delta}{2}, s_0+\delta]$, by \cite{ZhQ11}, we have
\begin{equation}\label{ubofhk2}
K_i(\cdot, t ;y_0, s_0)\leq B_2,
\end{equation}
on $M_i$, for some constant $B_2=B_2(n,\omega_0)<\infty$. Hence we can apply \cite[Theorem 3.2]{ZhQ06} to obtain that
\begin{equation}\label{gbofhk2}
\left| \nabla \sqrt{\ln  \frac{B_2}{K_i(\cdot, t ;y_0, s_0)} } \right|_{ g_{i,t}} \leq \sqrt{\frac{1}{ t - (s_0+\frac{\delta}{2})}}\leq C,
\end{equation}
for all $t\in [s_0+\frac{3\delta}{4}, s_0+\delta]$. Since $y_0\in B_{g_i}( x_0, s_0, D )$, by (\ref{upbohvi4}) and Proposition \ref{stdde}, we have
$$
d_{g_{i,s_0 + \delta}}( y_0 , x_0 ) \leq C,
$$
if we choose $\delta>0$ small enough. Hence for any $y\in B_{g_i}( x_0, s_0 + \delta, B_1 )$,
\begin{equation}\label{dbbybx0}
d_{g_{i,s_0 + \delta}}( y , y_0 ) \leq d_{g_{i,s_0 + \delta}}( y , x_0 ) + d_{g_{i,s_0 + \delta}}( x_0 , y_0 )\leq B_1 + C \leq C.
\end{equation}
Combining (\ref{gbofhk2}) and (\ref{dbbybx0}) with Lemma \ref{hklbaty0}, we have
\begin{equation}\label{hklbatbx0}
K_i(y ,s_0 + \delta ; y_0, s_0) \geq C^{-1} ,
\end{equation}
for all $y\in B_{g_i}( x_0, s_0 + \delta, B_1 )$. Now, by the reproduction formula, (\ref{bw0inbxo}), (\ref{hklbatbx0}), we have
\begin{eqnarray*}
&&K_i(x_0 ,t_0 ; y_0, s_0) \\
&=& \int_{M_i} K_i(x_0 ,t_0 ; y, s_0+ \delta)K_i(y , s_0 + \delta ; y_0, s_0) dg_{i,s_0 + \delta}(y) \\
&\geq& \int_{B_{g_i}( x_0, s_0 + \delta, B_1 )} K_i(x_0 ,t_0 ; y, s_0+ \delta)K_i(y , s_0 + \delta ; y_0, s_0) dg_{i,s_0 + \delta}(y) \\
&\geq& C^{-1} \int_{B_{g_i}( x_0, s_0 + \delta, B_1 )} K_i(x_0 ,t_0 ; y, s_0+ \delta) dg_{i,s_0 + \delta}(y) \\
&\geq& C^{-1} \nu_{x_0, t_0; s_0 + \delta} \left( B_{g_i}\left( w_0, s_0 + \delta, \sqrt{2H_{n}(t_0-(s_0 + \delta))} \right) \right) \\
&\geq& C^{-1},
\end{eqnarray*}
where in the last inequality we have used Lemma \ref{hnc'2}.
This completes the proof of Proposition \ref{hklb}.
\end{proof}
%


\medskip
\subsection{Proof of Theorem \ref{mainforRF2}} \label{proofofmainrf2}

Now, we can prove Theorem \ref{mainforRF2}.
\begin{proof}[\bf{Proof of Theorem \ref{mainforRF2}}]
According to \cite{Bam20b}, for almost every $t\in (-T_\infty, 0)$, (that is, except a countable set of times), we have the metric measure space $(M_{i}, d_{g_{i,t}}, \nu_{p_i,0; t})$ converge to $(\cX_t, d_t, \nu_{p_\infty; t})$ in the Gromov-$W_1$-Wasserstein topology. We fix one such time $t_0\in (-T_\infty, 0)$. Let $\eta\in (0,1)$ such that $t_0-10\eta>-T_\infty$.

Let $D<\infty$ be any constant, and let $y_0\in B_{g_i}(p_i, t_0, D)$ for all $i$ large enough (so that $t_0\in ( -T_i + 5\eta , 0 ]$). Consider one such fixed $i$.

We have assumed that $v_i (p_i, 0)\leq Y_0$, hence by Lemma \ref{viprop2}, we have $v_i(p_i, t_0)\leq C(t_0)$. But for any $r\in (0,1)$, we have $d_{g_{i,t_0}}(y,p_i)\leq D+1$ for all $y\in B_{g_i}(y_0, t_0, r)$. Hence by Lemma \ref{viprop2} again, we have
$$
v_i(y, t_0)\leq C,
$$
for all $y\in B_{g_i}(y_0, t_0, r)$, for some $C=C(t_0, D)<\infty$. Hence we have
$$
R_{g_{i}}(y, t_0)\leq C,
$$
for all $y\in B_{g_i}(y_0, t_0, r)$. Hence we can apply Perelman's volume non-collapsing estimate to obtain that
\begin{equation}\label{vnclpatak2}
\vol_{ g_{i,t_0}}(B_{g_i}(y_0, t_0, r))\geq C^{-1}r^{n}.
\end{equation}
On the other hand, by Proposition \ref{hklb}, we have
\begin{equation}\label{hklb2}
K_i(p_i, 0; y, t_0) \geq C^{-1} ,
\end{equation}
for all $y\in B_{g_i}(y_0, t_0, r)$, for some constant $C=C(t_0, D)<\infty$. Combining (\ref{vnclpatak2}) and (\ref{hklb2}), we conclude that
$$
\nu_{p_i,0; t_0} (  B_{g_i}(y_0, t_0, r)  ) = \int_{B_{g_i}(y_0, t_0, r)} K_i(p_i, 0; y, t_0) dg_{i, t_0} \geq C^{-1}r^{n} ,
$$
for some constant $C=C(t_0, D)<\infty$. Hence by \cite[Proposition 2.7]{Hal}, we conclude that, by passing to a subsequence, we have $(M_{i}, d_{g_{i,t_0}}, p_i)$ converge to $(\cX_{t_0}, d_{t_0}, q_{t_0})$ in the Gromov-Hausdorff topology, for some $q_{t_0}\in \cX_{t_0}$. This completes the proof of Theorem \ref{mainforRF2}.
\end{proof}
%


\medskip
\subsection{Proof of   (1) in Theorem \ref{2main1}}

We now come back to the set-up of Section \ref{setupofKRF}. Given any sequence of times $t_i\nearrow 1$ in the normalized flow, let $s_i=-\ln (1-t_i)\to\infty$ as $i\to\infty$. Let $B_0=B(n , \omega_0, \|\rho \|_{C^2(\omega_{\fs})}, 0)$ be the constant from Lemma \ref{a'3}, then we define
\begin{equation}\label{cordob2}
b_i(s) = e^{s-s_i}a(s_i)-B_0,
\end{equation}
where $s$ is the time parameter in the normalized flow $(X, \tilde g(s))$. By Lemma \ref{a'3}, we have $b_i(s)\leq a(s)$ for all $s\in [0 , s_i]$. We then denote by
\begin{equation}\label{cordovsi}
v_{i}=u-b_i(s)+1,
\end{equation}
which is a smooth function on the normalized flow $(X, \tilde g(s))$, $s\in [0, \infty)$. Then we have $v_i\geq 1$ on $X\times [0 , s_i]$. According to Theorem \ref{main1}, we have the following gradient and Laplacian estimates
\begin{equation} \label{gleou0}
\frac{ \left| \Delta u \right| }{ u - a + 1 } + \frac{ |\nabla u|^2}{ u - a + 1 } \leq C ,
\end{equation}
on the normalized flow $(X, \tilde g(s))$, $s\in [0, \infty)$. Hence for $v_i$, by Lemma \ref{a'3} and the parabolic Schwarz lemma, we have
\begin{equation} \label{gleovi}
\frac{ \left| \partial_s v_i \right| }{ v_i } + \frac{ \left| \Delta v_i \right| }{ v_i } + \frac{ |\nabla v_i |^2}{ v_i } \leq C ,
\end{equation}
on $X\times [0 , s_i]$. Now, recall from the unnormalized flow $(X, g(t))$, $t\in [0, 1)$, we define $M_i=X$ and $g_{i,t}:=(1-t_i)^{-1}g((1-t_i)t+t_i)$, $t\in [-T_i , 0]$ with $T_i=t_i/(1-t_i)$. Hence we can compute that
$$
\omega_{i,t} = (1-t) \tilde \omega(s(t)), ~~ s(t)= -\ln (1-t) -\ln (1-t_i), ~~t\in [-T_i , 0] .
$$
For the convenience of the notations, we still denote by $v_{i}(t)=v_{i}(s(t))$, where $s(t)=-\ln (1-t) -\ln (1-t_i)$, which makes $v_i$ a function on the Ricci flow $(M_i, (g_{i,t})_{t\in [-T_i , 0]})$, hence we have $v_i\geq 1$ on $M_i\times [-T_i, 0]$. From (\ref{gleovi}), on $(M_i, (g_{i,t})_{t\in [-T_i , 0]})$, we have
\begin{equation} \label{gleovi2}
\frac{ \left| \partial_t v_i \right| }{ v_i } + \frac{ \left| \Delta v_i \right| }{ v_i } + \frac{ |\nabla v_i |^2}{ v_i } \leq \frac{C}{1-t} \leq C.
\end{equation}
In conclusion, if we let $p_i$ be the Ricci vertex associated to $\theta_Y$ at $t_i=1-e^{-s_i}$, then we have the following estimates.
\begin{lemma}[] \label{viprop}
There exists constant $C=C(n , \omega_0, \|\rho \|_{C^4(\omega_{\fs})})<\infty$, such that the following statements hold on the Ricci flow $(M_i, (g_{i,t})_{t\in [-T_i , 0]})$.
\begin{enumerate}
\item $v_i\geq 1$;
\item $\frac{|\partial_t v_i|}{v_i} + \frac{| \Delta v_i |}{v_i} + \frac{|\nabla v_i|^2}{v_i} \leq C$;
\item $R_{g_{i}} \leq Cv_i$;
\item $v_i (p_i, 0)= B_0 + 1$.
\end{enumerate}
Here all the operators are with respect to the metric $g_{i,t}$.
\end{lemma}
Now we can finish the proof of item (1) of Theorem \ref{2main1}.

\begin{proof}[Proof of Theorem \ref{2main1}, item (1)]
Let $\theta_Y\in \vartheta$ be the given smooth closed $(1,1)$-form on $Y$, and let $p_i\in X$ be a Ricci vertex associated with $\theta_Y$ at $t_i$. According to \cite{HJST}, the limiting metric flow $\cX$ is continuous in the Gromov-$W_1$-sense, hence by the results in \cite[Section 7]{Bam20b}, for every $t<0$, the $\bF$-convergence (\ref{FcorKRF'1}) is time-wise at $t$. According to Lemma \ref{viprop}, we conclude that $v_i$ is a $C$-barrier of $R_{g_{i}}$ and is $2B_0$-based at $ (p_i, 0) $. Hence item (1) of Theorem \ref{2main1} follows immediately from Theorem \ref{mainforRF2}.
\end{proof}


\bigskip


\section{Complex geometric compactness}\label{cgc}

In this section, we will prove Theorem \ref{2main1}. The partial $C^0$-estimate introduced in \cite{T90} is a fundamental scheme to obtain quantitive and effective Kodaira embeddings with geometric and algebraic estimates for the Bergman metrics. We will apply the techniques developed in \cite{T90, DS1, DS2, T13, T15} to  establish the partial $C^0$-estimate for finite time solutions of the K\"ahler-Ricci flow using the gradient and Laplace estimates in Theorem \ref{main1}.

Consider any sequence $t_i \rightarrow 1$ and let $(X, (g_{i, t})_{t\in [-T_i, 0]})$ be the blow-up of $(X, (g(t))_{t\in [0, 1)})$ by
$$ g_{i, t}= (1-t_i)^{-1} g((1-t_i) t + t_i), ~ t\in [-T_i, 0), ~ T_i = \frac{t_i}{1-t_i}. $$

If we let $p_i$ be a Ricci vertex at $g(t_i)$ with respect to a fixed $\theta_Y$,  by Theorem \ref{mainforRF2}
$\{ (X, g_{i, t}, p_i) \}_{i=1}^\infty $ $\mathbb{F}$-converge to $(\cX, (\nu_{p_\infty,t} )_{t\in (-\infty, 0]})$ and for every $t\in (-\infty, 0)$ the convergence is also in pointed Gromov-Hausdorff topology. The tangent flow of $(\cX, (\nu_{p_\infty,t} )_{t\in (-\infty, 0]})$ at each time $t$ must be a Ricci-flat cone with closed singular set of Hausdorff dimension no greater than $2n-4$ because the scalar curvature vanishes for the tangent flow, which will have to be static and Ricci-flat. More precisely, we have the following proposition for the $\mathbb{F}$-convergent sequence $(X, g_{i, t}, p_i)$ (cf. \cite{HJST}).

\begin{proposition} \label{hjst1}
For each $t\in (-\infty, 0)$, $\{ (X, g_{i, t}, p_i) \}_{j=1}^\infty $ chosen above converge  in pointed Gromov-Hausdorff topology to a complete metric space $(\cX_t, d_t, q_t)$.  Furthermore, the blow-ups of $(\cX_t, d_t)$ at any point $p\in (\cX_t, d_t) $ converge sequentially in pointed Gromov-Hausdorff topology to a Ricci-flat tangent cone. The convergence is smooth to a K\"ahler metric on the regular part of $\cX_t$ and the tangent cones off closed singular sets of Hausdorff dimension no greater than $2n-4$.
\end{proposition}

For each fixed $t\in (-\infty, 0]$, the K\"ahler metric $g_{i, t}$ lies in the K\"ahler class
$$(1-t) [ - K_X + (1-t)^{-1}(1-t_i)^{-1} L_Y], $$
where $L_Y$ is the ample line bundle defined by $L_0+K_X$. Let $m_i$ be the smallest integer no less than $ (1-t)^{-1}(1-t_i)^{-1}$ and
$$\delta_i = m_i -  (1-t)^{-1} (1-t_i)^{-1}  \in [0, 1). $$
Then we define the line bundle $L_i$ by
$$L_i = -K_X + m_i L_Y = -K_X + (1-t)^{-1}(1-t_i)^{-1} L_Y + \delta_i L_Y. $$

Let $\rho$ be the Fubini-Study potential for $\omega_Y$. In particular, we can assume
$$\rho = - m_Y^{-1}\log |\sigma_Y|^2_{h_Y^{m_Y}} $$
for some global section $\sigma_Y\in H^0(Y, m_Y L_Y)$. Therefore $- \ddbar\rho= -\omega_Y + m_Y^{-1} [\Delta_{\sigma_Y}]$, where $\Delta_{\sigma_Y}$ is the support of $\sigma_Y$. We let $h_i$ be the hermitian metric on $L_i$ defined by
\begin{equation}
h_i = (h_Y)^{m_i}  \Omega e^{-(1-t)^{-1} (1-t_i)^{-1} \phi(\cdot, (1-t_i)t+ t_i) + \delta_i \rho},
\end{equation}
where $\phi(\cdot, t)$ is given by (\ref{mauflow2}) as the quasi-K\"ahler potential of $g(t)$ of the unnormalized K\"ahler-Ricci flow.
Then
\begin{eqnarray*}
&&Ric(h_i) \\
&=& -\ddbar \log h_i \\
&=&  (1-t)^{-1}(1-t_i)^{-1} \omega_Y + \chi  + (1-t)^{-1} (1-t_i)^{-1} \ddbar \phi (\cdot, (1-t_i)t+t_i)+ \delta_i m_Y^{-1} [\Delta_{\sigma_Y}] \\
&=& (1-t)^{-1} g_{i, t} +\delta_i m_Y^{-1} [\Delta_{\sigma_Y}],
\end{eqnarray*}
where $h_Y$ is the hermitian metric on $L_Y$ with $-\ddbar \log h_Y = \omega_Y$.
We also have
\begin{eqnarray*}
Ric(g_{i, t}) - (1-t)^{-1}g_{i, t} &=&  -\ddbar v_i - (1-t_i)^{-1}\theta_Y\geq - \ddbar v_i - Bg_{i,t},
\end{eqnarray*}
for some uniform $B>0$ by the parabolic Schwarz lemma, where $v_i$ is the Ricci potential associated to $\theta_Y$.

We will now fix $t =-1$ without loss of generality, and oppress it in the following notations.  We will also write $\hat h= h_i^k$ and $\hat g=k  g_{i, -1} $ for conveniences. In particular,
$$Ric(\hat h) = \frac{1}{2} \hat g$$
on $X\setminus \Delta_{\sigma_Y}$.  For $\sigma\in (L_i)^k$, we define  the scaled norms by
$$|\sigma|^2 = |\sigma|^2_{\hat h} =|\sigma|^2_{h_i^k}$$
and
$$ | \nabla \sigma|^2 = | \nabla \sigma|^2_{\hat h, \hat g } = | \nabla \sigma|^2_{h_i^k, 2^{-1} kg_{i, -1}}, $$
where $\nabla$ is the  gradient  associated to  $\hat h$ and $\hat g$ for conveniences.
By our gradient estimate for $v_i$, there exists $C>0$ such that for any $x\in X$ and $k\geq 1$, we have
$$|\nabla v_i|(x, -1) \leq Ck^{-1/2} (1+ k^{-1/2} d_{\hat g}(x, p_i)). $$

\begin{lemma} For any $\sigma \in H^{0}(X, (L_i)^k)$, we have

$$\Delta |\sigma|^2 = |\nabla \sigma|^2 - \frac{n}{2} |\sigma |^2,$$

$$\Delta |\nabla \sigma|^2 = |\nabla\nabla \sigma|^2 - \frac{n}{2} |\nabla \sigma|^2 + \frac{n}{4}| \sigma|^2 + Ric(\nabla \sigma, \overline{\nabla \sigma}), $$
where $\Delta$ and $Ric$ are the Laplacian operator and Ricci curvature associated to $\hat h$ and $\hat g$.

\end{lemma}

\begin{proof} We will use the holomorphic normal coordinates at a fixed point, where the gradient of $h$ also vanishes. Straightforward calculations give
\begin{eqnarray*}
\Delta |\sigma|^2 &=&  \hat g^{i\bar j} ( \sigma \overline{\sigma} \hat h )_{i\bar j}\\
&=&|\nabla \sigma|^2 - \frac{n}{2} |\sigma|^2,
\end{eqnarray*}
\begin{eqnarray*}
\Delta |\nabla \sigma|^2 &=&  \hat g^{i\bar j} \left(  \hat g^{k\bar l} \left( \sigma_k - \sigma (\log \hat h)_k \right)  \left(\overline{\sigma}_{\bar l} - \overline{\sigma} (\log \hat h)_{\bar l}\right) \hat h \right)_{i\bar j}\\
&=&|\nabla^2 \sigma|^2  - \frac{n}{2} |\nabla \sigma|^2+ Ric(\nabla \sigma, \overline{\nabla \sigma})+ \frac{n}{4} |\sigma|^2.
\end{eqnarray*}
This completes the proof.
\end{proof}

\begin{corollary} There exists $B>0$ such that for any $k\geq 1$, $i\geq 1$, $\sigma \in H^0(X, (L_i)^k)$, we have
\begin{equation}
\Delta |\sigma| \geq - \frac{n}{4} |\sigma| ,
\end{equation}
\begin{equation} \label{l2lap}
\Delta |\nabla \sigma| \geq  \frac{\nabla^2 v_i (\nabla\sigma, \overline{\nabla\sigma})}{2|\nabla\sigma|} +\frac{|\nabla^2\sigma|^2}{8 |\nabla \sigma|}- Bn |\nabla \sigma|.
\end{equation}

\end{corollary}

\begin{proof} Straightforward calculations show that

\begin{eqnarray*}
\Delta |\sigma| &=& \frac{\Delta |\sigma|^2}{2|\sigma|}- \frac{|\nabla |\sigma|^2 |^2}{4|\sigma|^3}\\
&=& \frac{|\nabla \sigma|^2}{2|\sigma|} - \frac{n}{4}|\sigma| - \frac{|\nabla \sigma|^2}{4|\sigma|}\\
&\geq& - \frac{n}{4} |\sigma|,
\end{eqnarray*}
and
\begin{eqnarray*}
&&\Delta |\nabla \sigma| \\
&=& \frac{\Delta |\nabla\sigma|^2}{2|\nabla \sigma|}- \frac{|\nabla (|\nabla\sigma|^2) |^2}{4|\nabla \sigma|^3}\\
&=& \frac{|\nabla^2 \sigma|^2 + Ric(\nabla \sigma, \overline{\nabla \sigma}) - \frac{n}{2}|\nabla \sigma|^2 + \frac{n}{4} |\sigma|^2}{2|\nabla\sigma|} - \frac{|\nabla (|\nabla\sigma|^2)|^2}{4|\nabla\sigma|^3}\\
&\geq& \frac{|\nabla^2 \sigma|^2 + Ric(\nabla \sigma, \overline{\nabla \sigma}) -\frac{n}{2} |\nabla \sigma|^2 + \frac{n}{4}  |\sigma|^2}{2|\nabla\sigma|}  - \frac{|\nabla^2 \sigma|^2|\nabla \sigma|^2 + |\sigma||\nabla \sigma|^2 |\nabla^2\sigma|+ \frac{1}{4} |\sigma|^2|\nabla\sigma|^2}{4|\nabla\sigma|^3}\\
&\geq& \frac{ Ric(\nabla \sigma, \overline{\nabla \sigma}) }{2|\nabla\sigma|}+ \frac{ |\nabla^2 \sigma|^2}{4|\nabla\sigma|} + \frac{(2n-1)|\sigma|^2}{16|\nabla\sigma|}- \frac{  |\sigma| |\nabla^2\sigma|}{4|\nabla\sigma|} - \frac{n|\nabla \sigma|}{4}\\
&\geq& \frac{ Ric(\nabla\sigma, \overline{\nabla\sigma})}{2|\nabla\sigma|} +\frac{|\nabla^2\sigma|^2}{8 |\nabla \sigma|}- \frac{n|\nabla \sigma|}{4}.
\end{eqnarray*}

The proof is completed with the observation that $Ric(\hat g) \geq  -B \hat g - \nabla^2 v_i. $
\end{proof}

\begin{proposition} For any $R\geq 1$, there exists $K>0$ such that for any $k\in \mathbb{Z}^+$ and $\sigma \in H^0(X, (L_i)^k)$, we have
\begin{equation}
||\sigma||_{L^\infty(B (p_i, R))} \leq K ||\sigma||_{L^2(B (p_i, 2R))}.
\end{equation}
\begin{equation}
||\nabla \sigma||_{L^\infty(B (p_i, R))} \leq K ||\sigma||_{L^2(B  (p_i, 2R))},
\end{equation}
where $B(p_i, R)=B_{\hat g}(p_i, R)$.

\end{proposition}

\begin{proof} Let $F_{r, R}$ be the standard nonnegative cut-off function such that $F_{r, R}=1$ on $[0, r]$ and $F_{r, R}=0$ on $[R, \infty)$ with
$$|F'_{r, R}| \leq A(R-r)^{-1}, ~ R-r\geq 1.$$
for some fixed $A>0$. We then let %
$$\eta(x) = F_{r,R}(d_{\hat g}(p_i, x)).$$
Immediately, we have
$$|\nabla \eta|^2+ |\Delta \eta| \leq A(R-r)^{-2}.$$ %
Then we have
\begin{eqnarray*}
&&\int_X |\nabla (\eta |\sigma|^{(p+1)/2})|^2 dV\\
&\leq& 2\int_X |\nabla \eta|^2 |\sigma|^{p+1} dV + 2\int_X \eta^2 |\nabla |\sigma|^{(p+1)/2}|^2 dV  \\
%
&\leq& A(R-r)^{-2} \int_{B(p_i, R)}|\sigma|^{p+1}dV  - 2\int_X \eta^2 |\sigma|^{(p+1)/2}\Delta|\sigma|^{(p+1)/2}dV \\
&&+ \frac{1}{2}\int_X |\nabla (\eta|\sigma|^{(p+1)/2})|^2 dV  + 32\int_X |\nabla\eta|^2|\sigma|^{p+1}dV \\
&\leq& Cp(R-r)^{-2}\int_{B(p_i,R)}|\sigma|^{p+1} dV + \frac{1}{2}\int_X |\nabla (\eta |\sigma|^{(p+1)/2})|^2 dV ,
\end{eqnarray*}
where $dV= dV_{\hat g}.$
Therefore we have
$$\int_X |\nabla (\eta |\sigma|^{(p+1)/2})|^2 dV  \leq Cp(R-r)^{-2}\int_{B(p_i, R)}|\sigma|^{p+1} dV .$$
Since the scalar curvature is uniformly bounded in $B(p_i, R)$, the Sobolev inequality holds uniformly for all smooth functions with compact support in $B(p_i, R)$ (cf. \cite[Theorem 1.5]{Ye}). After applying the Sobolev inequality, we have
$$|| |\sigma|^{(p+1)/2}||^2_{L^{\frac{2n}{n-1}}(B(p_i, r))} \leq Cp(R-r)^{-2} \int_{B(p_i, R)} |\sigma|^{p+1} dV,$$
or
$$||\sigma||_{L^{\frac{n(p+1)}{n-1}}(B(p_i, r))}\leq \left(\frac{Cp}{(R-r)^2}\right)^{\frac{1}{p+1}} ||\sigma||_{L^{p+1}(B(p_i, R))}.$$
We now let $\beta=\frac{n}{n-1}$ and $r_l= r + \frac{R-r}{2^l}$, $p+1=2\beta^l$.
Then
\begin{eqnarray*}
||\sigma||_{L^{2\beta^{l+1}}(B(p_i, r_{l+1}))} &\leq& \left(\frac{C\beta^l}{(r_l-r_{l+1})^2}\right)^{\frac{1}{2\beta^l}} ||\sigma||_{L^{2\beta^l}(B(p_i, r_l))} \\
&\leq & \left(C2^l\beta^l\right)^{\frac{1}{2\beta^l}} ||\sigma||_{L^{2\beta^l}(B(p_i, r_l))}.
\end{eqnarray*}
Hence
\begin{eqnarray*}
||\sigma||_{L^{\infty}(B (p_i, r))} &\leq& ||\sigma||_{L^2(B(p_i, R))} \exp\left(\log (2C\beta) \sum_{l=1}^\infty \frac{l}{2\beta^l}   \right) \\
&\leq &C(r,R) ||\sigma||_{L^2(B(p_i, R))}.
\end{eqnarray*}
Similar calculations show that for $p>1$
\begin{eqnarray*}
&&\int_X \left|\nabla \left(\eta |\nabla\sigma|^{\frac{p+1}{2}} \right)\right|^2 dV \\
&\leq& 2\int_X |\nabla \eta|^2 |\nabla \sigma|^{p+1} dg + 2\int_X \eta^2 |\nabla \left(|\nabla\sigma|^{(p+1)/2}\right)|^2 dV \\
&\leq& A(R-r)^{-2} \int_{B(p_i, R)}|\nabla\sigma|^{p+1}dV -\int_X \eta^2 |\nabla\sigma|^{(p+1)/2}\Delta\left(|\nabla\sigma|^{(p+1)/2}\right) dV\\
&&+ \frac{1}{2}\int_X |\nabla \left(\eta|\nabla\sigma|^{(p+1)/2}\right)|^2 dV + 32\int_X |\nabla\eta|^2|\nabla\sigma|^{p+1}dV\\
&\leq& Cp(R-r)^{-2}\int_{B(p_i, R)}|\nabla\sigma|^{p+1} dV+ \frac{1}{2}\int_X |\nabla (\eta |\nabla \sigma|^{(p+1)/2})|^2 dV \\
&& - \frac{p+1}{16} \int_X \eta^2  |\nabla \sigma|^{p-1} \left( 4\nabla^2 v_i(\nabla\sigma, \overline{\nabla \sigma}) +|\nabla^2 \sigma|^2 - 16Bn |\nabla \sigma|^2  \right)dV \\
&\leq& C(1+ \sup_{B(p_i, R)} |\nabla v_i|) p(R-r)^{-2}\int_{B(p_i, R)}|\nabla\sigma|^{p+1} dV+ \frac{1}{2}\int_X |\nabla (\eta |\nabla \sigma|^{(p+1)/2})|^2 dV\\
&\leq& Cp(1+ R)(R-r)^{-2}\int_{B(p_i, R)}|\nabla\sigma|^{p+1} dV+ \frac{1}{2}\int_X |\nabla (\eta |\nabla \sigma|^{(p+1)/2})|^2 dV,
\end{eqnarray*}
where we apply integration by parts and estimate (\ref{l2lap}).

We can apply the Sobolev inequality and derive the following estimate.

$$ || \nabla \sigma  ||_{L^{\frac{n(p+1)}{n-1}}(B(p_i, r))} \leq  \left( \frac{Cp (1+R)}{(R-r)^2} \right)^{\frac{1}{p+1}} || \nabla \sigma ||_{L^{p+1}(B(p_i, R))}.$$
By similar calculations for the $L^2$-estimate of $|\sigma|$, we can apply Moser's iteration to derive
$$||\nabla \sigma||_{L^\infty(B(p_i, r))} \leq C(r, R) ||\sigma||^2_{L^2(B(p_i, R))}.$$
This completes the proof.
\end{proof}

The following corollary immediately follows by the bound of the Ricci potential $v_i$.

\begin{corollary} If we let $\tilde h = \hat h e^{-v_i}$, then for and $R\geq 1$, there exists $K=K(R)>0$ such that for all $k\geq K$ and $\sigma \in H^0(X,  (L_i)^k)$, we have
\begin{equation}
||\sigma||_{\tilde h,  L^\infty(B(p_i, R) )} \leq K ||\sigma||_{\tilde h_, L^2(B(p_i, 2R), \hat g)}.
\end{equation}
\begin{equation}
||\nabla\sigma||_{\tilde h,  L^\infty(B(p_i, R) )} \leq K ||\sigma||_{\tilde h_, L^2(B(p_i, 2R), \hat g)}.
\end{equation}
\end{corollary}

For the global $L^2$-estimate, the Ricci curvature will appear and we do not have control of the Ricci curvature for $g_{i, t}$. However, it will help us eliminate the Ricci curvature by adding the twist $e^{-v_i}$ and establish the following $L^2$-estimate.

\begin{proposition} There exist $K>0$ and $A>0$ such that for all $i\geq 1$, $k\geq K$ and any $(L_i)^k$-valued $(0,1)$-section $\tau$ on $X $ satisfying
$$\dbar \tau =0, ~ \int_X |\tau|^2_{\tilde h} dV < \infty,$$
there exists an $(L_i)^k$-valued section $\nu$ such that
$$\dbar \nu = \tau$$ and
\begin{equation}\label{l2}
\int_X |\nu|^2_{\tilde h} dV \leq A  \int_X |\tau|^2_{\tilde h, \hat g} dV .
\end{equation}
\end{proposition}
\begin{proof} By definition,
\begin{eqnarray*}
Ric(\tilde h ) &=& Ric(\hat h) + \ddbar v_i \\
&\geq & \left(\frac{1}{2} - Bk^{-1}\right) \hat g - Ric(\hat g) + \delta_i m_Y^{-1}[\Delta_{\sigma_{Y}}]\\
&\geq&   \left(\frac{1}{2} - Bk^{-1}\right) \hat g - Ric(\hat g)
\end{eqnarray*}
as currents. The standard global $L^2$-estimate (cf. \cite{T1}) implies there exists  $(L_i)^k$-valued smooth section $\nu$ solving the $\dbar$-equation
$$\dbar \nu = \tau$$
and satisfies the estimate (\ref{l2}).
\end{proof}
%


For any $q \in \cX_{-1}$, we can assume that $(\cX_{-1}, kd_{-1}, q) $ converges in pointed Gromov-Hausdorff topology to a metric cone $C(Z)$ over the cross section $Z$ as $k\rightarrow \infty$ by \cite[Proposition 4.31]{HJST}. Let $O$ be  the vertex of $C(Z)$.  We write $\cR_Z$ and $\cS_Z$ the regular and singular part of $Z$. $\cS_Z$ is closed and has Hausdorff dimension no greater than $2n-4$.  $C(\cR_Z)\setminus \{ O \}$ has a natural complex structure induced from the Gromov-Hausdorff limit and the cone metric $g_C$ on $C(Z)$ is given by
$$g_C = \frac{1}{2} \ddbar r^2 ,$$
where $r$ is the distance function for any point $z\in C(Z)$ to $p$.

We can consider the rescaled sequence $(X, k_i g_{i, -1}, q_{k_i})$ such that its Gromov-Hausdorff limit after passing to a subsequence is $C(Z)$ with $q_{k_i}\rightarrow O$ as $i \rightarrow \infty$. In particular, the convergence is smooth on the regular part of $C(Z)$.

One considers the trivial line bundle $L_C$ on $C(Z)$ equipped with the connection $A_C$ whose curvature coincides with $\frac{1}{2} g_C$. In particular, the curvature of the hermitian metric defined by $h_C=e^{-r^2}$ is $g_C$.   Then $k_i g_{i, -1}$ and $((L_i)^{k_i}, A_i)$ converge smoothly  to $g_C$ and $(L_C, A_C)$ on the regular part of $C(Z)$, where $A_i$ is the connection of $((L_i)^{k_i}, (h_i)^{k_i})$.

One can apply the same techniques in \cite{DS1, T15} to construct the following peak section with the $L^2$-estimates we derive earlier.

\begin{lemma} There exist $\bar k>0$, $R>0$, $\delta>0$  such that for each $i\geq 1$ and $q\in B_{\bar k g_{i, -1}}(p_i, R)$, there exist holomorphic section $\sigma\in H^0(X, L_i^{\bar k})$ with $|| \sigma||_{L^2(B_{\bar k g_{i, -1}}(p_i, 2R), (h_i)^{\bar k}, \bar k g_{i, -1})} =1$,
$$|\sigma|^2_{(h_i)^{\bar k}} (q) \geq  \delta. $$

\end{lemma}

For any sequence $q_j$ of $X$ with
$$q_j \rightarrow q \in \cX_{-1},$$
there exist $\delta_1>\delta_2>0$, $r_1>r_2>0$, $\bar k>0$, $C>0$ such that one can further construct peak sections of Gaussian type
$$\{ \sigma_{i, 0}, \sigma_{i, 1}, \sigma_{i, 2}, ..., \sigma_{i, N}\} \in H^0(X, L_i^{\bar k})$$
as in \cite{DS2} satisfying the following properties.

\begin{enumerate}

\item \begin{equation}\label{Fpar1}
\inf_{B_{\bar k g_{i, -1}}(q_i, \delta_1)} |\sigma_{i, 0} |_{(h_i)^{\bar k}} \geq 1/2.
\end{equation}

\item Let $f_{i, j} =\frac{ \sigma_{i, j} }{\sigma_{i, 0}}$ and
\begin{equation}\label{Fpar2}
F_i = (f_{i, 1}, ..., f_{i, N}): B_{\bar k g_{i, -1}}(q_i, 2\delta_1) \rightarrow \mathbb{C}^N.
\end{equation}
Then
  \begin{equation}\label{Fpar3}
 \sup_{B_{\bar k g_{i, -1}}(q_i, 2\delta_1)} |\nabla F_i|_{ \bar k g_{i, -1}} \leq C.
 \end{equation}
\item \begin{equation}\label{Fpar4}
\inf_{\partial B_{\bar k g_{i, -1}}(q_i, \delta_1)} |F_i| > r_1, ~\sup_{\partial B_{\bar k g_{i, -1}}(q_i, \delta_2)} |F_i| < r_2.
\end{equation}

\end{enumerate}
If we let $W_i = F( B_{\bar k g_{i, -1}}(q_i, \delta_1)$. After passing to a subsequence, $W_i$ converges to an analytic subvariety of $B_{g_{\mathbb{C}^N}}(q, \delta_1) \subset \mathbb{C}^N$ and $F_i$ converges to a map
$$F:  B_{d_{-1}}(q, \delta_1) \subset \cX_{-1} \rightarrow W \subset \mathbb{C}^N, $$
where $d_t$ is the metric of the limiting flow $\cX$ at $t\in (-\infty, 0)$. By the same argument in \cite{DS2}, one can further assume $W$ is normal and $F$ is bijective, after suitable modification of $F_i$.  We can replace $t=-1$ by any $t\in (-\infty, 0)$. Immediately, we have the following proposition as the analogue of the consequence from the partial $C^0$-estimates in \cite{DS1, T15, DS2, TiZZL16, CSW}.

\begin{proposition} \label{normana}
For each $t\in (-\infty, 0)$, $\cX_t$ is homeomorphic to an analytic normal variety $\cY_t$ of complex dimension $n$.
\end{proposition}

We will also show that the regular parts of the metric space $\cX_t$ and algebraic variety $\cY_t$ coincide.

\begin{lemma} \label{recoin} The regular set of $\cY_t$ coincides with the regular set of $\cX_t$.

\end{lemma}

\begin{proof} It is obvious that the regular set of $\cX_t$ must be a subset of the regular set of $\cY_t$ because $g_{i, t}$ and the complex structure $J_i$ of $(X, p_i)$ converge smoothly on the regular part of $\cX_t$. Suppose $q$ is a regular point of $\cY_t$ for fixed $t\in (-\infty, 0)$. Then there exists  $F: B_{d_t}(q, \delta) \rightarrow \mathbb{C}^n$ such that $F(B_{d_t}(q, \delta))$ is an open neighborhood of $0\in \mathbb{C}^n$. Furthermore, $F$ is a homeomorphism from $B_{d_t}(q, \delta)$ to $F(B_{d_t}(q, \delta))$ and $F$ is holomorphic on the regular part of $B_{d_t}(q, \delta)$. We can approximate $F$ by a sequence of biholomorphisms
$$F_i: B_{g_{i, t}}(q_i , \delta) \rightarrow F(B_{g_{i, t}}(q_i , \delta)) \subset \mathbb{C}^n$$
satisfying (\ref{Fpar1}), (\ref{Fpar2}), (\ref{Fpar3}), (\ref{Fpar4}) for some sufficiently small $\delta>0$ and uniformly $\delta_2, \delta_2 < \delta$, $r_1> r_2>0$ and $C>0$ with $N=n$,  by constructing Gaussian peak sections from the $L^2$-estimates and the partial $C^0$ techniques in \cite{DS1, T15, DS2}.

We can now assume $g_{i, t} = \ddbar \psi_i$ in $B_{g_{i, t}}(q_i, \delta)$. Let $v_{i, t}$ be the Ricci potential. Since $F_i$ is biholormphic and $B_{g_{\mathbb{C}^n}}(0, r_1) \subset F_i (B_{g_{i, t}}(q_i, \delta))$, we can identify $g_{i, t}$ as a K\"ahler metric on $B_{g_{\mathbb{C}^n}}(0, r_1)$. Then there exists a nonvanishing holomorphic function $f_i$ on $B_{g_{\mathbb{C}^n}}(0, r_1) $ such that
\begin{equation}\label{maF}
 (\ddbar \psi_i)^n = e^{v_{i, t}} |f_i|^2 (\ddbar |z|^2)^n
 \end{equation}%
on $B_{g_{\mathbb{C}^n}}(0, r_1) $, where $z$ is the holomorphic coordinates of $\mathbb{C}^n$ and $g_{\mathbb{C}^n}= \ddbar |z|^2$. By our Li-Yau gradient estimates, there exists $C>0$ independent of $i$ such that
$$||v_{i, t}||_{C^1\left(B_{g_{\mathbb{C}^n}}(0, r_1), g_{i, t} \right)} \leq C. $$
By (\ref{Fpar3}), there exists $C>0$ independent of $i$ such that
$$ \inf_{B_{g_{\mathbb{C}^n}}(0, r_1)} |f_i| \geq C^{-1}. $$
On the other hand, there exists $C>0$ independent of $i$ such that
$$\vol (B_{g_{i, t}}(q_i , \delta), g_{i, t}) \leq C$$
by the volume non-inflation estimate (cf. \cite[Theorem 8.1]{Bam20a}). This immediately implies that the $L^2$-norm of $f_i$ with respect to $g_{\mathbb{C}^n}$ is uniformly bounded and so there exists $C>0$ independent of $i$ such that
$$ \sup_{B_{g_{\mathbb{C}^n}}(0, r_1)} |f_i| \leq C.$$
By (\ref{Fpar3}), $g_{i,,t}$ is uniformly bounded below by a multiple of $g_{\mathbb{C}^n}$. By the uniform bounds on $f_i$ and $v_i$, $g_{i, t}$ is also bounded above by a multiple of $g_{\mathbb{C}^n}$ from (\ref{maF}). Therefore, $g_{i, t}$ is $C^0$-equivalent to $g_{\mathbb{C}^n}$. We now apply the Schauder estimate of complex Monge-Amp\`ere equations for $\psi_i$ and so $g_{i, t}$ is $C^{\alpha}$ equivalent to $g_{\mathbb{C}^n}$ for some uniform $\alpha>0$. Therefore $g_{i, t}$ converges in $C^\alpha$ and so the tangent flow of $\cX$ at $q\in \cX_t$  is the flat $\mathbb{C}^n$. This implies that $q$ must be a regular point of $\cX$. The lemma is then proved.
\end{proof}

We can now completed the proof Theorem \ref{2main1}.

\begin{proof}[Proof of Theorem \ref{2main1}, item (2)]
This follows immediately from combining Proposition \ref{hjst1}, Proposition \ref{normana} and Lemma \ref{recoin}.
\end{proof}


\bigskip
\section{Diameter and scalar curvature estimates on Fano bundles}\label{dandseonbundle}

In this section, we are under the set-up of Section \ref{setupofKRF}, that is, we have the solution to the unnormalized K\"ahler-Ricci flow, which we denote by $\omega(t),~~ t\in [0, 1)$; and we have the solution to the normalized K\"ahler-Ricci flow, which we denote by $\tilde\omega(s),~~ s\in [0, \infty)$. These two flows are related by
\begin{equation}\label{rb unkrf and nkrf'4}
s=-\ln (1-t),~~ t=1-e^{-s},~~ \tilde{\omega}(s)=(1-t)^{-1}\omega(t).
\end{equation}
We use the notations and conventions from Section \ref{setupofKRF}.

%

%
Now we assume that $0< m:=\dim_{\mathbb{C}} Y < \dim_{\mathbb{C}} X = n$, and that the regular fibres of $\Phi: X \rightarrow Y$ are biholomorphic to each other.

For any point $q\in Y^{\circ}$, there exists a Zariski open set $U_q$ of $Y$ such that
$$\pi^{-1}(U_q) = U_q \times Z, $$
where $Z : = \Phi^{-1}(q)$ is a Fano manifold of dimension $n-m$. For simplicity, we can assume $U_q = Y \setminus D_q$ for some divisor $D_q$ of $Y$. Let $\sigma_q$ be the holomorphic section associated to $D_q$ equipped with a smooth hermtian metric $h_q$. Without loss of generality, we can assume
$$ Ric(h_q) \leq B_0 \omega_Y , $$
for some fixed constant $ B_0 = B_0 ( q ) < \infty $. We let
$$F_q: U_q\times Z \rightarrow Z $$
be the horizontal projection.

The following lemma is essentially proved in \cite{FZ} (cf. Lemma 2.2).

\begin{lemma} \label{hor} There exist constants $C, k<\infty$, depending on $n, \omega_0$, such that for any $t\in [0, 1)$ and $z \in Z$, we have
$$  \omega(t)|_{F_q^{-1}(z)} \leq C\left( |\sigma_q|^{-2k}_{h_q} \right)  \omega_Y.$$
Here we identify $F_q^{-1}(z)$ with $U_q$ and $\omega(t)|_{F_q^{-1}(z)}$ is the restriction of $\omega(t)$ on $F_q^{-1}(z)$.
\end{lemma}
Then we have the following estimates on volume forms.
\begin{lemma}\label{vbonFanobundle}
There exists $C, k<\infty$, depending on $n, \omega_0$, such that for any $t\in [0, T)$, we have
$$C^{-1}\omega_Y^m \wedge \omega(t)^{n-m} \leq \omega(t)^n \leq C\left( |\sigma_q|^{-2mk}_{h_q}\right) \omega_Y^m\wedge \omega(t)^{n-m} , $$
on $U_q\times Z$.
\end{lemma}
\begin{proof} The first inequality is a direct consequence of the parabolic Schwarz lemma. For the second inequality follows from the following linear algebraic inequality
\begin{equation} \label{algdet} \det \left(
             \begin{array}{cc}
              A & B \\
              \bar B^T & D \\
             \end{array}
           \right)  \leq \det A \det D
        \end{equation}
for any positive definite hermitian matrix $\left(\begin{array}{cc}
   A  & B \\
   \bar B^T  & D
  \end{array}  \right)$ and Lemma \ref{hor}.
For the sake of completeness, we will prove (\ref{algdet}) below. There exist a  unitary matrix $U$ and a positive definite diagonal matrix $V$ such that $V^2=\bar U^T D U$. Then
$$  \left(
             \begin{array}{cc}
             A & B \\
             \bar B^T & D \\
             \end{array}
           \right) =  \left(
             \begin{array}{cc}
              A & 0 \\
              \bar B^T & I \\
             \end{array}
           \right)  \left(
             \begin{array}{cc}
             I &  A^{-1} B\\
          0 &  D- \bar B^T A^{-1} B \\
             \end{array}
           \right)  $$    
  and so
  $$\det \left(
             \begin{array}{cc}
              A & B \\
              \bar B^T  & D \\
             \end{array}
           \right) 
           = (\det A) (\det D) \det( I- K )   \leq \det A \det D, $$
where $K= \left(\overline{BUV^{-1}} \right)^T A^{-1} \left(BUV^{-1} \right)$. The last inequality follows from the observation that both $I-K$ and $K$ are semi-positive definite with each eigenvalue of $K$ lying in $[0, 1]$. We have now proved (\ref{algdet}).
\end{proof}
As before, we fix a smooth K\"ahler metric $\omega_B$ on the base space $\mathbb{CP}^N$, e.g. $\omega_B=\lambda \omega_\fs$ for any $\lambda>0$, and for any $p\in X$, we denote by $\mathcal{B}_{\omega_B}(p, r) = \Phi^{-1}\left(B_{\omega_B} ( \Phi(p) , r) \right)$ the pre-image of the geodesic ball $B_{\omega_B}(\Phi(p), r) \subset \mathbb{CP}^N$. We have the following estimates on the Type I tubular neighborhood of the Ricci vertices.

\begin{theorem} \label{scanddeonFanobundle}
Let $g(t)$ be the maximal solution of the K\"ahler-Ricci flow (\ref{unkrflow1'1}) on $X\times [0, 1)$ as described above. For any open subset $U\subset\subset Y^\circ$, $T<\infty$, there exists a smooth closed $(1,1)$-form $\theta_Y\in \vartheta$ on $Y$, constants $\ep>0$, $C<\infty$, both depend on $n, g_0, \theta_Y, U, T$, such that the followings hold for all $t\in [1-\ep, 1)$.

\begin{enumerate}

\item There exists a Ricci vertex $p_t \in \Phi^{-1}(U) $ associated with $\theta_Y$ at time $t$.

\medskip

\item For any $s\in [t-T(1-t), 1-T^{-1}(1-t) ]$, we have the estimates
\begin{equation} \label{scandde1}
\diam \left( \mathcal{B}_{\omega_B}(p_t, (1-s)^{1/2}) , g(s) \right)  \leq C(1-s)^{1/2} ,
\end{equation}
\begin{equation} \label{scandde2}
\sup_{ \mathcal{B}_{\omega_B}(p_t, (1-s)^{1/2}) } |\rr ( \cdot , s) |\leq \frac{C}{1-s},
\end{equation}
where $R$ is the scalar curvature of $g(t)$.
\end{enumerate}
\end{theorem}
\begin{proof}
Throughout the proof, all the constants depend at most on $n, g_0, \theta_Y, U, T$.

Let $q\in U \subset Y^{\circ}$ be a given point. Then we can choose $r_0>0$ small enough, which depends only on $q$ (hence $U$), such that
\begin{enumerate}
\item $B_{\omega_B} \left (q , 10r_0 \right )\cap Y \subset U\cap U_q$ ;
\item $B_{\omega_B} \left( p , r \right)\cap Y$ has only one component for any $p\in B_{\omega_B} \left (q , r_0 \right )$ and $r\in (0, 10r_0)$.
\end{enumerate}
%
Hence there is a constant $C<\infty$, such that $|\sigma_q|\geq C^{-1}$ on $B_{\omega_B} \left (q , 2r_0 \right )\cap Y$. Then for any $p\in B_{\omega_B} \left (q , r_0 \right )$ and $r\in (0, r_0)$, we have $B_{\omega_B} \left( p , r \right) \subset B_{\omega_B} \left( q , 2r_0 \right)$,  hence by Lemma \ref{vbonFanobundle}, we have
\begin{equation}\label{vbaq}
\vol_{g(s)}\left(\Phi^{-1}(B_{\omega_B}(p, r))\right)\leq C r^{2m}(1-s)^{n-m},
\end{equation}
for any $s\in [0, 1)$. Then we choose $\ep>0$ small enough, such that for any $s\in [t-T(1-t), 1-T^{-1}(1-t) ]$, we have
$$
10(1-s) \leq 10(1+T)(1-t)< r_0,
$$
provided that $1-t\leq\ep$ .Hence from (\ref{vbaq}), for any $p\in B_{\omega_B} \left (q , r_0 \right )$ and $s\in [t-T(1-t), 1-T^{-1}(1-t) ]$, we have
\begin{equation}\label{vbaq2}
\vol_{g(s)}\left(\Phi^{-1}(B_{\omega_B}(p, 10(1-s)^{1/2}))\right)\leq C (1-s)^{n},
\end{equation}
for some constant $C<\infty$.

Now, by Lemma \ref{a'1}, we can choose a smooth closed $(1,1)$-form $\theta_Y\in \vartheta$ on $Y$, such that for any $t\in [0, 1)$, if $p_t$ is a Ricci vertex associated to $\theta_Y$, then we have $p_t \in \Phi^{-1}(B_{\omega_B}(q, r_0)) \subset \Phi^{-1}(U) $. Hence by (\ref{vbaq2}), we can apply Corollary \ref{cor2} to obtain
$$
\diam \left( \mathcal{B}_{\omega_B}(p_t, (1-s)^{1/2}) , g(s) \right)  \leq C(1-s)^{1/2} ,
$$
$$
\sup_{ \mathcal{B}_{\omega_B}(p_t, (1-s)^{1/2}) } |\rr ( \cdot , s) |\leq \frac{C}{1-s},
$$
this completes the proof.
\end{proof}
Now we can finish the proof of Theorem \ref{app3}.
\begin{proof}[Proof of Theorem \ref{app3}]
This follows from Theorem \ref{scanddeonFanobundle} immediately.
\end{proof}

\begin{proof}[Proof of Corollary \ref{mainbund}]
(1) follows directly from Theorem \ref{app3}. For (2), we consider the fixed holomorphic map $\Phi: X \rightarrow Y$ induced by the limiting cohomology class. For any fixed $t\in (-\infty, 0]$ and a fixed smooth K\"ahler metric $\mathrm{g}_Y$ on $Y$,  there exists $C>0$ such that for all $i\geq 1$,
$$\sup_X |\nabla \Phi|_{ g_{i, t}, \mathrm{g}_{Y, i}}  \leq C$$
by the parabolic Schwarz lemma, where
$$\mathrm{g}_{Y, i}= (1-t_i)^{-1} \mathrm{g}_Y. $$
 Hence $\Phi_i=\Phi: (X, g_{i, t}) \rightarrow (Y, \mathrm{g}_Y)$ is uniformly Lipschitz for all $i\geq 1$. After passing to a subsequence, $(X, g_{i, t}, p_i)$ converges in pointed Gromov-Hausdorff sense to $(\cX_t, g_t, p_\infty)$ and $(Y, g_{Y, i}, \Phi(p_i))$ converges in pointed Cheeger-Gromov sense to $(\mathbb{C}^m, g_{\mathbb{C}^m}, O)$. Then by holomorphic extension,  $\Phi_i$ converges in $C^\alpha$ to a holomorphic map $$\Phi_\infty: (\cX_t, g_t) \rightarrow (\mathbb{C}^m, g_{\mathbb{C}^m}).$$
We claim that $\Phi_\infty$ is surjective.  If not, $V=\Phi_\infty(\cX_t)$ is a closed analytic subvariety of $\mathbb{C}^m$. We choose a fixed point $y \in \mathbb{C}^m \setminus V$ such that
$$d_{\mathbb{C}^m}(y, O)=1, ~d_{\mathbb{C}^m}(y, V)>\delta>0$$
for some $\delta>0$.
 Suppose $y_i \in (Y, \mathrm{g}_{Y, i}) \rightarrow y \in (\mathbb{C}^m, g_{\mathbb{C}^m})$. We pick  $z_i\in \Phi_i^{-1}(y_i)$. By Lemma \ref{hor} and the fibre diameter estimates, there exists $C>0$ such that for all $i>0$,
 $$ d_{g_{i, t}}(z_i, p_i) \leq C. $$
 After passing to a subsequence, $z_i$ converges to $z_\infty\in \cX_t$ with $d_{g_t}(z_\infty, p_\infty) \leq C$. On the other hand, $\Phi_i(z_i)=y_i \rightarrow y$ and so $\Phi_\infty(z_\infty) = y$. Contradiction.
\end{proof}

\begin{proof}[Proof of Corollary \ref{rules}]
Since $\dim_{\mathbb{C}} X=2$, the limiting ancient solution $\cX$ is a complete K\"ahler surface with isolated orbifold singularities. Suppose $p$ is a nontrivial orbifold singularity of the limiting space $(\cX_t, g_t)$ for some $t\in (-\infty, 0]$. Let $p_i \in (X, g_{i, t})$ with $p_i \rightarrow p$.  By Corollary \ref{mainbund},
$\Phi_i: X \rightarrow Y$ converges to a surjective Lipschitz map $\Phi_\infty: \cX_t \rightarrow \mathbb{C}$.
By Lemma \ref{hor} and the parabolic Schwarz lemma,  there exists $C>0$ such that
$$C^{-1} \leq \inf_{\mathcal{B}_{\mathrm{g}_{Y, i}}(p_i, 1)} |  \nabla \Phi_i |_{g_{i, t}, g_{Y, i}} \leq \sup_{\mathcal{B}_{\mathrm{g}_{Y, i}}(p_i, 1)} |  \nabla \Phi_i |_{g_{i, t}, g_{Y, i}} \leq C, $$
where $$ \mathcal{B}_{\mathrm{g}_{Y, i}}(p_i, 1) = \Phi^{-1} \left( B_{\mathrm{g}_{Y, i}}(\Phi(p_i), 1)\right). $$
Immediately, we have
$$C^{-1} \leq \inf_{\mathcal{B}_{g_{\mathbb{C}} (p, 1)}}  |  \nabla \Phi_\infty |_{g_t, g_{\mathbb{C}}} \leq \sup_{\mathcal{B}_{\mathbb{C}}(p, 1)} |  \nabla \Phi_\infty |_{g_t, g_{\mathbb{C}}} \leq C.$$
We can assume that $\Phi_\infty: B\slash \Gamma \rightarrow \mathbb{C}$ is a holomorphic map and $\Phi_\infty(0)=0$, where $B$ is a Euclidian ball in $\mathbb{C}^2$ centered at $0$ and $\Gamma$ is finite subgroup in $U(2)$. We can now lift $\Phi_\infty$ to a $\Gamma$-invariant holomorphic map $\tilde \Phi: B \rightarrow \mathbb{C}$. Since $\Gamma$ is not trivial and $\tilde\Phi(0)=0$, $\tilde \Phi$ does not have linear terms in $z=(z_1,z_2)$. Immediately we have $\partial \tilde\Phi (0)=0$ and  so $\partial \Phi_\infty(p)=0$, which leads to contradiction. Therefore $p$ must be a smooth point and so $\cX_t$ must be smooth. This further implies smooth convergence $g_{i, t}$ and so $\cX_t$ must be a $\mathbb{CP}^1$-bundle over $\mathbb{C}$. We have completed the proof of the corollary.
\end{proof}


\bigskip
\section{K\"ahler-Ricci flow with Calabi symmetry}\label{calsym}


\subsection{The Calabi symmetry}
We consider the un-normalized Kahler-Ricci flow
$$\ddt{g (t)} = - Ric(g(t))$$
on a projective manifold $X$ satisfying the Calabi symmetry. The manifold $$X= \mathbb{P}(\mathcal{O}_{Z^n} \oplus L^{\oplus (m+1)}) $$ is a projective bundle over an $n$-dimensional  Kahler-Einstein manifold $Z$, where $L$ is a negative line bundle over $Z$. Let $h$ be a smooth hermitian metric on $L$ so that $\omega_Z = -Ric(h)= \ddbar \log h$ is a Kahler-Einstein metric satisfying
$$Ric(\omega_Z) = \lambda \omega_Z, $$
where we can assume $\lambda = -1$,  $0$ or $1$ after normalization.
A special example is $X=X_{m, n}$ when $Z=\mathbb{CP}^n$ and $L= \mathcal{O}_{\mathbb{CP}^n}(-1)$. When $m=0$, $X_{0, n}$ is the blow-up of $\mathbb{CP}^{n+1}$ at one point. When $m\geq 1$, $X_{m,n}$ and $X_{n,m}$ are birationally equivalent for $m\geq 1$, and differ by a flip or a flop.

Let $D_\infty$ be the divisor in $X $ at the infinity given by $\mathbb{P}(L^{\oplus (m+1)})$. We also let $P_0$ be the zero section, i.e. $(1, 0,0,..., 0) \in \mathcal{O}_Z \oplus L^{m+1}$.     We also define the divisor $D_H$ on $Z $ by the pullback of the divisor $L^{-1}$ on $Z$.
Then  %
\begin{equation}\label{canbundle}
[ K_X ] = - (m+2) [D_\infty] + (m+1-\lambda ) [D_H] .
\end{equation}
In particular, $D_\infty$ is a big and semi-ample divisor. The divisor $ a[D_H] + b[D_\infty]$ is ample if and only if $a>0$ and $b>0$. Hence $X$ is Fano if and only if $\lambda>m+1$, i.e,
$-K_Z + (m+1) L$ is ample.

We choose sufficiently large $k\in \mathbb{Z}^+$ so that $L^{-k}$ is very ample. Let $\{\sigma_0, ..., \sigma_{N_k}\}$ be a basis for $H^0(Z, L^{-k})$.   Let $W_0, ..., W_m$ be the homogenous coordinates for $\mathbb{CP}^m$. Then monomials of $W_0, .., W_m$ of degree $k$ can be identified as a basis $H^0(\mathbb{CP}^{m+1}, \mathcal{O}(k))$ and we denote it as $\{\eta_0, ...., \eta_{M_k}\}$. The linear system $| k [D_\infty]|$ is base-point-free for sufficiently large $k$ and it induces a birational morphism

$$\Phi_k: X \rightarrow \mathbb{CP}^{(N_k+1)\times (M_k+1)-1}.$$

$\Phi_k$ is an immersion on $X\setminus P_0$ and it contracts $P_0$ to a point. $Y$, the image of $\Phi_k$ in $\mathbb{CP}^{(N_k+1)\times (M_k+1)-1}$, is a projective cone over $Z \times \mathbb{CP}^m $ in $\mathbb{CP}^{(N_k+1)\times (M_k+1)-1}$ by the Segre embedding $$[\sigma_0, ..., \sigma_{N_k} ]\times[\eta_0, ..., \eta_{M_k} ]\rightarrow [\sigma_0\eta_0, ..., \sigma_i \eta_j, ..., \sigma_{N_k} \eta_{M_k} ]\in \mathbb{CP}^{(N_k+1)\times (M_k+1)-1} .$$

Let us define the Calabi ansatz constructed by Calabi \cite{Cal} (also see \cite{Li}). We consider the vector bundle $$E= L ^{\oplus (m+1)}$$ over $Z$. The induced hermitian metric $h_E$ on $E$ is given by $h_E= h^{\oplus (m+1)}$.
Under local trivialization of $E$, we write
$$e^\rho = h(z) |\xi|^2, ~~~ \xi = (\xi^0, \xi^1, ..., \xi^{m}),$$ where $h(z)$ is a local representation for $h$. In the future calculations, we will always compute in terms of $\rho$.

We would like to find appropriate conditions for $a\in \mathbb{R}$ and a smooth real valued function $\phi=\phi(\rho)$ such that
\begin{equation}\label{metricrep1}
\omega = a \omega_Z + \ddbar \phi(\rho)
\end{equation}
defines a K\"ahler metric on $X$.
In fact,
\begin{equation}\label{metricrep2}
\omega = (a + \phi'(\rho)) \omega_Z + \frac{\sqrt{-1}}{2\pi} h e^{-\rho} ( \phi' \delta_{\alpha \beta} + h e^{-\rho} ( \phi'' - \phi') \xi^{\bar \alpha} \xi^{\beta} ) \nabla \xi^\alpha \wedge \overline{\nabla \xi^\beta}.
\end{equation}
where $$\nabla \xi^\alpha = d \xi^\alpha + h^{-1} \partial h \xi^\alpha$$ and $\{ dz^i, \nabla \xi^\alpha\}$ is dual to the basis

$$ \nabla_{z^i} = \frac{\partial}{\partial z^i} - h^{-1} \frac{\partial h }{\partial z^i} \sum_\alpha \xi^\alpha \frac{\partial }{\partial \xi^\alpha}, ~~ \frac{\partial}{\partial \xi ^\alpha}.$$
At each point of $X$, we can choose the local coordinate for $Z$ with $d\log h=0$ and assume $\xi=(0, 0, ..., 0, \xi_{m})$ at this point. Then (\ref{metricrep2}) becomes
\begin{equation}\label{metricrep3}
\omega = (a + \phi'(\rho)) \omega_Z + \frac{\sqrt{-1}}{2\pi} he^{-\rho} \phi'  \sum_{j=1}^m d\xi^j\wedge d\overline{\xi^j}  + \frac{\sqrt{-1}}{2\pi} h e^{-\rho} \phi'' d\xi^{0}\wedge d\overline{\xi^{0}}
\end{equation}

The following criterion is due to Calabi \cite{Cal}.
\begin{proposition}\label{kacon}

$\omega$ as defined above is a K\"ahler metric if and only if

\begin{enumerate}

\item

$a>0$.

\item $\phi'>0$ and $\phi''>0$ for $\rho\in (-\infty, \infty)$.

\item $\varPhi_0 (e^\rho) = \phi(\rho)$ is smooth on $(-\infty, 0]$ and $\varPhi_0' (0)>0$.

\item $\varPhi_\infty (e^{-\rho}) = \phi(\rho) - b \rho$ is smooth on $[0, \infty)$ for some $b>0$ and $\varPhi_\infty'(0)>0$.

\end{enumerate}

\end{proposition}

We remark that given $a$, $b>0$, the K\"ahler metric constructed above lies in the K\"ahler class

\begin{equation}
\omega=a\omega_Z + \ddbar \phi (\rho) \in a[D_H]+ b[D_\infty]
\end{equation}
and
$$\lim_{\rho\rightarrow-\infty} \phi'(\rho)=0 < \phi'(\rho) < \lim_{\rho\rightarrow \infty} \phi'(\rho) = b.$$

Straightforward calculations show that the induced volume form of $\omega$ is given by

\begin{equation}\label{volrep}
\omega^{m+n+1} = (a + \phi')^n h^{m+1} e^{-(m+1)\rho} ( \phi')^{m} \phi''  (\omega_Z)^n \wedge \prod_{\alpha=0}^{m} \frac{\sqrt{-1}}{2\pi} d\xi^\alpha\wedge d \xi^{\bar \alpha}.
\end{equation}
Therefore,

$$- Ric(\omega) = \ddbar (  \log [(a+ \phi')^n (\phi')^{m} \phi''] - (m+1)\rho )  +  ( m+1  - \lambda  )\omega_Z.$$
Let $u= -  \log \left((a+ \phi')^n (\phi')^{m} \phi'' \right) +  (m+1)\rho.$ Then using the special coordinates,  we have
\begin{eqnarray*}
Ric(\omega) &=& \ddbar \log u +(n-m) \omega_Z \\
&=& (\lambda-m -1 +u') \omega_Z + \frac{\sqrt{-1}}{2\pi} he^{-\rho} u'  \sum_{j=1}^m d\xi^j\wedge d\overline{\xi^j}  + \frac{\sqrt{-1}}{2\pi} h e^{-\rho} u'' d\xi^{0}\wedge d\overline{\xi^{0}}
\end{eqnarray*}
and
\begin{eqnarray} \label{scurv}
R &=& \frac{n(\lambda-m-1+u')}{a+\phi'} + m \frac{ u'}{\phi'} + \frac{u''}{\phi''}\\
&=&- \frac{\phi^{(4)}}{(\phi'')^2} +  \frac{(\phi''')^2}{(\phi'')^3} - 2n \frac{\phi'''}{(a+\phi')\phi''} - 2m\frac{\phi'''}{\phi'\phi''} - 2mn \frac{\phi''}{\phi'(a+\phi')}   \nonumber\\
&& + (m - m^2) \frac{\phi''}{(\phi')^2} + (n-n^2) \frac{\phi''}{(a+\phi')^2}  +  \frac{n\lambda}{a+\phi'}  + \frac{m(m+1)}{\phi'}    \nonumber
\end{eqnarray}

If we let
$$(y_1, ..., y_{m+n+1}) = (\xi_0, \xi_1, ..., \xi_m, z_1, ..., z_n)$$
 be the holomorphic coordinates on $X_{m, n}$, by the Calabi symmetry, we can assume that $z=(z_1, ..., z_n)$ is the normal coordinates for $\omega_Z$ and  $ \xi_1=....= \xi_m=0$ after suitable rotations.

The curvature formula
$$R_{i\bar j k\bar l} = - \frac{\partial^2 g_{i\bar j}}{\partial y_k \partial y_{\bar l}} + g^{p\bar q} \frac{\partial g_{i\bar q}}{\partial y_k} \frac{\partial g_{p\bar j}}{\partial y_{\bar l}}.$$
gives
\begin{equation}\label{R1bar1}
R_{1\bar 1 1 \bar 1} = - e^{-2\rho} \phi^{(4)} +  e^{-2\rho} F_{1\bar 1 1\bar 1}( \phi', \phi'', \phi'''),
\end{equation}
\begin{equation}\label{Rijkl}
(1-\delta_{ijkl1}) R_{i\bar j k \bar l} =  e^{-2\rho}  F_{ijkl}(\phi', \phi'', \phi'''),
\end{equation}
where $\delta_{ijkl1} =1$ when $i=j=k=l=1$ and $\delta_{ijkl1}=0$ otherwise. Each term of $F_{ijkl}$ is a monomial in $\phi'$, $\phi''$ and $\phi'''$ of degree $1$ with $\phi'''$ only appearing on the numerator. The only curvature term that contains $\phi^{(4)}$ is $R_{1\bar 1 1\bar 1} $.

It is straightforward to check that  the Calabi ansatz is preserved by the Ricci flow. Indeed, the K\"ahler-Ricci flow
\begin{equation}
\ddt{\omega(t)} = - Ric (\omega(t)),  ~~~\omega|_{t=0} = \omega_0= a_0 \omega_Z + \ddbar u_0 \in a_0[D_H]+b_0[D_\infty]
\end{equation}
is equivalent to the following parabolic equation
$$ a'(t)\omega_Z + \ddbar \ddt{\phi} = (m+1-\lambda) \omega_Z + \ddbar (  \log [(a+ \phi')^n (\phi')^{m} \phi''] - (m+1) \rho ) .
 $$
Separating the variables, we have that

\begin{equation}
a=a(t) =a_0  - ( \lambda-m-1) t
\end{equation}
 and
\begin{equation}\label{krfu}
\ddt{\phi} = \log [(a+ \phi')^n (\phi')^{m} \phi''] - (m+1)\rho + c_t,
\end{equation}
where we choose
$$
c_t = - \log \phi''(0,t) - m \log \phi'(0,t)- n\log (a(t)+\phi'(0,t)) ,
$$
so that $\frac{\partial \phi}{\partial t}(0, t)=0$ for all $t\geq 0$.  From the formula (\ref{canbundle}) and the K\"ahler class evolves by $[\omega(t)]=(a_0-(\lambda-m-1)t)[D_H]+ (b_0-(m+2)t) [D_\infty]$, and so $$ b=b(t) = b_0 - (m+2) t. $$
Therefore, the Ricci flow must develop finite time singularities as either $a(t)$ or $b(t)$ will become $0$. There are only there possibilities at the first singular time $T$.

\begin{enumerate}

\item When $a(T)=0$ and $b(T)>0$, then the flow will contract the zero section $P_0$.

\medskip

\item When $a(T)>0$ and $b(T)=0$, the flow will collapse all the $\mathbb{CP}^m$ fibres.

\medskip

\item When $a(T)=b(T)=0$, the flow will be extinct.

\end{enumerate}

 It is straightforward to show that equation (\ref{krfu}) admits a smooth solution $\phi$ satisfying the Calabi ansatz as long as the K\"ahler-Ricci flow admits a smooth solution, by comparing $\phi$ to the solution of the Monge-Amp\`ere flow associated to the K\"ahler-Ricci flow. Let us also give the evolution equations for $\phi'$ and $\phi''$  by
\begin{eqnarray}
 \ddt{\phi'} &=& \frac{\phi'''}{\phi''} + \frac{ m \phi''}{\phi'} + \frac{ n \phi''}{a + \phi'} - (m+1),
 \\ \label{udpevolution}
\ddt{\phi''}& =& \frac{\phi^{(4)}} {\phi''} - \frac{ (\phi''')^2 }{ (\phi'')^2} + \frac{ m \phi'''}{\phi'} - \frac{ m (\phi'')^2}{ (\phi')^2} + \frac{ n \phi'''}{ a + \phi'} - \frac{ n (\phi'')^2}{( a + \phi')^2},
\end{eqnarray}
as can be seen from differentiating (\ref{krfu}).


\subsection{Analytic estimates}

The first singular time of the K\"ahler-Ricci flow on $X$ is given by
\begin{equation}
T=\sup\{ t>0~|~[\omega_0] + t [K_{X}] >0\}.
\end{equation}
Since $X$ is not a minimal model, $T<\infty$.




%
%


We also notice that $0< \phi'(\rho, t) < b(t)$ for $\rho\in(-\infty, \infty)$ because $\phi'$ is increasing and $\lim_{\rho\rightarrow \infty} \phi'(\rho, t)=  b(t)$.  Therefore, $\phi'$ is uniformly bounded above for $t\in [0, T)$, where $T$ is the singular time for the Ricci flow on $X$. The following estimate is the main result of this section.

\begin{proposition}\label{2ndder1}
There exist $C>0$ and  such that for all $(\rho, t)\in (-\infty,\infty)\times [0, T)$,
\begin{equation}\label{calest1}
C^{-1} \leq \frac{ b(t)\phi''(\rho, t)}{\phi'(\rho, t) (b(t)- \phi'(\rho, t))}  \leq C ,
\end{equation}
\begin{equation}\label{calest2}
\frac{|\phi'''(\rho, t) |}{\phi''(\rho, t)} \leq C  .
\end{equation}

\end{proposition}

\begin{proof}

Let $H= \log \phi'' - \log \phi' - \log (b(t) - \phi')+ \log b(t).$ Notice that by Proposition \ref{kacon}, for fixed $t\in[0, T)$ and near $\rho=-\infty$, $\phi(\rho, t) = \varPhi_0(e^{\rho}, t)$ for some smooth function $\varPhi_0$, and near $\rho=\infty$,  $\phi(\rho, t) = \varPhi_\infty(e^{-\rho}, t)+ b(t) \rho$ for some smooth function $\varPhi_\infty$ and $b(t)>0$.
We then have
$$
\lim_{\rho\rightarrow -\infty} \frac{\phi''}{\phi'}  = \lim_{\rho\rightarrow -\infty}  \frac{ (\varPhi_0(e^\rho))''}{ (\varPhi_0(e^\rho))'}=\lim_{\rho\rightarrow -\infty}  (  \frac{\varPhi_0'  +  e^{\rho}  \varPhi_0''} {  \varPhi_0' }) = 1 + \lim_{\rho\rightarrow -\infty} e^\rho \frac{\varPhi_0''}{\varPhi_0'} =1 ,
$$
$$
\lim_{\rho\rightarrow \infty}  \frac{\phi''}{b(t)- \phi'}  =  \lim_{\rho\rightarrow \infty} \frac{ (\varPhi_\infty(e^{-\rho}) + b\rho)''}{ (- \varPhi_\infty(e^{-\rho}) )'}= \lim_{\rho\rightarrow \infty} \frac{e^{-\rho} \varPhi_\infty'  +  e^{-2\rho}  \varPhi_\infty'' } { e ^{-\rho} \varPhi_\infty' } =1.
$$
Therefore  $H(-\infty, t) = H(\infty, t) = 0$, and we can apply maximum principle for $H$ in $ (-\infty, \infty)\times [0, T)$.
\begin{eqnarray*}
\ddt{H}
&=&  \frac{1}{\phi''} \{ \frac{ \phi^{(4)} }{\phi''} - \frac{ (\phi''')^2}{ (\phi'')^2} + \frac{  m\phi'''}{ \phi'} - \frac{ m (\phi'')^2}{ (\phi')^2} + \frac{ n \phi'''}{a+ \phi'} - \frac{ n (\phi'')^2}{ (a + \phi')^2} \} \\
&& - \frac{1}{\phi'} \{ \frac{\phi'''}{\phi''} + \frac{ m \phi''}{\phi'} + \frac{ n \phi''}{ a+ \phi'} - (m+1)\} \\
&& + \frac{1}{b-\phi'} \{ \frac{\phi'''}{\phi''} + \frac{ m \phi''}{\phi'} + \frac{ n \phi''}{ a+ \phi'} +1\}  - \frac{m+2}{b}.
\end{eqnarray*}
Suppose that $H(\rho_0, t_0) = \sup_{(-\infty, \infty) \times  [0, t_0]  } H(\rho, t)$ is achieved for some $t_0\in (0, T)$, $\rho_0\in (-\infty, \infty)$.  At $(t_0, \rho_0)$, we have

$$H'= \frac{ \phi'''}{\phi''} - \frac{ \phi''}{\phi'} + \frac{\phi''}{b-\phi'}= 0 , $$ and

$$ \frac{ \phi^{(4)}} {\phi''} - \frac{ (\phi''')^2}{(\phi'')^2} - \frac{ \phi''' } { \phi'} + \frac{ (\phi'')^2}{ (\phi')^2}  + \frac{\phi'''}{b- \phi'} + \frac{(\phi'')^2}{(b-\phi')^2} \leq 0.
$$
Then at $(\rho_0, t_0)$,
\begin{eqnarray*}
&&\ddt{H}\\
& \leq & \frac{1}{\phi''} \{ \frac{ \phi''' }{\phi'} - \frac{ (\phi''')^2}{ (\phi'')^2} - \frac{\phi'''}{b-\phi'} - \frac{(\phi'')^2}{(b-\phi')^2}+ \frac{ m \phi'''}{ \phi'} - \frac{ m (\phi'')^2}{ (\phi')^2} + \frac{n \phi'''}{a+\phi'} - \frac{n(\phi'')^2}{ (a+\phi')^2} \}  \\
&&- \frac{1}{\phi'} \{ \frac{\phi'''}{\phi''} + \frac{ m \phi''}{\phi'}  +\frac{n \phi''}{a+\phi'} - (m+1)\}   + \frac{1}{b- \phi'} \{ \frac{\phi'''}{\phi''} + \frac{ m \phi''}{\phi'}  +\frac{n \phi''}{a+\phi'} +1 \}  - \frac{m+2}{b}\\
&=&   -   \frac{(2m+1)\phi''}{(\phi')^2} - \frac{\phi''}{(b-\phi')^2} - \frac{n \phi''}{(a+\phi')^2} + \frac{m \phi'''}{\phi'\phi''} + \frac{n\phi'''}{(a+\phi')\phi''} - \frac{n\phi''}{\phi'(a+\phi')} + \frac{m+1}{\phi'} \\
&& + \frac{m \phi''}{\phi'(b-\phi')} + \frac{n\phi''}{(a+\phi')(b-\phi')} + \frac{1}{b-\phi'} - \frac{m+2}{b}\\
&=&  -   \frac{(2m+1)\phi''}{(\phi')^2} - \frac{\phi''}{(b-\phi')^2} - \frac{n \phi''}{(a+\phi')^2}  - \frac{n\phi''}{\phi'(a+\phi')} + \frac{m+1}{\phi'} + \frac{m \phi''}{\phi'(b-\phi')} + \frac{1}{b-\phi'}\\
&&  + \frac{n\phi''}{(a+\phi')(b-\phi')} +\frac{m\phi''}{(\phi')^2} - \frac{m\phi''}{\phi'(b-\phi')} + \frac{n\phi''}{\phi' (a+\phi')} - \frac{n\phi''}{(a+\phi')(b-\phi')}- \frac{m+2}{b}\\
&=&  -   \frac{(m+1)\phi''}{(\phi')^2}  - \frac{\phi''}{(b-\phi')^2} - \frac{n\phi''}{(a+\phi')^2}  + \frac{m+1}{\phi'} +\frac{1}{b- \phi'} - \frac{m+2}{b} . %
\end{eqnarray*}
Hence at $(\rho_0, t_0)$, we have
\begin{eqnarray*}
\frac{\phi''}{\phi'(b-\phi')} &\leq& \frac{ (m+1)(b-\phi') + \phi'}{(m+1) (b-\phi')^2 + (\phi')^2}\leq \frac{ (m+1)^{1/2} \left( (m+1)^{1/2} (b-\phi') + \phi'\right)}{(m+1) (b-\phi')^2 + (\phi')^2} \\
&\leq& \frac{2 (m+1)^{1/2}} {(m+1)^{1/2} (b-\phi') + \phi'} \leq  \frac{4m+1}{b} .
\end{eqnarray*}
Therefore by the maximum principle, $H(\rho_0, t_0) \leq 0$ and so
$$\sup_{ (-\infty, \infty) \times [0, t_0)} H(\rho, t) \leq \sup_{(-\infty, \infty)} H(\rho, 0) < \infty.$$

Now we will prove the lower bound for $H$.  Suppose that $H(t_0, \rho_0) = \inf_{(-\infty, \infty) \times  [0, t_0] } H(\rho, t)$  for some $t_0\in (0, T)$, $\rho_0\in (-\infty, \infty)$.  At $(\rho_0, t_0)$, we have

$$0\geq \ddt{H} \geq  -   \frac{(m+1)\phi''}{(\phi')^2}  - \frac{\phi''}{(b-\phi')^2} - \frac{n\phi''}{(a+\phi')^2}  + \frac{m+1}{\phi'} + \frac{1}{b- \phi'} - \frac{m+2}{b} , $$
and so
\begin{eqnarray*}
\left( \frac{m+n+1}{(\phi')^2} + \frac{1}{(b-\phi')^2} \right) \phi'' &\geq& \left( \frac{m+1}{(\phi')^2} +\frac{n}{(a+\phi')^2} + \frac{1}{(b-\phi')^2} \right) \phi'' \\
&\geq& \frac{m+1}{\phi'} + \frac{1}{b-\phi'}- \frac{m+2}{b}\\
&=& \frac{(m+1)(b-\phi')^2 + (\phi')^2 }{b \phi' (b-\phi')}.
\end{eqnarray*}
Hence at $(t_0, \rho_0)$, we have
\begin{eqnarray*}
e^H &\geq& \frac{(m+1)(b- \phi')^2+ (\phi')^2 }{(\phi')^2+ (m+n+1)(b-\phi')^2} \geq \frac{m+1}{m+n+1}.
\end{eqnarray*}
We have now completed the proof for estimate (\ref{calest1}).

We will now prove estimate (\ref{calest2}) Let $H=\frac{\phi'''}{\phi''}. $ The evolution of $H$ is given by
\begin{eqnarray*}
\ddt{H} &=&\frac{H''}{\phi''} + \left( \frac{m\phi''}{\phi'} + \frac{n\phi''}{a+\phi'} - \frac{\phi'''}{\phi''} \right) \frac{H'}{\phi''} \\
&& - \left( 2m \frac{(\phi'')^2}{(\phi')^2} + 2n \frac{(\phi'')^2 }{ (a+\phi')^2} \right) \frac{H}{\phi''} + \frac{1}{\phi''}\left(2m \left( \frac{\phi''}{\phi'} \right)^2 + 2n \left( \frac{\phi''}{a+\phi'} \right)^2\right).
\end{eqnarray*}
Since $\frac{\phi''}{\phi'}$ is uniformly bounded by (\ref{calest1}), the bound for $H$ immediately follows by the maximum principle and we have completed the proof of the proposition.
\end{proof}


\subsection{The case of contractions}

In this section, we consider the case of
$$a(T)=0, ~b(T)>0. $$
For simplicity, we can assume $T=1$ after rescaling.  The limiting Kahler class $[g_0]+K_X$ for $[g(t)]$ is a big and semi-ample class that induces the unique birational morphism $\Phi: X \rightarrow Y$ by contracting the zero section $P_0$ to a point. Let $s= -\ln (1-t)$.  Then
$$\tilde g (s) = e^s g(t)$$
is the solution of the normalized K\"ahler-Ricci flow
$$\frac{\partial \tilde g (s) }{\partial s} = - Ric(\tilde g (s) ) + \tilde g (s) . $$

Let $y_0=\Phi(P_0)$. We let
$$\omega_Y =b \ddbar \log (1+e^\rho) \in [g_0]+K_X $$ be the pullback of a smooth Kahler metric on $Y$ as a trivial extension through $P_0$ and $D_\infty$. It obviously satisfies the Calabi symmetry and we choose $u_0$ as in (\ref{defofu0}) which also satisfies the Calabi symmetry. By (\ref{ceofu0}) and the maximum principle, there exists $C_0>0$ such that $$|u_0| \leq C_0 e^s$$ on $X\times [0, \infty)$. Let $\eta_A(\rho)$ be a smooth positive increasing function defined by
\begin{equation}\label{weightcs}
\eta_A(\rho) = \left\{
\begin{array}{ll}
b   \log (1+e^\rho), & \rho\leq 2A\\
b   \log (1+e^{4A}) & \rho \geq 4A.
\end{array} \right.
\end{equation}
Then for a fixed sufficiently large $A>0$,
$$\inf_{\rho\in \mathbb{R}} \left( u_0(\rho, s) + e^s\eta_A(\rho) \right)$$ must be achieved for $\rho\in (-\infty, A)$ for all $s\in \mathbb{R}$.
We now fix a sufficiently large $A>0$ satisfying the above and let
$$\theta_Y= \omega_Y - \ddbar \eta_A.$$
$\theta_Y$ vanishes in the open subset $\{ \rho < 2A \}$ of $X$.
 Let $$\tilde v(\rho, s) = (u_0(\rho, s) + e^s \eta_A(\rho) ) -  \inf_{\mathbb{R}}\left(u_0(\cdot, s) + e^s \eta_A(\cdot) \right)+1.$$ Then $\tilde v$ is the weighted Ricci potential associated to $\theta_Y$ and satisfies
$$\ddbar \tilde v = - Ric(\tilde g (s) ) + \tilde g (s) - e^s \theta_Y. $$
In particular, for all $s\in [0, \infty)$ we have
$$\inf_{\mathbb{R}}\tilde v(\cdot, s)=\tilde v(\rho_{min}, s), ~{\rm for ~some~} \rho_{min}\in (-\infty, A), $$
$$\ddbar \tilde v = - Ric(\tilde g (s) ) + \tilde g (s), ~{\rm for~all~} \rho\in (-\infty, 2A).$$
The above choice of $\theta_Y$ and $\tilde v$ will be used to show that the Ricci vertex must be sufficiently close to $P_0$. One  advantage is that $\theta_Y$ will always vanish in a fixed large neighborhood of the Ricci vertex.

We let
$$\Psi: X\setminus\{P_0 \cup D_\infty\} \rightarrow \mathbb{R}$$
be the map associated to $\rho$ by the $U(m+1)$-action. We then can identify $\tilde v(\cdot, s)$ as a function in $\rho$ on $\mathbb{R}$ with $\tilde v( \Psi^{-1}(\cdot), s)$.

At each $s\in [0, \infty)$, there is a Ricci vertex $p_s \in X$ such that $\tilde v$ achieves its minimal point at $p_s$. Furthermore, by our gradient and Laplacian estimates, $\tilde v$ has quadratic growth in distance from $p_s$ and the scalar curvature is uniformly bounded in any geodesic ball centered at $p_s$ with fixed radius.

Due to Lemma \ref{a'1}, we can choose the weighted Ricci potential $\tilde v$ in a way, such that for each $s\in [0, \infty)$, $\Psi(p_s)\leq 0$, (in fact, as small as we want).

We  let
$$\tilde a(s)=e^s a(t), ~\tilde b(s)=e^s b(t).$$
In fact, we can check that $\tilde a(s) = a(0)$ for all $s\geq 0$.  Let
$$\rho_s = \Psi(p_s)\leq 0,$$
and let the potential for $\tilde g(s) =a(0)\omega_Z+ \ddbar \tilde \phi(s)$ be given by
$$\tilde \phi  (\rho, s) =  e^s \phi (\rho, t), ~s= -\ln (1-t).  $$
Then $ \tilde \phi$ is smooth  away from $P_0$ and $P_\infty$ in $X$. The weighted Ricci potential $\tilde v$ satisfies the equation
\begin{equation}\label{ricpeqn}
\Delta_{\tilde g} \tilde v= n - R(\tilde g) + tr_{\tilde g} (e^s \theta_Y).
\end{equation}
Furthermore, the Schwarz lemma implies that $tr_{\tilde g} (e^s \theta_Y)$ is uniformly bounded above on $X\times [0, \infty)$.
%
%
%

We note that the K\"ahler metrics with Calabi symmetry along the Ricci flow is a warp product metric expressed in the following formula
\begin{equation}\label{warppro}
\tilde g (s) = (\tilde a(s) + \tilde \phi'(\rho) ) g_{Z} + \tilde \phi'(\rho) g_{\mathbb{CP}^m} + \tilde \phi''(\rho) g_{cyl},
\end{equation}
where $g_{cyl}= d\rho^2 + \eta^2$, $\eta$ is the contact one form associated to the $S^1$-bundle over $Z\times \mathbb{CP}^m$. $\tilde g$ is the warp product of a line and $S^1$-bundle over $\mathbb{CP}^m\times Z$. Such an $S^1$-bundle is the link of the contraction of $P_0$ of the normal bundle of $P_0$.

We first prove the following estimates for $\tilde\phi'$ and $\tilde \phi''$.

\begin{lemma}\label{caldifin}
For any $K>0$, there exists $A>0$ such that for all $s\geq 0$, $\rho_0 \leq K$,  we have
$$
e^{A^{-1} \rho} \leq \frac{\tilde\phi'(\rho+\rho_0, s)}{\tilde\phi'(\rho_0, s)} \leq  e^{A \rho},~ ~ for~ 0\leq \rho \leq K.
$$
$$
e^{A \rho } \leq \frac{\tilde\phi'(\rho+\rho_0, s)}{\tilde\phi'(\rho_0, s)} \leq  e^{A^{-1} \rho},~ for ~ \rho\leq 0,
$$
$$
e^{- A |\rho| } \leq \frac{\tilde\phi''(\rho+\rho_0, s)}{\tilde\phi''(\rho_0, s)} \leq  e^{A |\rho| },~ for ~ all ~ \rho .
$$
\end{lemma}
\begin{proof}   We note that $b(t) -\phi(\cdot, t)$ is uniformly bounded below from $0$ for all $t$ on any fixed compact set of $\mathbb{R}$. By Proposition \ref{2ndder1},  for any $K\geq 0$, there exists $A>0$ such that for all $s\geq 0$ and $\rho \leq K$, we have
$$
A^{-1} \leq \frac{ \tilde\phi''(\rho, s)}{  \tilde\phi'(\rho, s)}\leq A  .
$$
There exists $B>0$, such that for all $s\geq 0$,
$$
-B  \leq  \frac{ \tilde\phi'''(\rho,s)}{  \tilde\phi''(\rho,s)}\leq B.
$$
The lemma immediately follows from the above  differential inequalities.
\end{proof}
\begin{lemma} \label{uup}
There exists $C>0$ such that for all $s\geq 0$,
$$\tilde \phi' (\rho_s, s) \leq C. $$
\end{lemma}
\begin{proof}
Suppose $\tilde \phi'(\rho_{s_j}, s_j)=K_j \rightarrow \infty$ with $s_j \rightarrow \infty$ as $j\rightarrow \infty$. There exists $A>0$ such that
\begin{equation}\label{lem621}
K_j e^{-A |\rho |} \leq \tilde \phi'(\rho+\rho_{s_j} , s_j)  \leq  K_j e^{A |\rho |}.
\end{equation}
Let
$$
W_j=\{ x \in X~|~| \Psi(x) - \rho_{s_j} |< 1\} .
$$
Since $K_j \rightarrow \infty$,  (\ref{lem621}) implies that $\tilde \phi'(\rho, s_j)$ tends to $\infty$ uniformly on $\rho\in (\rho_{s_j} -1, \rho_{s_j} +1)$ as $j \rightarrow \infty$. Then by Proposition \ref{2ndder1} and the fact $\rho_{s_j}\leq 0$, we have $\tilde \phi''(\rho, s_j)$ tends to $\infty$ uniformly on $\rho\in (\rho_{s_j} -1, \rho_{s_j} +1)$ as $j \rightarrow \infty$. Then for any $D>0$, there exits $j_D>>1$ such that if $j>j_D$, we have
$$B_{\tilde g(s_j)}(p_{s_j}, D) \subset W_j $$
from (\ref{warppro}).
By the choice of $p_{s_j}$, the scalar curvature of $\tilde g(s_j)$ is uniformly bounded in $B_{\tilde g(s_j)}(p_{s_j}, D)$ for any fixed $D\geq 1$. Therefore by (\ref{scurv}), there exists $C=C(D)>0$ such that in $B_{\tilde g(s_j)}(p_{s_j}, D)$, we have
$$\left| \frac{\tilde \phi^{(4)}}{(\tilde \phi'')^2}\right| \leq C$$
for all $j$.
It in turns implies that both the scalar curvature and the curvature tensor of $\tilde g(s_j)$ are uniformly bounded in $B_{\tilde g(s_j)}(p_{s_j}, D)$.

By Perelman's pseudolocality theorem, there exists $\delta=\delta(D)>0$ such that the curvature of $\tilde g$ is uniformly bounded on
$$B_{\tilde g(s_j)}(p_{s_j}, D) \times [s_j-\delta, s_j+\delta].$$
We now show that the curvature in fact tends to $0$ uniformly.

Let
$$l_j = \frac{\tilde \phi''(\rho_{s_j} , s_j)}{ \tilde \phi'(\rho_{s_j}, s_j) }.$$
By Proposition \ref{2ndder1}, $l_j$ is uniformly bounded above and below away from $0$. By Lemma \ref{caldifin}, the metric near $p_{s_j}$ is bounded by
\[
\begin{split}
e^{-A|\rho|} \big( (K_j^{-1} a(0)+ 1) g_Z +  g_{\mathbb{CP}^m} + & l_j g_{cyl} \big) \leq  K_j^{-1} \tilde g(\rho+\rho_{s_j}, s_j) \\
& \leq  e^{A|\rho|} \left( (K_j^{-1} a(0) + 1) g_Z +  g_{\mathbb{CP}^m} +  l_j g_{cyl} \right).
\end{split}
\]
We remark that the metric $(K_j^{-1} a(0)+ 1) g_Z +  g_{\mathbb{CP}^m} +  l_j g_{cyl}$ is a warp product metric on $X$ but it is not a Kahler metric compatible to the complex structure on $X$. After taking a subsequence, we can assume  for some $l_\infty>0$, we have
$$
(K_j^{-1} a(0)+ 1) g_Z +  g_{\mathbb{CP}^m} +   l_j g_{cyl} \rightarrow g_Z +  g_{\mathbb{CP}^m} +   l_\infty g_{cyl} .
$$
%
%

%
Therefore $(X, \tilde g (s_j), p_{s_j})$ converges in $C^0$ locally to the tangent space of
$(Z\times \mathbb{CP}^m \times \mathbb{C}^*, g_Z +  g_{\mathbb{CP}^m} +  l_\infty g_{cyl})$.  Then the convergence must also be in pointed Cheeger-Gromov sense due to the curvature bounds.

%
%
%
%
%

%
%
%

%
Hence $(X, \tilde g (s_j), p_{s_j})$ converges smoothly to the flat $(\mathbb{C}^{m+n+1}, \tilde g_\infty)$. By the linear estimates for equation (\ref{ricpeqn}), i.e., $\Delta_{\tilde g} \tilde v= n -R(\tilde g)$ since $\theta_Y$ vanishes near $P_0$, we have
$$\tilde v (\rho+\rho_s, s) - \tilde v (\rho_s, s)= \tilde v (\rho+\rho_s, s) - \inf_{\mathbb{R}} \tilde v (\cdot, s)$$
converges smoothly to $\tilde v_\infty$ based at $\rho_s$.   The limiting equation for $\tilde v_\infty$ is then given by
\begin{equation}\label{ric0}
0=Ric(\tilde g_\infty)= \tilde g_\infty -\ddbar \tilde v_\infty .
\end{equation}
Note that the flat space is a trivial $\mathbb{C}^{m}$-bundle over $\mathbb{C}^{n+1}$ associated to the blow-up of the cylinder and $Z$. Since $\tilde v$ is $U(m+1)$-invariant , it is constant on the $S^1$-bundle over $\mathbb{CP}^m$. Therefore the limit $\tilde v_\infty$  is constant on each $\mathbb{C}^{m}$-fibre over $\mathbb{C}^{n+1}$ in $\mathbb{C}^{m+n+1}$. Restricting (\ref{ric0}) to any such fibre $\mathbb{C}^{m}$, we have
$$\tilde g_\infty|_{\mathbb{C}^{m} }= 0.$$
This is a contradiction.
\end{proof}

\begin{corollary}  \label{rhos}
$\lim_{s\rightarrow \infty} \rho_s = -\infty$ .
\end{corollary}
\begin{proof}
Since $\tilde \phi'(\rho_s, s)$ is uniformly bounded above, $\phi'(\rho_s, s)$ must tend to $0$ as $s\rightarrow \infty$. Since $\phi'$ converges to a positive increasing function smoothly on any compact subset of $\mathbb{R}$, we immediately can conclude  that $\rho_s$ tends to $-\infty$ as $s\rightarrow \infty$.
\end{proof}

\begin{corollary} \label{dist}
There exists $C>0$ such that for all $s>0$,
$$d_{\tilde g(s)}(p_s, P_0) \leq C . $$
\end{corollary}
\begin{proof}
By Lemma \ref{uup}, and Corollary \ref{rhos}, there exists $A>0$ such that for all $s>0$ and $\rho\leq \rho_s$, we have
$$
\tilde \phi'(\rho, s) \leq \tilde \phi'(\rho_s, s) e^{A^{-1} (\rho-\rho_s)} \leq C e^{A^{-1} (\rho-\rho_s)}  .
$$
By Proposition \ref{2ndder1}, we have $\tilde \phi''\leq C \tilde \phi'$ globally. Hence we have
$$
d_{\tilde g(s)}(p_s , P_0) \leq \int_{-\infty}^{\rho_s} \sqrt{ \tilde \phi''(\rho,s)} d\rho \leq C \int_{-\infty}^{0} e^{(2A)^{-1} \rho} d\rho \leq C,
$$
by picking the geodesic line along the radial direction of the warp product metric $\tilde g$.
\end{proof}

Let
$$ \mathcal{B}_{\tilde g(s)}(P_0, d)= \{ p: d_{\tilde g(s)}(p, P_0)< d\}$$
be the tubular neighborhood of $P_0$ of radius $d>0$ with respect to $\tilde g(s)$.

\begin{lemma} \label{calpolow}
For any $ d>0$, there exists $c=c(d)>0$, such that for all $s\geq 0$, we have
$$\left. \tilde \phi' (\cdot , s) \right|_{\partial \mathcal{B}_{\tilde g(s)}(P_0, d)} \geq c . $$
\end{lemma}
\begin{proof} Let $\rho_{d, s} =\Psi(q)$ corresponding to any point $q \in \partial \mathcal{B}_{\tilde g(s)}(P_0, d)$. The volume of $\mathcal{B}_{\tilde g(s)}(P_0, d)$ can be estimated by
\begin{eqnarray*}
&&Vol_{\tilde g(s)}(\mathcal{B}_{\tilde g(s)}(P_0, d)) \\
&=& [\omega_Z]^n \int_{-\infty}^{\rho_{d,s}} (\tilde a(s)+\tilde \phi')^n (\tilde \phi')^m \tilde \phi'' d\rho\\
&\leq& Ca(0) [\omega_Z]^n \int_{-\infty}^{\rho_{d,s}}  (\tilde \phi')^m \tilde \phi'' d\rho + C\int_{-\infty}^{\rho_{d,s}}  (\tilde \phi')^{m+n}  \tilde \phi'' d\rho\\
&=& Ca(0)(m+1)^{-1} [\omega_Z]^n\int_{-\infty}^{\rho_{d,s}} \left( (\tilde \phi')^{m+1} \right)' d\rho+ C \int_{-\infty}^{\rho_{d,s}} \left( (\tilde \phi')^{m+1} \right)' d\rho\\
&\leq& Ca(0) (m+1)^{-1}  [\omega_Z]^n(\tilde \phi' (\rho_{d,s},s) )^{m+1} +  C (\tilde \phi' (\rho_{d,s},s) )^{m+n+1}.
\end{eqnarray*}
On the other hand, if we pick any point $p_0\in P_0$, $ B_{\tilde g(s)}(p_0, d)) \subset \mathcal{B}_{\tilde g(s)}(P_0, d) $ and $d_{\tilde g(s)}(p_0, p_s)$ is uniformly bounded above. Therefore by our scalar curvature estimate and $\kappa$-noncollapsing, the volume of $  B_{\tilde g(s)}(p_0, d)) $ must have a uniform lower bound. This immediately implies the uniform lower bound of $\tilde \phi' (\rho_{q,s},s)$.
\end{proof}

We now pick $q_s$ be a point satisfying
\begin{equation}\label{unidistpt}
d_{\tilde g(s)}(q_s, P_0) = 1, ~ \rho_{1,s} =\Psi (q_s).
\end{equation}

\begin{corollary} \label{locc1}
For any compact subset $K$ of $\mathbb{R}$ and $k>0$, there exists $C>0$ such that
$$ \inf_{(\rho,s)\in K\times[0, \infty)} \tilde \phi' (\rho+\rho_{1,s}, s) \geq C^{-1},$$
$$\inf_{(\rho,s)\in K\times[0, \infty)} \tilde \phi''(\rho+\rho_{1,s} , s) \geq C^{-1},$$
$$||\tilde \phi' (\rho+\rho_{1,s}, s)||_{C^k(K\times[0, \infty))} \leq C. $$
\end{corollary}
\begin{proof}
Both of the lower bounds for $\tilde \phi'$ and $\tilde \phi''$ follow from Lemma \ref{calpolow} and Lemma \ref{caldifin}. Then $\tilde \phi'''$ is uniformly bounded and $\tilde \phi^{(4)}$ is also uniformly bounded by the scalar curvature bound and the scalar curvature formula.
Then the curvature tensors are uniformly bounded which implies higher order estimates for $\tilde \phi'$.
\end{proof}

Immediately we have the following corollary.
\begin{corollary}
For any sequence $s_j\rightarrow \infty$, (after possibly passing to a subsequence), $(X\setminus \{ P_0\cup D_\infty\} , \tilde g(s_j), q_{s_j})$ converges smoothly to a shrinking gradient soliton on $(X\setminus \{ P_0\cup D_\infty\}, \tilde g_\infty)$.
\end{corollary}
\begin{proof}
Let $\hat \phi_j(\rho) = \tilde\phi(\rho+\rho_{1,s_j} , s_j) -\tilde\phi(\rho_{1,s_j}, s_j)$. By Corollary \ref{locc1}, $\hat \phi_j$ converges smoothly to a smooth convex function $\hat\phi_\infty(\rho)$ for $\rho\in \mathbb{R}$.  The K\"ahler metric
$$
\tilde g_\infty = \ddbar \hat \phi_\infty = (a(0) + \hat\phi_\infty'(\rho))g_Z + \hat \phi_\infty'(\rho) g_{\mathbb{CP}^m} + \hat \phi_\infty''(\rho) g_{cyl}
$$
is a smooth K\"ahler-Ricci soliton metric for $\rho\in \mathbb{R}$. Hence it is a K\"ahler-Ricci soliton metric on the open manifold $X\setminus \{ P_0\cup D_\infty\} $.
\end{proof}

In the next lemma, we will prove the uniform curvature bounds for $\tilde g$ near $P_0$.
\begin{lemma}\label{RmboundP0}
For any $D>0$, there exists $C>0$ such that for  $s\geq 0$
$$\sup_{\mathcal{B}_{\tilde g(s)}(P_0, D)} |Rm(\tilde g(s))| \leq C. $$
\end{lemma}
\begin{proof}
It suffices to prove for $D=1$. We prove by contradiction. We first note that the curvature of $\tilde g (\rho+\rho_{1,s}, s)$ is uniformly bounded for all $s\geq 0$ and $\rho$ on any compact subset of $\mathbb{R}$ by the previous two corollaries. Let $(x_j, s_j)$ be the quasi-maximal point for the curvature, i.e.,
$$ K_j=|Rm(x_j, s_j)| \geq  \sup_{(\rho, s)\in(-\infty, 0] \times [0, s_j]} |Rm(\tilde g(\rho+\rho_{1, s_j}, s_j))| -1\rightarrow \infty. $$
We can assume $x_j$ does not lie in $P_0$. Let $\rho_{x_j} =\Psi(x_j) $ be the value of $\rho$ associated to $x_j$. We note that $\rho_{x_j}\to -\infty$ as $j\to\infty$. We consider the rescaled metric $\tilde g_j = K_j \tilde g(s_j)$ with potential $\tilde\phi_j (\rho, s_j)= K_j \tilde \phi(\rho, s_j)$.

We break the proof into following cases.

\noindent {\bf Case I:}
Suppose $\tilde \phi'_j(\rho_{x_j}, s_j) \rightarrow \infty$ as $j\to\infty$.

By Lemma \ref{caldifin}, near $x_{j}$ we have
\[
\begin{split}
e^{-A|\rho|} \big( (L_j a(0)+ 1) g_Z  +  g_{\mathbb{CP}^m} + & l_j g_{cyl} \big) \leq \tilde \phi'_j(\rho_{x_j}, s_j)^{-1} \tilde g_j(\rho+\rho_{x_j}) \\
& \leq  e^{A|\rho|} \left( (L_j a(0)+ 1) g_Z +  g_{\mathbb{CP}^m} +  l_j g_{cyl} \right),
\end{split}
\]
where $l_j=\tilde \phi''(\rho_{x_j}, s_j)/\tilde \phi'(\rho_{x_j}, s_j)$ and $L_j=K_j/\tilde \phi'_j(\rho_{x_j}, s_j)$. Since $\rho_{x_j}\to -\infty$, by Proposition \ref{2ndder1}, $l_j$ is uniformly bounded above and below away from $0$. Then we can make use of the scalar curvature bound and the same argument in the proof of Lemma \ref{uup} to show that $(X, \tilde g_j, x_j)$ converges in pointed Cheeger-Gromov sense to flat $\mathbb{C}^{n+m+1}$. This leads to contradiction as the norm of the curvature tensor at $x_j$ is $1$.

\medskip

\noindent {\bf Case II:}
Suppose $\tilde \phi'_j(\rho_{x_j}, s_j) \rightarrow 0$ as $j\to\infty$.

By Lemma \ref{caldifin} and the same argument of Corollary \ref{dist}, $d_{\tilde g_j}(x_j, P_0)\to 0$. Hence by the choice of $x_j$ and Hamilton's compactness theorem, passing to a subsequence, $(X, \tilde g_j, x_j)$ smoothly converges to a smooth limit $(\tilde X_\infty, \tilde g_\infty, x_\infty)$. Hence we can find $\delta>0$ such that there exists $\epsilon>0$, such that for any point $q \in X$ with $d_{\tilde g_j}(x_j, q) <\delta$, we have
$$|Rm(\tilde g_j)|(q) \geq \epsilon.$$
This can also be proved by Shi's derivative estimates.

We then let $y_{j} \in \Psi^{-1}(\rho_{y_{j}})$ be the point on the radial geodesic of $x_j$ with $d_{\tilde g_j}(x_j, y_{j})=\delta$ and $\rho_{y_{j}} > \rho_{x_j}$. We claim that there exists $c>0$, such that for each $j>0$, we have
$$\tilde \phi_j'(\rho_{y_j}, s_j) \geq c.$$
Otherwise we have $\tilde \phi_j'(\rho_{y_j}, s_j) \to 0$, then by the same argument of Corollary \ref{dist}, we have $d_{\tilde g_j}(x_j, y_{j})\to 0$, which is impossible. Indeed, we have $\rho_{y_j}\to -\infty$ as $j\to\infty$. Hence by Proposition \ref{2ndder1}, we have for some $A>0$
$$
\tilde\phi_j'(\rho+\rho_{y_j}, s_j)\leq \tilde\phi_j'(\rho_{y_j}, s_j)e^{A^{-1} \rho}, ~ for ~ \rho\leq 0,
$$
hence $\tilde\phi_j''(\rho+\rho_{y_j}, s_j)\leq C\tilde\phi_j'(\rho_{y_j}, s_j)e^{A^{-1} \rho}$ for all $\rho\leq 0$. Hence we have
$$
d_{\tilde g_j}(x_j , y_j) \leq \int_{-\infty}^{\rho_{y_j}} \sqrt{ \tilde \phi_j''(\rho,s)} d\rho \leq C\tilde\phi_j'(\rho_{y_j}, s_j),
$$
by picking the geodesic line along the radial direction of the warp product metric $\tilde g_j$.

Hence by the continuity of $\tilde \phi_j'$, we can find $z_j$ on the radial geodesic of $x_j$ such that
$$\tilde \phi'_j(\rho_{z_j}, s_j)=c,~ \rho_{z_j}=\Psi(z_j)\in (\rho_{x_j}, \rho_{y_j} ] .$$
By Lemma \ref{caldifin} and the same argument of Corollary \ref{dist}, $d_{\tilde g_j}(z_j, P_0)$ is uniformly bounded. By Lemma \ref{caldifin} and Proposition \ref{2ndder1}, there exist $C>0$ such that for all $\rho \in (-1, 1)$,
$$C^{-1} \leq \tilde\phi_j'(\rho_{z_j} + \rho) \leq C, ~C^{-1} \leq \tilde\phi_j''(\rho_{z_j} + \rho) \leq C, ~|\tilde \phi_j'''(\rho_{z_j}+\rho)| \leq C, $$
then by the scalar curvature formula (\ref{scurv}) and the fact that the scalar curvature of $\tilde g_j$ tend to zero on the region under consideration, we have $|\tilde \phi_j^{(4)}(\rho_{z_j}+\rho)| \leq C$.

Then after passing to a subsequence $f_j(\rho):=\tilde\phi_j'(\rho_{z_j} +\rho) $ converges in $C^{2,\alpha}$-topology to some function $f_\infty(\rho)$ for $\rho \in (-1,1)$, and so the warp product metrics
\begin{equation}\label{wme3}
\tilde g_j(\rho_{z_j}+\rho)=\left(K_j a(0) +  \tilde \phi'_j(\rho_{z_j}+\rho)\right) g_Z + \tilde\phi'_j(\rho_{z_j}+\rho)g_{\mathbb{CP}^m} + \tilde \phi''_j(\rho_{z_j}+\rho) g_{cyl},
\end{equation}
converge to
$$\hat g_\infty (\rho) = g_{\mathbb{C}^{n} }+ f_\infty(\rho)g_{\mathbb{CP}^m}+f_\infty'(\rho) g_{cyl},~~ \rho \in (-1,1).$$
Here $\theta_\infty(\rho):=f_\infty(\rho)g_{\mathbb{CP}^m}+f_\infty'(\rho) g_{cyl}$ is a smooth $U(m+1)$-invariant  non-flat K\"ahler metric satisfying the $U(m+1)$-symmetry. In particular,
$$ \theta_\infty = \ddbar \tilde \phi_\infty$$
for some smooth $U(m+1)$-invariant $\tilde \phi_\infty$. We can identify $\tilde \phi_\infty \in C^\infty((-1,1)))$ by the moment map associated to $\theta_\infty$. By Corollary \ref{dist}, the scalar curvature of $\tilde g_j$ tends to $0$,  $\theta_\infty$ must be Ricci flat for $\rho \in (-1,1)$. Let
$$\tilde v_\infty = \log \left( (\tilde \phi'_\infty)^m \tilde \phi_\infty''  \right) - (m+1) \rho$$
be the Ricci potential for $\tilde \phi_\infty$. Then $\tilde v_\infty$ must be pluriharmonic on $\rho \in (-1,1)$, i.e.,
$$\ddbar \tilde v_\infty (\rho) =0, ~~\rho \in (-1,1).$$
This is equivalent to
$$\tilde v_\infty' =\tilde v_\infty''=0$$
and so
$$ e^{-(m+1))\rho} (\tilde \phi'_\infty)^m  \tilde\phi_\infty'' =C $$
for some constant $C$. This implies that
$$ \left( \left( \tilde \phi'_\infty \right)^{m+1} \right)' = C e^{(m+1)\rho}, $$ or
$$ \tilde \phi'_\infty = Ce^\rho.$$
Hence $\theta_\infty(\rho)$ must be flat for $\rho \in (-1,1)$. This contradicts to the fact that $|Rm(\tilde g_j)|(z_j) \geq \epsilon$.

\medskip

\noindent {\bf Case III:} Suppose $\tilde \phi_j'(\rho_{x_j}, s_j)$ is uniformly bounded above and below away from $0$ for all sufficiently large $j$. Then we can apply the argument of the latter part of Case II (with $z_j$ been replaced by $x_j$) and show that $\tilde g_j$ converges to a flat metric near $x_\infty$. Contradiction.
\end{proof}
%

%
%
%
%

%
Finally, we will prove the curvature of $\tilde g$ is uniformly bounded globally on $X$ and classify the limits of $\tilde g(s)$ as $s\rightarrow \infty$.

\begin{theorem}
There exists $C>0$ such that
$$\sup_{X\times [0, \infty)}|Rm(\tilde g)| \leq C. $$
Furthermore, for any $p \in P_0$, $(X, \tilde g(s))$ converges smoothly to the unique shrinking gradient K\"ahler-Ricci soliton metric with Calabi symmetry on the total space of the vector bundle $ L^{\oplus (m+1)}$ over $Z$.
For any $p\in X\setminus P_0$, $(X, \tilde g(s))$ converges smoothly to the flat Euclidean space $\mathbb{C}^{m+n+1}.$
\end{theorem}
\begin{proof}
We prove by contradiction. We first note that  the curvature of $\tilde g (\rho+\rho_{1,s}, s)$ is uniformly bounded for all $s\geq 0$ and $\rho \subset (-\infty, K)$ for any fixed $K>0$. We also can assume that $b- \phi' $ is uniformly bounded below away from $0$ since the curvature is uniformly bounded for the unnormalized flow when $\phi'$ is close to $b$.

Let $(x_j, s_j)$ be maximal point for the curvature, i.e.,
$$
K_j=|Rm(x_j, s_j)| = \sup_{(\rho, s)\in \mathbb{R} \times [0, s_j]} |Rm(\tilde g(\rho+\rho_{1,s}, s_j))| \rightarrow \infty.
$$
Let $\rho_{x_j} $ be the value of $\rho$ associated to $x_j$, we may assume that $\rho_{x_j}\leq 0 $. Then by Lemma \ref{RmboundP0}, we have
$$
\rho_{x_j} - \rho_{1, s_j}  \rightarrow \infty, ~ \tilde \phi'(\rho_{x_j}, s_j ) \rightarrow \infty, ~\tilde \phi''(\rho_{x_j}, s_j ) \rightarrow \infty.
$$
Let
$$L_j(s) =  \tilde \phi'(\rho_{x_j}, s+ s_j ), ~M_j(s) =  \tilde \phi''(\rho_{x_j}, s+ s_j ). $$
Then $L_j(0), M_j(0) \rightarrow \infty$ as $j\rightarrow \infty$.
Since
$$\frac{\partial \tilde \phi'}{\partial s} = \ddt{\phi'} = \frac{\phi'''}{\phi''} + \frac{m\phi''}{\phi'} + \frac{n\phi''}{a+\phi'} - (m+1)$$
is uniformly bounded, there exists $C>0$ such that for $s\in (-1, 1)$,
$$ L_j(0) - C \leq L_j(s) \leq L_j(0) + C.  $$
Therefore for $s\in (-1, 1)$,  $L_j(s), M_j(s) \rightarrow \infty$ as $j\rightarrow \infty$.

We then rescale $\tilde \phi$ and $\tilde g$ by
%
%
%
$$
\hat \phi_j(\rho, s) = L_j(s)^{-1} \tilde \phi(\rho+\rho_{x_{j}}, s+ s_j),~ \hat g_j (\rho, s)= L_j(s)^{-1} \tilde g(\rho+\rho_{x_{j}}, s+s_j).
$$
Then for $\rho \in (-1, 1)$ and $s\in (-1 , 1)$, we have
$$
e^{-A|\rho|}  \leq      \hat \phi_j' (\rho, s) \leq e^{A|\rho|},  ~ l_j(s) e^{-A|\rho|}  \leq      \hat \phi_j '' (\rho, s) \leq l_j(s) e^{A|\rho|}  ,
$$
where $l_j(s):=\frac{M_j(s)}{L_j(s)}$ is uniformly bounded above and below from $0$, for each fixed $s\in (-1, 1)$,  after passing to a subsequence, we can assume that $l_j (s) \rightarrow l_{\infty, s}>0$ as $j\rightarrow \infty$. Then by (\ref{warppro}), we have
\[
\begin{split}
e^{-A|\rho|} & \left( (L_j(s)^{-1} a(0)+ 1) g_Z +  g_{\mathbb{CP}^m} +  l_j(s) g_{cyl} \right) \leq  \hat g_j(\rho, s) \\
& \qquad\qquad\qquad\qquad \leq  e^{A|\rho|} \left( (L_j(s)^{-1} a(0)+ 1) g_Z +  g_{\mathbb{CP}^m} +  l_j(s) g_{cyl} \right).
\end{split}
\]
Therefore for each $s\in (-1, 1)$ and any fixed radius $R>>1$,
$$( B_{\tilde g(s+s_j)}(x_j, R) \subset X , \tilde g(s+s_j))$$
is $C^0$-close to $\mathbb{R}^{2(n+m+1)}$ with the flat metric, the tangent space of $(X, g_Z \times g_{\mathbb{CP}^m} \times g_{cyl})$, for sufficiently large $j=j(s, R)>>1$. Therefore the Sobolev constant for $B_{\tilde g(s+s_j)}(x_j, R)$ is arbitrarily close to that of the Euclidean ball by choosing sufficiently large $j$.
By Perelman's pseudolocality theorem, there exists $C>0$ such that for all $s\in (-1/2, 1/2)$, we have the curvature bound
$$\sup_{B_{ \tilde g(s+s_j)}(x_j, 1)} |Rm(\tilde g(s+s_j))| \leq C.$$
In particular, the curvature tensor $Rm(x_j, s_j)$ of $\tilde g$ at $(x_j, s_j)$ is uniformly bounded for all $j$. This leads to contradiction.

We will now classify the limits of $(X, \tilde g(s), p)$ for a fixed base point $p$. The unnormalized flow converges smoothly to a K\"ahler metric $g_T$ in $X\setminus P_0$ as $t\rightarrow T$. If $p\notin P_0$,  $\tilde g(s)$ of the normalized flow converge to the tangent space of $g_T$ at $p$, which must be the flat $\mathbb{C}^{m+n+1}$ as $s\rightarrow \infty$.  If $p\in P_0$, then $(X, \tilde g(s))$ must converge in pointed Cheeger-Gromov sense to a smooth shrinking K\"ahler-Ricci soliton $(X_\infty, g_\infty)$ by our curvature estimates. We choose $\rho_{1,s}$ as in (\ref{unidistpt}) and let
$$\mathrm{g}(\rho, s) = \tilde g(\rho+\rho_{1,s}), s), ~\psi(\rho, s) =\tilde \phi(\rho+\rho_{1,s}, s).$$
Then $\mathrm{g}(\rho, s)$ converges smoothly with a fixed gauge and the limiting K\"ahler metric  $g_\infty (\rho)$ satisfies the same Calabi symmetry on the total space of $L^{\oplus(m+1)}$.   $X_\infty$ is either the total space of $L^{\oplus(m+1)}$ over $Z$ or $\{ \rho \leq \hat \rho\} \cap L^{\oplus(m+1)}$ for some fixed $\hat \rho\in \mathbb{R}$. Lemma \ref{caldifin} guarantees that $g_\infty(\rho)$ is defined for all $\rho\in \mathbb{R}$ and so $X_\infty$ must be the total space of $L^{\oplus(m+1)}$. Therefore $g_\infty$ is a complete shrinking K\"ahler-Ricci soliton on the total space of $L^{\oplus(m+1)}$ and it must be unique as proved in \cite{Li}.
\end{proof}


\subsection{The case of collapsing with extinction}

In this section, we consider the case $a(T)>0$ and $b(T)=0$. We will again assume $T=1$ and   the limiting cohomology class is the pullback of $-a(1)c_1(L)$. We will use the same normalized flow
$$\tilde g(s)=e^s g(t), ~\tilde\phi(\rho, s) = e^s \phi(\rho, t),  ~ s= -\ln (1-t).$$
We let
$$\tilde b_\infty = \tilde b(s) \equiv m+2.$$
\begin{lemma}
There exists $C>0$ such that for all $s\geq 0$, we have %
$$C^{-1} \leq \int_{-\infty}^{\infty} \sqrt{\tilde \phi''(\rho, s)} d\rho \leq C. $$
In particular,  there exists $C>0$ such that for all $p\in Z$,
$$C^{-1} \leq diam(F_p, \tilde g(s)|_{F_p}) \leq C,$$
where $F_p = \pi^{-1} (p)$ be the fibre over $p\in Z$.
\end{lemma}

\begin{proof}
The diameter bound of each fibre can be calculated below:
\begin{eqnarray*}
&&\int_{-\infty}^\infty  \sqrt{\tilde \phi''(\rho, s)} d\rho = \int_{-\infty}^\infty \left( \frac{\tilde \phi'}{\tilde \phi''} \right)^{1/2} (\tilde \phi')^{-1/2} \tilde \phi'' d\rho\\
&\leq&C \int_{-\infty}^\infty \left(  \frac{\tilde b(s)}{\tilde \phi' (\tilde b(s) - \tilde \phi')} \right)^{1/2} \tilde \phi'' d\rho \leq C \int_0^{\tilde b_\infty} \sqrt{ \frac{\tilde b_\infty}{ x(\tilde b_\infty-x)}} dx\leq C ,
\end{eqnarray*}
for a uniform constant $C<\infty$ for all $s \geq 0$.
We have completed the proof of the lemma.
\end{proof}

\begin{corollary}\label{gscb}
There exists $C<\infty$ such that
$$\sup_{X\times [0, \infty)} |R(\tilde g)| \leq C. $$
\end{corollary}
\begin{proof}
We pick any point $z\in Z$ and let $F_z$ be the fibre of $\pi: X \rightarrow Z$ over $z$. We then choose a twisted Ricci potential $\tilde v$ associated to $z$ or $F_z$. We remark that $\tilde v$ is not necessary $U(m+1)$-invariant.  Let $p_s$ be a Ricci vertex at $s$ for $\tilde v$. Since the fibre diameter of $\tilde g(s)$ is uniformly bounded for all $s$, the scalar curvature is uniformly bounded on the fibre containing $p_s$. Since the scalar curvature only depends on $\tilde \phi$, $\tilde \phi'$ and $\tilde \phi''$ by (\ref{scurv}), the scalar curvature is uniformly bounded for each fibre, hence everywhere on $X$.
\end{proof}
%


%
For each $s\geq 0$, we let $\rho_s\in \mathbb{R}$ be the point where
$$\tilde\phi'(\rho_s, s) = \frac{\tilde b(s)}{2}. $$

\begin{lemma} \label{colphiab}
There exists $C>0$ such that for all $\rho\in \mathbb{R}$ and $s\geq 0$, we have
$$\tilde\phi'(\rho, s) \leq C, ~ \tilde\phi''(\rho, s) \leq C, ~~\left| \tilde\phi'''(\rho, s)\right| \leq C. $$
\end{lemma}
\begin{proof}
$\tilde\phi'(\rho, s) \leq \tilde b(s)$ is uniformly bounded above for all $\rho\in \mathbb{R}$ and $s\geq 0$. Then by Proposition \ref{2ndder1}, $\tilde\phi''$ and $|\tilde\phi'''|$ are uniformly bounded above as well.
\end{proof}

\begin{lemma} \label{colphilow}
There exists $C<\infty$, such that for all $\rho\in \mathbb{R}$ and $s\geq 0$, we have
$$\tilde\phi'(\rho+\rho_s, s) \geq C^{-1}e^{-C|\rho|}, ~ \tilde\phi''(\rho+\rho_s, s) \geq C^{-1}e^{-C|\rho|}.$$
\end{lemma}
\begin{proof}
By Proposition \ref{2ndder1}, there exists $C>0$ such that
$$0\leq \left( \log \tilde\phi'(\rho, s) \right)'= \frac{\tilde\phi''(\rho, s)}{\tilde\phi'(\rho, s)} \leq C  \frac{\tilde b(s) - \tilde\phi'(\rho, s)}{\tilde b(s)} \leq C.$$
Therefore for all $\rho\in \mathbb{R}$, we have
$$\tilde\phi'(\rho+ \rho_s, s) \geq e^{-C|\rho|} \tilde\phi'(\rho_s, s)= \frac{\tilde b(s)}{2}e^{-C|\rho|}.$$
This proves the estimate for $\tilde\phi'$.

Next, by Proposition \ref{2ndder1}, we have
$$\tilde\phi''(\rho_s, s) \geq C^{-1}  \frac{\tilde b(s) - \tilde\phi'(\rho_s, s)}{\tilde b(s)} \tilde\phi'(\rho_s, s) = \frac{\tilde b(s)}{4}C^{-1}.$$
Hence by Proposition \ref{2ndder1} again, we have
$$\tilde\phi''(\rho+ \rho_s, s) \geq e^{-C|\rho|} \tilde\phi''(\rho_s, s)\geq C^{-1}e^{-C|\rho|}.$$
This proves the estimate for $\tilde\phi''$.
\end{proof}
\begin{lemma}\label{lRmbound}
For any $K\subset \subset \mathbb{R}$, there exists $C=C(K)<\infty$ such that
$$\sup_{(\rho, s) \in K\times [0, \infty)} |Rm\left(\tilde g(\Psi^{-1}(\rho+\rho_s), s)\right)| \leq C. $$
\end{lemma}
\begin{proof}
Since the scalar curvature is uniformly bounded globally, the curvature tensor $Rm(\Psi^{-1}(\rho+\rho_s), s)$ is uniformly bounded on $K \times [0, \infty)$ if $\tilde\phi'(\Psi^{-1}(\rho+\rho_s), s)$ and $\tilde \phi''(\Psi^{-1}(\rho+\rho_s), s)$ are uniformly bounded both from above and from below away from $0$ on $K\times [0, \infty)$. The lemma then immediately follows from Lemma \ref{colphiab} and Lemma \ref{colphilow}.
\end{proof}

\begin{theorem}
There exists $C>0$ such that
$$ \sup_{X\times [0, \infty)} |Rm(\tilde g)| \leq C. $$
Furthermore, for any point $p\in X$, $(X, \tilde g(s))$ converges smoothly to $$\left( \mathbb{C}^n \times \mathbb{CP}^{m+1}, g_{\mathbb{C}^n}\times g_{\mathbb{CP}^{m+1}} \right).$$
\end{theorem}
\begin{proof}
We have obtained uniform bounds for  $\tilde \phi'(\rho+\rho_s)$, $\tilde \phi''(\rho+\rho_s)$ and $\tilde \phi'''(\rho+\rho_s)$ for $\rho$ on any compact subset of $\mathbb{R}$ for all $s\geq 0$ as well as a uniform positive lower bound for $\tilde \phi'$ and $\tilde\phi''$. This implies a local parabolic $C^{2,\alpha}$-estimate for $\tilde\phi$ by suitable translation for $\tilde\phi$, which gives higher order estimates for $\tilde\phi$.  Then $\tilde\phi(\rho+\rho_s, s)- \tilde\phi(\rho_s, s)$ subconverges smoothly to a K\"ahler potential $\tilde\phi_\infty$ with the associated smooth K\"ahler metric
$\tilde g_\infty = \ddbar \tilde\phi_\infty$. Since $\tilde a(s) \rightarrow \infty$,  we have
\begin{equation}\label{wme1}
\tilde g_\infty =  g_{\mathbb{C}^n} + \tilde \phi_\infty' g_{\mathbb{CP}^m} + \tilde \phi_\infty'' g_{cyl}.
\end{equation}
on $\mathbb{C}^n \times \left( \mathbb{CP}^{m+1} \setminus \{ P_{\mathbb{CP}^{m+1}, 0} \cup D_{\mathbb{CP}^{m+1}, \infty} \} \right) $ with a splitting flat $\mathbb{C}^n$, where $P_{\mathbb{CP}^{m+1}, 0}$ annd $D_{\mathbb{CP}^{m+1}, \infty}$ are the zero and $\infty$ sections of $\mathbb{CP}^{m+1}$ respectively.

Our goal is now to show the curvature must be uniformly bounded for $\tilde g(\rho, s)$ for all $s$ and $\rho$ near $\infty$. We will prove by contradiction.  Suppose there exist $(\rho_{s_j}+\eta_j, s_j)$ such that
$$
K_j=|Rm(\tilde g)|( \rho_{s_j} + \eta_j, s_j)= \sup_{(\rho, s)\in \mathbb{R} \times [0, s_j]} |Rm(\tilde g( \rho_{s_j} + \rho , s_j))|-1\rightarrow \infty.
$$
Then by Lemma \ref{lRmbound} we must have $|\eta_j| \rightarrow \infty$. Let $x_j \in \Psi^{-1}(\rho_{s_j}+ \eta_j)$ and
$$\tilde g_j(\rho) = K_j \tilde g(\rho, s_j), ~\tilde\phi_j(\rho) = K_j \tilde\phi(\rho, s_j).$$
We further let $(\tilde X_\infty, \tilde g_\infty, x_\infty)$ be the smooth limit of $(X, \tilde g_j, x_j)$.

\medskip

\noindent {\bf Case I:} $\lim_{j\rightarrow \infty} \eta_j = -\infty$.

For this case, we should note that, for all $\rho\leq 0$, we have $\tilde \phi'(\rho_{s_j} + \rho) \leq \tilde \phi'(\rho_{s_j}) \leq \tilde b(s)/2$, hence by Proposition \ref{2ndder1}, we have $\tilde \phi''(\rho_{s_j} + \rho)/\tilde \phi'(\rho_{s_j} + \rho)\geq C^{-1}$. We consider the following three cases.

\begin{enumerate}

\item Suppose $ \tilde\phi'_j(\rho_{s_j}+\eta_j) \rightarrow \infty$ as $j\to\infty$. Then by $\eta_j<0$, we have $\tilde \phi''_j(\rho_{s_j}+\eta_j) \rightarrow \infty$ as  well. Since the scalar curvature of $\tilde g_j$ tends to $0$, we can apply the curvature formula (\ref{scurv}) to obtain that
$$
\frac{\tilde \phi^{(4)}_j(\rho_{s_j}+\eta_j)}{\left( \tilde \phi''_j(\rho_{s_j}+\eta_j) \right)^2}   \rightarrow 0.
$$
Hence by the curvature formulas (\ref{R1bar1}), (\ref{Rijkl}), the full curvature of $\tilde g_j$ at $x_j$ tends to $0$. Contradiction.

\medskip

\item Suppose $\tilde \phi'_j(\rho_{s_j}+\eta_j) \rightarrow 0$ as $j\to\infty$. The proof for this case is quite similar to Case II of Lemma \ref{RmboundP0}, we outline the main steps of the arguments.

First we can find $\delta>0$ such that there exists $\epsilon>0$, such that for any $q \in X$ with $d_{\tilde g_j}(x_j, q) <\delta$, we have $|Rm(\tilde g_j)|(q) \geq \epsilon$. We then let $y_{j} \in \Psi^{-1}(\rho_{y_{j}})$ be the point on the radial geodesic of $x_j$ with $d_{\tilde g_j}(x_j, y_{j})=\delta$ and $\rho_{y_{j}} > \rho_{x_j}=\rho_{s_j}+\eta_j$. By the Schwarz lemma we have $\rho_{y_j}\to-\infty$, hence $\tilde \phi''(\rho_{y_j} + \rho)/\tilde \phi'(\rho_{y_j} + \rho)\geq C^{-1}$ for all $\rho\leq 1$. Hence by the same argument of Corollary \ref{dist}, $\tilde \phi_j'(\rho_{y_j}) \to 0$ would imply that $\delta=d_{\tilde g_j}(x_j, y_{j})\to 0$, which is impossible, hence we have $\tilde \phi_j'(\rho_{y_j})>c>0$ for all large $j$.

Then we find $z_j$ on the radial geodesic of $x_j$ such that
$$\tilde \phi'_j(\rho_{z_j})=c,~ \rho_{z_j}=\Psi(z_j)\in (\rho_{x_j}, \rho_{y_j} ] .$$
By Proposition \ref{2ndder1}, there exist $C>0$ such that for all $\rho \in (-1, 1)$,
$$C^{-1} \leq \tilde\phi_j'(\rho_{z_j} + \rho) \leq C, ~C^{-1} \leq \tilde\phi_j''(\rho_{z_j} + \rho) \leq C, ~|\tilde \phi_j'''(\rho_{z_j}+\rho)| \leq C, $$
then by the scalar curvature formula (\ref{scurv}) and Corollary \ref{gscb}, we have $|\tilde \phi_j^{(4)}(\rho_{z_j}+\rho)| \leq C$.

Then after passing to a subsequence $f_j(\rho):=\tilde\phi_j'(\rho_{z_j} +\rho) $ converges in $C^{2,\alpha}$-topology to some function $f_\infty(\rho)$ for $\rho \in (-1,1)$, and so the warp product metrics
\begin{equation}\label{wme4}
\tilde g_j(\rho_{z_j}+\rho)=\left(K_j \tilde a(s_j) +  \tilde \phi'_j(\rho_{z_j}+\rho)\right) g_Z + \tilde\phi'_j(\rho_{z_j}+\rho)g_{\mathbb{CP}^m} + \tilde \phi''_j(\rho_{z_j}+\rho) g_{cyl},
\end{equation}
converge to (note that $\tilde a(s_j)\to\infty$)
$$\hat g_\infty (\rho) = g_{\mathbb{C}^{n} }+ f_\infty(\rho)g_{\mathbb{CP}^m}+f_\infty'(\rho) g_{cyl},~~ \rho \in (-1,1).$$
Now the same argument of Case II of Lemma \ref{RmboundP0} shows that $\hat g_\infty (\rho)$ is flat on $\rho \in (-1,1)$, which contradicts to the fact that $|Rm(\tilde g_j)|(z_j) \geq \epsilon$.

\medskip

\item Suppose $\tilde \phi_j'(\rho_{s_j} + \eta_j)$ is uniformly bounded above and below away from $0$ for all sufficiently large $j$. Then we can apply the argument of the latter part of (1) in Case I (with $z_j$ been replaced by $x_j$) and show that $\tilde g_j$ converges to a flat metric near $x_\infty$. Contradiction.

\end{enumerate}

\medskip

\noindent {\bf Case II:} $\lim_{j\rightarrow \infty} \eta_j = \infty$.

For this case, we should note that, for all $\rho\geq 0$, $ \tilde \phi_j'(\rho_{s_j}+\rho,s_j) \rightarrow \infty$ since $\tilde \phi'(\rho_{s_j}+\rho,s_j)\geq\tilde \phi'(\rho_{s_j},s_j)\geq \tilde b(s)/2$. We will discuss in the following three cases.

\begin{enumerate}

\item Suppose $\tilde\phi_j''(\rho_{s_j}+\eta_j, s_j) \rightarrow \infty$ as $j\to\infty$. By the scalar curvature formula (\ref{scurv}) and the fact that the scalar curvature of $\tilde g_j$ tends to $0$, we have
$$
\frac{\tilde \phi^{(4)}_j(\rho_{s_j}+\eta_j)}{\left( \tilde \phi''_j(\rho_{s_j}+\eta_j) \right)^2}   \rightarrow 0.
$$
Hence by the curvature formulas (\ref{R1bar1}), (\ref{Rijkl}), the full curvature of $\tilde g_j$ at $x_j$ must tend to $0$. Contradiction.

\medskip

\item Suppose $ \tilde\phi_j''(\rho_{s_j}+\eta_j) \rightarrow 0$ as $j\to\infty$.  We choose $\delta>0$ such that there exists $\epsilon>0$ such that for any point $q \in X$ with $d_{\tilde g_j}(x_j, q) <\delta$, we have
$$|Rm(\tilde g_j)|(q) \geq \epsilon.$$
We then let $y_{j} \in \Psi^{-1}(\rho_{y_{j}})$ be the point on the radial geodesic of $x_j$ with $d_{\tilde g_j}(x_j, y_{j})=\delta$ and $\rho_{y_{j}} > \rho_{s_j}+ \eta_j$.

We claim that there exists $c>0$ such that for each $j>0$, there exists  $\eta'_j \in (\eta_j, \rho_{y_{j}} -\rho_{s_j})$ with
$$\tilde \phi_j''(\rho_{s_j} +\eta'_j) \geq c.$$
Otherwise $ \tilde \phi_j''(\rho)$ will tend to $0$ uniformly on the annulus defined by
$$\{ \rho \in (\rho_{s_j}+\eta_j, \rho_{y_{j}})\} $$
with respect to the warp product metric
\begin{equation}\label{wme2}
\left(K_j \tilde a(s_j) +  \tilde \phi'_j(\rho_{s_j}+\eta_j+\rho)\right) g_Z + \tilde\phi'_j(\rho_{s_j}+\eta_j+\rho)g_{\mathbb{CP}^m} + \tilde \phi''_j(\rho_{s_j}+\eta_j+\rho) g_{cyl},
\end{equation}
since here $\tilde\phi_j'''$ is uniformly bounded. Then this annulus will collapse to the limit $\mathbb{C}^{m+n}\times (0, \delta)$. This leads to contradiction due Perelman's $\kappa$-noncollapsing theorem.

We can now assume that $\tilde\phi_j''(\rho_{s_j}+\eta_j')=c>0$ for all $j>0$ with $\eta_j' \in (\eta_j, \rho_{y_{j}}- \rho_{s_j})$ by the above claim and the continuity of $\tilde\phi_j''$. By Proposition \ref{2ndder1}, there exist $C>0$ such that for all $\eta \in (\eta_j'-1, \eta_j'+1)$,
$$C^{-1} \leq \tilde\phi_j''(\rho_{s_j} + \eta) \leq C, ~|\tilde \phi_j'''(\rho_{s_j}+\eta)| \leq C. $$
Then after passing to a subsequence $f_j(\rho)= \tilde\phi_j''(\rho_{s_j} + \eta'_j+\rho) $ converges in $C^\alpha$-topology to some Lipschitz function $f_\infty(\rho)$ for $\rho \in (-1,1)$ and so the warp product metrics (\ref{wme2}) converge to
$$\tilde g_\infty = g_{\mathbb{C}^{m+n} }+ f_\infty(\rho) g_{cyl}$$
for $\rho\in (-1, 1)$.  Since $\tilde g_\infty$ is Ricci flat, $f_\infty(\rho) g_{cyl}$ must be flat and so $\tilde g_\infty$ is flat. By the choice of $\delta$, $\tilde g_\infty$ cannot be flat and it leads to contradiction.

\medskip

\item Suppose $\tilde \phi_j''(\rho_{s_j} + \eta_j)$ is uniformly bounded above and below away from $0$ for all sufficiently large $j$. Then we can apply the argument of the latter part of (2) in Case II and show that $\tilde g_j$ converges to a flat metric near $x_\infty$. Contradiction.

\end{enumerate}

Combining the above the two cases, we have proved the uniform curvature bound for $\tilde g$.  From the curvature estimates and (\ref{wme1}), $(X, \tilde g(s))$ converges smoothly to a soliton metric on $\mathbb{C}^{n}\times \mathbb{CP}^{m+1}$ with a splitting metric $\tilde g_\infty = g_{\mathbb{C}^n}+ h$, where $h$ is a smooth K\"ahler Ricci soliton metric on $\mathbb{CP}^{m+1}$. Since the only K\"ahler-Ricci soltion metric on $\mathbb{CP}^{m+1}$ is the Fubini-Study metric, we can conclude that the limiting space $(X_\infty, \tilde g_\infty)$ is $(\mathbb{C}^{n}\times \mathbb{CP}^m, g_{\mathbb{C}^n}\times g_{\mathbb{CP}^{m+1}})$.

We have now completed the proof of the theorem.
\end{proof}

\section{Further discussions on Ricci vertices}\label{FDonRV}
In this section, we will propose some conjectures and questions concerning the Ricci vertices based on the results obtained in this paper. They can be seen as the intermediate steps toward Conjecture \ref{mainconj}.

In Theorem \ref{cor1}, we have obtained the local Type I scalar curvature estimate in Conjecture \ref{mainconj} around the Ricci vertex. A natural question is how to control the location of  Ricci vertices.

\begin{conjecture}\label{maincon3}
Let $g(t)$ be the maximal solution of the K\"ahler-Ricci flow (\ref{unkrflow1'1}) on $X\times [0, 1)$. For any $q\in Y$, there exist smooth closed $(1,1)$-form $\theta_Y\in \vartheta$ on $Y$ and $C=C(n, g_0, \theta_Y)>0$, such that for all $t\in [0, 1)$, if $p_t$ is a Ricci vertex associated to $\theta_Y$ at $t$, then  we have
$$d_{g(t)}(p_t, \mathcal{F}_q) \leq C(1-t)^{\frac{1}{2}}, $$
where $\mathcal{F}_q= \Phi^{-1}(q)$.
\end{conjecture}
Conjecture \ref{maincon3} is confirmed in Theorem \ref{calmain1} for the K\"ahler-Ricci flow with large symmetry, which further leads to the Type I curvature bounds. We would also like to ask if the Ricci vertices $p_t$ in Conjecture \ref{maincon3} has Type I distance to the $H_{2n}$-center of any point in $\mathcal{F}_q$.

Theorem \ref{cor2} establishes the Type I scalar curvature and diameter bound for the fibre containing the Ricci vertex, provided a Type I volume bound of a suitable tubular neighborhood of  the same fibre (see (\ref{cor2a})). Naturally, one would like to remove such a volume bound assumption as in the following conjecture.
\begin{conjecture}\label{maincon2}
Let $g(t)$ be the maximal solution of the K\"ahler-Ricci flow (\ref{unkrflow1'1}) on $X\times [0, 1)$. For any smooth closed $(1,1)$-form $\theta_Y\in \vartheta$ on $Y$, for all $t\in [0, 1)$, if $p_t$ is a Ricci vertex associated to $\theta_Y$ at $t$, then for any $q\in B(p_t, t , D(1-t)^{1/2})$, we have
$$\textnormal{Diam}(\mathcal{F}_{q}, g(t)) \leq C(1-t)^{\frac{1}{2}}, $$
$$\sup_{\mathcal{F}_{q} }|\rr(\cdot, t)| \leq \frac{C}{1-t} , $$
where $\mathcal{F}_{q} = \Phi^{-1}(\Phi(q))$, $\rr(t)$ is the scalar curvature of $g(t)$ and $\textnormal{Diam}(\mathcal{F}_{q}, g(t))$ is the diameter of $\mathcal{F}_{q}$ in $(X, g(t))$, and $C=C(n, g_0, \theta_Y, D)>0$.
\end{conjecture}
 Conjecture \ref{maincon2} is confirmed by Theorem \ref{app3} in the case of collapsing solutions on Fano bundles.


\bigskip

\noindent{\bf Acknowledgements.} The authors would like to thank Richard Bamler, Bin Guo, Maximilien Hallgren, Yalong Shi and Zhenlei Zhang for many inspiring discussions. The first named author thanks Xiaochun Rong, Zhenlei Zhang and Kewei Zhang for hospitality and providing an excellent environment during his visits to Capital Normal University and Beijing Normal University where part of this work was carried out.

\bigskip



\begin{thebibliography}{99}


\bibitem{Bam18} Bamler, R. {\em Convergence of Ricci flows with bounded scalar curvature}, Ann. of Math. (2) 188 (2018), no. 3, 753--83

\bibitem{Bam20a} Bamler, R. {\em Entropy and heat kernel bounds on a Ricci flow background},arXiv:2008.07093

\bibitem{Bam20b} Bamler, R. {\em Compactness theory of the space of Super Ricci flows}, Invent. Math. 233 (2023), no. 3, 1121--1277

\bibitem{Bam20c} Bamler, R. {\em Structure theory of non-collapsed limits of Ricci flows}, arXiv:2009.03243

\bibitem{BZ17} Bamler, R., and Zhang, Q. {\em Heat kernel and curvature bounds in Ricci flows with bounded scalar curvature}, Adv.Math. {\bf 319} (2017), 396--450

\bibitem{Cao} Cao, H.D. {\em Deformation of K\"ahler metrics to K\"aher-Einstein metrics on compact K\"ahler manifolds}, Invent. Math. 81 (1985), 359--372

\bibitem{Cal} Calabi, E.{\em Extremal K\"ahler metrics}, in Seminar on Differential Geometry, pp. 259-290, Ann. of Math. Stud., 102, Princeton Univ. Press, Princeton, N.J., 1982

\bibitem{CC1}  Cheeger, J. and Colding, T.H. {\em On the structure of
space with Ricci curvature bounded below I}, J. Differential. Geom.
 {\bf 46} (1997), 406--480

\bibitem{CC2}  Cheeger, J. and Colding, T.H.  {\em On the structure of space with Ricci curvature bounded below II}, J. Differential. Geom. {\bf 52} (1999), 13--35

\bibitem{CL} Chu, J. and Lee, M. {\em On the H\"older estimate of K\"ahler-Ricci flow}, Int. Math. Res. Not. IMRN 2023, no. 6, 4932--4951

\bibitem{CSW} Chen, X., Sun, S. and Wang, B. {\em K\"ahler-Ricci flow, K\"ahler-Einstein metric, and K-stability}, Geom. Topol. 22 (2018), no. 6, 3145--3173


\bibitem{CT} Collins, T. and Tosatti, V. {\em KÃ¤hler currents and null loci}, Invent. Math. 202 (2015), no. 3, 1167--1198

\bibitem{CW3} Chen, X.X. and Wang, B. {\em Space of Ricci flows (II)—part B: weak compactness of the flows}, J. Differential Geom, 116 (2020), no. 1, 1--123

\bibitem{DeS} Dervan, R. and Szekelyhidi, G. {\em The K\"ahler-Ricci flow and optimal degenerations}, J. Differential Geom. 116 (2020), no. 1, 187--203


\bibitem{DS1} Donaldson, S. and Sun, S. {\em Gromov-Hausdorff limits of K\"ahler manifolds and algebraic geometry}, Acta Math. 213 (2014), no. 1, 63--106


\bibitem{DS2} Donaldson, S. and Sun, S. {\em Gromov-Hausdorff limits of K\"ahler manifolds and algebraic geometry, II}, J. Differential Geom. 107 (2017), no. 2, 327--371

\bibitem{EMT} Enders, J., Mueller R. and Topping, P. {\em On Type I Singularities in Ricci flow}, Communications in Analysis and Geometry, 19 (2011) 905--922

\bibitem{FIK} Feldman, M., Ilmanen, T. and Knopf, D. {\em Rotationally symmetric shrinking and expanding gradient K\"ahler-Ricci solitons},  J. Differential Geometry 65  (2003),  no. 2, 169--209


\bibitem{Fong} Fong, F. {\em  K\"ahler-Ricci flow on projective bundles over K\"ahler-Einstein manifolds}, Trans. Amer. Math. Soc. 366 (2014), no. 2, 563--589

\bibitem{FL} Fong, F. and Lee, M. {\em Higher-order estimates of long-time solutions to the K\"ahler-Ricci flow}, J. Funct. Anal. 281 (2021), no. 11, Paper No. 109235, 34 pp

\bibitem{FoZh} Fong, F. and Zhang, Y. {\em Local curvature estimates of long-time solutions to the K\"ahler-Ricci flow}, Adv. Math. 375 (2020): 107416, 25 pp

\bibitem{FoZhZ} Fong, F. and Zhang, Z.{\em The collapsing rate of the K\"ahler-Ricci flow with regular infinite time singularity},J. Reine Angew. Math. 703 (2015), 95--113

\bibitem{FZ} Fu, X. and Zhang, S. {\em The K\"ahler Ricci flow on Fano bundle},  Math. Z. 286, 1605--1626 (2017)

\bibitem{GS} Guo, B. and Song, J. {\em On Feldman-Ilmanen-Knopf conjecture for the blow-up behavior of the K\"ahler-Ricci flow},  Math. Res. Lett. 23 (2016), no. 6, 1681--1719

\bibitem{GPS} Guo, B.,  Phong, D.H. and Sturm, J. {\em On the Kahler-Ricci flow on Fano manifolds},  Pure Appl. Math. Q. 18 (2022), no. 2, 573--581

\bibitem{GPSS} Guo, B.,  Phong, D.H., Song, J. and Sturm, J. {\em Diameter estimates in K\"ahler geometry}, arXiv:2209.09428


\bibitem{Ham} Hamilton, R.{\em Three-manifolds with positive Ricci curvature}, J. Differential Geom, 1982, 17: 255--306

\bibitem{HL} Han, J. and Li, C. {\em On the Yau-Tian-Donaldson conjecture for generalized Kahler-Ricci soliton equations}, Comm. Pure Appl. Math. 76 (2023), no. 9, 1793--1867


\bibitem{HN} Hein, H. and Naber, A. {\em New logarithmic Sobolev inequalities and an $\varepsilon$-regularity theorem for the Ricci flow}, Comm. Pure Appl. Math. 67 (2014), no. 9, 1543--1561

\bibitem{Hal} Hallgren, M.{\em Ricci flow with Ricci curvature and volume bounded below}, arXiv:2104.03386

\bibitem{HJ} Hallgren, M. and Jian, W. {\em Tangent Flows of K\"ahler Metric Flows}, arXiv:2202.06185, to appear in J. Reine Angew. Math

\bibitem{HJST} Hallgren, M., Jian, W., Song, J. and Tian, G. {\em Geometric regularity of the blow-up limits of the K\"ahler-Ricci flow}, preprint

\bibitem{J} Jian, W. {\em On the improved no-local-collapsing theorem of Ricci flow}, Peking Math. J. 6 (2023), no. 2, 459--468

\bibitem{JS} Jian, W. and Song, J. {\em Diameter estimates for long-time solutions of the K\"ahler-Ricci flow}, Geom. Funct. Anal. 32 (2022), no. 6, 1335--1356


\bibitem{JShi} Jian, W. and Shi, Y. {\em $L^\infty$ estimates for K\"ahler-Ricci flow on K\"ahler-Einstein Fano manifolds: a new derivation}, arXiv: 2305.09421, to appear in Proc. Amer. Math. Soc

\bibitem{JST23b} Jian, W., Song, J. and Tian, G. {\em A new proof of Perelman's scalar curvature and diameter estimates for the K\"ahler-Ricci flow on Fano manifolds}, preprint

\bibitem{Li} Li, C. {\em  On rotationally symmetric K\"ahler-Ricci solitons}, arXiv:1004.4049



\bibitem{LaTi} La Nave, G. and Tian, G. {\em A continuity method to construct canonical metrics}, Math. Ann. 365 (2016), no. 3-4, 911--921


\bibitem{LiTo} Li, Y. and Tosatti, V. {\em On the collapsing of Calabi-Yau manifolds and KÃ¤hler-Ricci flows}, J. Reine Angew. Math. 800 (2023), 155--192

\bibitem{LiYa79} Li, P. and Yau, S.T. {\em Estimates of eigenvalues of a compact Riemannian manifold}, Geometry of the Laplace operator (Proc. Sympos. Pure Math., Univ. Hawaii, Honolulu, Hawaii, 1979), pp. 205--239, Proc. Sympos. Pure Math., XXXVI, Amer. Math. Soc., Providence, R.I., 1980

\bibitem{LiYa86} Li, P. and Yau, S.T. {\em On the parabolic kernel of the SchrÃ¶dinger operator}, Acta Math. 156 (1986), no. 3-4, 153--201

\bibitem{Per1} Perelman, G. {\em The entropy formula for the Ricci flow and its geometric applications}, preprint, math.DG/0211159

\bibitem{Per2} Perelman, G. {\em Ricci flow with surgery on three-manifolds}, preprint, math.DG/0303109v1

\bibitem{Per3} Perelman, G. {\em Finite extinction time for the solutions to the Ricci flow on certain three manifolds}, preprint, math.DG/0307245

\bibitem{PS} Phong, D.H. and Sturm, J. {\em On stability and the convergence of the Kahler-Ricci flow}, J. Differential Geom. 72 (2006), no. 1, 149--168


\bibitem{PSSW} Phong, D. H.,  Song, J.,  Sturm, J. and Wang, X. {\em The Ricci flow on the sphere with marked points}, J. Differential Geom. 114 (2020), no. 1, 117--170

\bibitem{SeT} Sesum, N. and Tian, G. {\em Bounding scalar curvature and diameter along the K\"ahler Ricci flow (after Perelman)},  J. Inst. Math. Jussieu  7  (2008),  no. 3, 575--587

\bibitem{Shi} Shi, W.-X. {\em Deforming the metric on complete Riemannian manifolds}, J. Differential Geom. 30 (1989), no. 1, 223--301

\bibitem{So1}Song, J. {\em Some type I solutions of Ricci flow with rotational symmetry}, Int. Math. Res. Not. IMRN(2014): rnu134


\bibitem{So13} Song, J. {\em Ricci flow and birational surgery}, arXiv:1304.2607


\bibitem{So14} Song, J. {\em Riemannian geometry of K\"ahler-Einstein currents}, arXiv:1404.0445


\bibitem{SSW} Song, J.,  Szekelyhidi, G. and Weinkove, B. {\em The K\"ahler-Ricci flow on projective bundles}, Int. Math. Res. Not. IMRN(2013), no. 2, 243--257

\bibitem{SW0} Song, J. and Weinkove, B. {\em The K\"ahler-Ricci flow on Hirzebruch surfaces}, J. Reine Angew. Math. 659 (2011), 141--168

\bibitem{SW1} Song, J. and Weinkove, B. {\em Contracting exceptional divisors by the K\"ahler-Ricci flow}, Duke Math. J. 162 (2013), no. 2, 367--415

\bibitem{SW2} Song, J. and Weinkove, B. {\em Contracting exceptional divisors by the K\"ahler-Ricci flow II}, Proc. Lond. Math. Soc. (3) 108 (2014), no. 6, 1529--1561

\bibitem{SW3} Song, J. and Weinkove, B. {\em Lecture notes on the K\"ahler-Ricci flow}, 89--188, Lecture Notes in Math., 2086, Springer, Cham, 2013

\bibitem{SY} Song, J. and Yuan, Y. {\em Metric flips with Calabi ansatz}, Geom. Func. Anal. 22 (2012), no. 1, 240--265

\bibitem{ST1} Song, J. and Tian, G. {\em The K\"ahler-Ricci flow on surfaces of positive Kodaira dimension}, Invent. Math. {\bf 170} (2007), no. 3, 609--653

\bibitem{ST2} Song, J. and Tian, G. {\em Canonical measures and K\"ahler-Ricci flow}, J. Amer. Math. Soc., {\bf 25} (2012), no. 2, 303--353

\bibitem{ST3} Song, J. and Tian, G.{\em Bounding scalar curvature for global solutions of the K\"ahler-Ricci flow}, Amer. J. Math. {\bf 138} (2016), no. 3, 683--695

\bibitem{ST4} Song, J. and Tian, G. {\em The K\"ahler-Ricci flow through singularities}, Invent. Math. {\bf 207} (2017), 519-595

\bibitem{STZ} Song, J. and Tian, G. and Zhang, Z.L. {\em Collapsing behavior of Ricci-flat K\"ahler metrics and long time solutions of the K\"ahler-Ricci flow}, arXiv:1904.08345

\bibitem{T1} Tian, G. {\em On a set of polarized K\"ahler metrics on algebraic manifolds}, J. Differential Geom. 32 (1990), no. 1, 99--130


\bibitem{T90} Tian, G. {\em On Calabi's conjecture for complex surfaces with positive first Chern class}, Invent. Math. 101, no. 1 (1990), 101--172

\bibitem{T97} Tian, G. {\em K\"ahler-Einstein metrics with positive scalar curvature}, Invent. Math. 130 (1997), no. 1, 1--37

\bibitem{T13} Tian, G. {\em Partial  $C^0$-estimate for K\"ahler-Einstein metrics}, Commun. Math. Stat. 1 (2013), no. 2, 105-113

\bibitem{T15} Tian, G. {\em K-stability and K\"ahler-Einstein metrics}, Comm. Pure Appl. Math. 68 (2015), no. 7, 1085--1156


\bibitem{TiZXH1} Tian, G. and Zhu, X.H. {\em  Convergence of the K\"ahler-Ricci flow}, J. Amer. Math. Sci. 17 (2006), 675--699


\bibitem{TiZXH2} Tian, G. and Zhu, X.H. {\em  Convergence of the K\"ahler-Ricci flow on Fano manifolds}, J. Reine Angew. Math. 678 (2013), 223--245


\bibitem{TiZha} Tian, G. and Zhang, Z. {\em On the K\"ahler-Ricci flow on projective manifolds of general type},  Chinese Ann. Math. Ser. B  27  (2006),  no. 2, 179--192

\bibitem{TiZZL16} Tian, G. and Zhang, Z.L. {\em Regularity of K\"ahler-Ricci flows on Fano manifolds}, Acta Math. 216 no. 1 (2016), 127--176

\bibitem{TiZZL162} Tian, G. and Zhang, Z.L. {\em Convergence of KÃ¤hler-Ricci flow on lower-dimensional algebraic manifolds of general type}, Int. Math. Res. Not. IMRN 2016, no. 21, 6493--6511

\bibitem{TiZZL20} Tian, G. and Zhang, Z.L. {\em Relative volume comparison of Ricci flow}, Sci. China Math. 64 (2021), no. 9, 1937--1950

\bibitem{TiZZZ} Tian, G., Zhang, S., Zhang, Z.L. and Zhu, X. {\em Perelman's entropy and K\"ahler-Ricci flow on a Fano manifold}, Trans. Amer. Math. Soc. 365 (2013), no. 12, 6669--6695

\bibitem{Ts} Tsuji, H. {\em Existence and degeneration of K\"ahler-Einstein metrics on minimal algebraic varieties of general type}, Math. Ann. {\bf 281} (1988), no. 1, 123--133

\bibitem{Tos} Tosatti, V. {\em KAWA lecture notes on the KÃ¤hler-Ricci flow}, Ann. Fac. Sci. Toulouse Math. (6) 27 (2018), no. 2, 285--376


\bibitem{TWY} Tosatti, V., Weinkove, B. and Yang, X. {\em The KÃ¤hler-Ricci flow, Ricci-flat metrics and collapsing limits}, Amer. J. Math. 140 (2018), no. 3, 653--698

\bibitem{TZy} Tosatti, V. and Zhang, Y. {\em Finite time collapsing of the KÃ¤hler-Ricci flow on threefolds}, Ann. Sc. Norm. Super. Pisa Cl. Sci. (5) 18 (2018), no. 1, 105--118

\bibitem{W18} Wang, B. {\em The local entropy along Ricci flow, Part A: the no-local-collapsing theorems}, Camb. J. Math. 6 (2018), no. 3, 267-346


\bibitem{Ye} Ye, R., {\em The logarithmic Sobolev and Sobolev inequalities along the Ricci flow}, Commun. Math. Stat. 3 (2015), no. 1, 1–36

\bibitem{Y1}  Yau, S.T. {\em On the Ricci curvature of a compact  K\"ahler manifold and complex Monge-Amp\`{e}re  equation I}, \rm  Comm. Pure Appl. Math.  31  (1978),  339--411

\bibitem{ZhQ06} Zhang, Qi S. {\em Some gradient estimates for the heat equation on domains and for an equation by Perelman}, Int. Math. Res. Not., 39 pages Art. ID 92314, 39, 2006

\bibitem{ZhQ11} Zhang, Qi S. {\em Bounds on volume growth of geodesic balls under Ricci flow}, Math. Res. Lett. 19 (2012), no. 1, 245-253

\bibitem{Zhu} Zhu, X. {\em K\"ahler-Ricci flow on a toric manifold with positive first Chern class}, arXiv:math/0703486

\bibitem{ZhangZ} Zhang, Z. {\em Scalar curvature behavior for finite-time singularity of KÃ¤hler-Ricci flow}, Michigan Math. J. 59 (2010), no. 2, 419--433

\end{thebibliography}
\end{document}